\documentclass[12pt]{article}
\usepackage{amsthm}
\usepackage{dsfont} 
\usepackage{amssymb,amsmath,graphicx,
	amsfonts,euscript}
\usepackage{color}
\usepackage{cite}
\allowdisplaybreaks
\usepackage{hyperref}
\numberwithin{equation}{section}

\DeclareMathOperator{\Op}{Op}

\newcommand{\be}{\begin{eqnarray}}
\newcommand{\ee}{\end{eqnarray}}
\newcommand{\bes}{\begin{eqnarray*}}
\newcommand{\ees}{\end{eqnarray*}}

\renewcommand{\theequation}{\arabic{section}.\arabic{equation}}

\newtheorem{corollary}{Corollary}[section]
\newtheorem{definition}[corollary]{Definition}

\newtheorem{lemma}[corollary]{Lemma}
\newtheorem{proposition}[corollary]{Proposition}
\newtheorem{remark}[corollary]{Remark}
\newtheorem{theorem}[corollary]{Theorem}
\newtheorem{Notations}[corollary]{Notations}
\newcommand{\beq}{\begin{equation}}
\newcommand{\eeq}{\end{equation}}
\newcommand{\ben}{\begin{eqnarray}}
\newcommand{\een}{\end{eqnarray}}
\newcommand{\beno}{\begin{eqnarray*}}
\newcommand{\eeno}{\end{eqnarray*}}
\usepackage{mathrsfs}
\usepackage{enumerate}
\usepackage{geometry}
\usepackage{lastpage}
\usepackage{fancyhdr}
\title{\large\bf{On the Wave of Klein-Gordon Type: Propagation Equation at a Boundary and Interior Observability}} 

\author{
Siwar Ben Said\\
Laboratoire Analyse, G{\'e}om{\'e}trie et Applications\\
  Université Sorbonne Paris Nord\\
   99 Avenue Jean Baptiste Cl{\'e}ment \\93430 Villetaneuse, France\\
  siwar.bensaid@math.univ-paris13.fr
  }

\begin{document}
\maketitle
\begin{abstract}
Understanding the transport of semiclassical measures along generalized bicharacteristics is a key ingredient in the proof of observability for the wave equation. Here, the coefficients are only assumed to be of class $C^1$, which is a limit case for the existence
of bicharacteristics. We consider boundary conditions of the form $\partial_\nu u + i \partial_t u = 0$, as a first step towards more general Lopatinskii-type conditions. We derive the transport equation satisfied by the measure arising from the concentration of
high-frequency waves that obstruct observability. Observability of solutions associated with positive time-frequencies is then obtained by contradiction under the geometrical control condition.
\end{abstract}
\noindent\textbf{MSC-2020:} 35L05, 35L20, 35Q49, 35R05, 93B07, 34A99, 35S05.

\tableofcontents

\section{Introduction} \label{intro2.1tH}
The observability of the wave equation, which concerns  the ability to estimate the energy of a free solution using localized measurements in space and time, has been a major focus of research for many years, due to its strong link to the problem of exact controllability. This issue arises in both parabolic \cite{16A1,17A1,18A1} and hyperbolic \cite{5A1,7A1} systems. The observability problem has been studied on bounded domains in the case of smooth coefficients with methods that rely on microlocal analysis, a theory that originated in the 1950s with the introduction of pseudo-differential  operators by Kohn and Nirenberg \cite{14A1}, further generalized by Hörmander \cite{15A1}. 
Building on the work of Rauch and Taylor \cite{6A1}, Bardos, Lebeau and Rauch demonstrated boundary observability of waves in \cite{7A1}. Consequently, they established boundary stabilization under a condition known as the geometric control condition (GCC) for bounded, open, connected sets $\Omega$ of $\mathbb{R}^d$ with $C^\infty$ boundary. GCC concerns generalized geodesics. Away from boundaries they coincide with geodesics. At boundaries they obey the laws of geometrical optics. Existence and uniqueness of generalized geodesics is guaranteed by coefficients and boundary smoothness. GCC roughly states the following:
\begin{center}
\textit{(GCC): All generalized geodesics enter the observation region in some time $T > 0$.}
\end{center}
Then, in \cite{8A1}, Burq and Gérard established that the GCC is a necessary and sufficient condition for the exact controllability of the wave equation with Dirichlet boundary conditions. However, microlocal techniques demand high regularity of coefficients and boundaries.
For domains with rougher boundaries, older methodologies appear more applicable. Up until the late 1990s, most results were established under a global geometrical assumption known as the $\Gamma$-condition introduced by Lions \cite{5A1}. That condition allows one to use a multiplier method. Restriction  however applies for non constant coefficients. The $\Gamma$-condition however is not sharp: observability holds for geometries that do not satisfy the $\Gamma$-condition; see \cite{7A1}.\\
The works \cite{7A1,8A1} rely on microlocal tools, including the propagation of wavefront sets and microlocal defect measures introduced by Gérard \cite{9A1} and Tartar \cite{12A1}. 
Initially developed within the framework of the Melrose-Sjöstrand 
$C^\infty$ singularity propagation, which necessitated $C^\infty$ smoothness, the theory has progressed using the framework of semi-classical measures \cite{10A1} making it possible to consider $C^2$ coefficients in \cite{30A1}. Above this level of smoothness is just the natural $W^{2,\infty}$ regularity needed to define a geodesic flow. Below this smoothness threshold, uniqueness fails for geodesics. Below $C^1$ regularity, geodesics may further fail to exist. For $C^1$ coefficient regularity and $C^2$ boundary generalized geodesics do exist as proven in \cite{11A1}. Away from the boundary they are projections of bicharacteristics that are integral curve of the $C^0$ Hamiltonian vector field.
A natural question is then to understand the relationship between these nonunique integral curves and the observability property. This problem is addressed in the recent work of Burq, Dehman, and Le Rousseau \cite{1A1}. More precisely, for a $C^1$ metric and $C^2$ boundary, they prove that GCC as above remains sufficient for observability in the case of a Dirichlet boundary condition.\\

Note that in this case GCC is asked to hold for all generalized geodesics going through any point. If uniqueness holds, this condition is no different from that in the smooth case. The result of \cite{1A1} is obtained in the case of homogeneous Dirichlet boundary conditions.
The proof in \cite{1A1}, relies on a contradiction argument. A semiclassical measure emerges from potential concentration phenomena for sequences of waves that obstruct observability. This measure satisfies a transport equation along the Hamiltonian vector field associated with the wave operator. Here, we should mention that this transport equation originates from \cite{10A1}, where the measure was established in the context of the flat metric on a bounded open set in $\mathbb{R}^d$
with a $W^{2,\infty}$-boundary. However, no propagation result for the support of the measure was derived from this equation in \cite{10A1}.
In \cite{11A1}, where the setting is
different than in \cite{10A1}, a propagation result for the support of the measure along generalized bicharacteristics was obtained. Specifically, Burq, Dehman, and Le Rousseau showed that the support of the semi-classical measure is a union of maximal generalized bicharacteristics.
A natural question lies in the extension of this result for conditions different from Dirichlet conditions. Burq-Lebeau showed that a natural framework for the study of measure propagation is that of Lopatinskii conditions. Here, as a first attempt in this direction, we consider a Klein-Gordon wave equation with boundary condition given by 
\begin{align*}
   \partial_\nu u+i\partial_tu=0 
\end{align*}
which verifies the Lopatinskii condition for $\tau > 0$ but not for $\tau<0$, where $\tau$ denotes the dual variable associated with $t$.
With such condition, the wave energy is preseved similarly to Dirichlet and Neumann conditions. As for Neumann, constant functions are trivial solutions of the wave equation; hence, the introduction of the  Klein-Gordon model for which this does not occur (see below).
In this paper, following the strategy of \cite{1A1}, under the GCC, we establish interior observability of solutions associated with positive time-frequencies for the Klein-Gordon wave equation with a $C^1$ metric and a $C^2$ boundary. We first reduce the observation estimate to a high-frequency observation for semi-classically localized initial data, following  \cite{299A1}. We then argue by contradiction, constructing sequences of $H^1\times L^2$-normalized initial data that are spectrally localized in high frequencies and that vanish asymptotically in the observation region. These sequences are associated with a semi-classical measure $\mu$ that satisfies a propagation equation at the boundary, as stated in Theorem \ref{th1A1}. This theorem is central to the proof of our main result. Finally, we use this equation, along with the result from \cite{11A1}, which states that the support of $\mu$ is a union of maximal generalized bicharacteristics. By GCC, this leads to a contradiction, as $\mu$ vanish in the observation region, and yields the observability result.
\subsection{Outline}
The content is organized as follows. In Section \ref{settingA1} we set up the geometric framework used throughout the paper. In Section \ref{mainresultA1}, we state the main result (Theorem \ref{mainA1}).  In Section \ref{GeoA1}, we introduce the geometrical notions necessary to precisely formulate the geometric control condition (GCC) within our low-regularity framework, along with the main result. Section \ref{dyaA1} presents a semiclassical reduction of the observability estimates, with Section \ref{strategyA1}, outlining a contradiction argument that leads to a proof of a semiclassical observability inequality for positive time-frequencies. In Section \ref{semiopA1}, we review the theory of semiclassical pseudo-differential operators under minimal regularity assumptions on the symbols and recall the notions of semiclassical measures. Section \ref{uuuA1} is dedicated to proving the measure-propagation equation stated in Theorem \ref{th1A1}. Finally, in Section \ref{conclusionA1}, we use the result from \cite{11A1}, stated in Theorem \ref{generabichaA1}, and the propagation equation of Theorem \ref{th1A1} to conclude the proof of the main result.
\subsection{Setting and well-posedness} \label{settingA1}
Throughout this paper, we consider $\mathcal{M}$ a compact connected Riemanian
manifold of dimension $d\geq 2$ with a $C^2$ boundary, endowed with a $ C^1$ metric $g=(g_{ij})$. Note that $\mathcal{M}$ being compact it contains its boundary $\partial\mathcal{M}$.
\begin{Notations}
Denote by $\mu_g$ the canonical positive Riemannian density on $\mathcal{M}$.\\
The $L^2$-inner product and norm are considered with respect to this density $\mu_g$, that is,
\begin{align}
(u,v)_{L^2(\mathcal{M})}=\int_{\mathcal{M}}u\bar{v} d\mu_g, \quad\quad\quad  \|u\|^2_{L^2(\mathcal{M})}=\int_{\mathcal{M}}|u|^2 d\mu_g.
\end{align}
In local coordinates, the volume element $d\mu_g$ is given by $(det\;g)^{\frac{1}{2}}dx$. We denote the space of $L^2$-vector fields on $\mathcal{M}$ by $L^2V(\mathcal{M})$, equipped with the norm
\begin{align}
\|v\|^2_{L^2V(\mathcal{M})}=\int_{\mathcal{M}}g(v,\bar{v})d\mu_g, \quad\quad\quad v\in L^2V(\mathcal{M}).
\end{align}
For a function $f$ and a vector field $u$, the Riemannian gradient and the divergence are given respectively in local coordinates by
\begin{align*}
(\nabla_gf)^i=\sum_{1\leq j\leq d}g^{i,j}\partial_{x_j}f,\quad \quad \mathrm{div}(u)=(\mathrm{det}(g))^{-\frac{1}{2}}\sum_{1\leq i\leq d}\partial_{x_i}\big((\text{det}\; g)^{\frac{1}{2}}u^i),
\end{align*}
where $(g^{i,j})=(g_{i,j})^{-1}$.

The space $C_c(\mathbb{R}^{2d})$ denotes the set of compactly supported
continuous functions on $\mathbb{R}^{2d}$. 
$C^\infty_c(\mathbb{R}^{2d})$ refers to the space of smooth functions with compact support on $\mathbb{R}^{2d}$. The space $C_0(\mathbb{R}^{2d})$ consists of continuous functions on $\mathbb{R}^{2d}$ that converge to $0$ at infinity. Finally,
$\mathcal{D}'(\mathbb{R}\times \mathcal{M})$ denote the space of distributions on $\mathbb{R}\times \mathcal{M}$.
\end{Notations}
In this paper, we consider the following wave equation
of Klein-Gordon type
\begin{equation}
\label{eqimA1}
    \begin{cases}
       ( \partial_t^2-\Delta_g+m)u = 0 \quad  \quad\quad&\text{in} \;\mathbb{R}\times \mathcal{M}, \\
        \;\partial_nu+i\partial_tu\;\;\;\quad= 0 \;\;\;  \quad \quad&\text{in} \;\mathbb{R}\times \mathcal{\partial M},\\
        u_{\mid t=0}=\underline{u}^0,\;\; \partial_tu_{\mid t=0}=\underline{u}^1 \quad &\text{in} \;\mathcal{ M},
    \end{cases}
\end{equation}
where $m > 0$ is a constant, $n$ stands for the unitary normal inward pointing vector field to $\partial\mathcal{M}$ in the sense of the metric $g$ and
$\Delta_g$ corresponds to the unbounded Laplace-Beltrami opetator on $L^2(\mathcal{M})$, defined in local coordinates by
\begin{align}
\Delta_g f=(\text{det}\; g)^{-\frac{1}{2}}\sum_{1\leq i,j\leq d}\partial_{x_i}\big((\text{det}\; g)^{\frac{1}{2}}g^{i,j}(x)\partial_{x_j}f\big) .
\end{align}
For simplicity we consider the case $m = 1$ in \eqref{eqimA1}.

Let $\mathcal{H}$ be the space $H^1(\mathcal{M})\times L^2(\mathcal{M})$ and $A$ the unbounded operator on $\mathcal{H}$ defined by
\begin{align}
A:=i\begin{pmatrix}
0 & -1  \\
-\Delta_g +1 & 0
\end{pmatrix},\label{opA1}
\end{align}
with domain $$D(A)=\{(u^0,u^1)\in \mathcal{H};\;u^0\in H^2(\mathcal{M}),\; u^1 \in H^1(\mathcal{M}) \; \text{and} \; (\partial_nu^0+iu^1)_{|\partial\mathcal{M}}= 0 \}.$$
Setting $U=\,^t(u,\partial_tu)$, the equation \eqref{eqimA1} can be written
\begin{align}\label{eq2A1}
\partial_tU=i AU, \quad\quad U_{|t=0}=\underline{U}=\,^t(\underline{u}^0,\underline{u}^1).
\end{align}
A norm on $\mathcal{H}$ is defined by
\begin{align*}
\|(u^0,u^1)\|^2_\mathcal{H}&=\|u^0\|^2_{H^1(\mathcal{M})}+\|u^1\|^2_{L^2(\mathcal{M})}
\\&=\|u^0\|^2_{L^2(\mathcal{M})}+\|\nabla_gu^0\|^2_{L^2V(\mathcal{M})}+\|u^1\|^2_{L^2(\mathcal{M})}.
\end{align*}
This norm is associated with the inner product
\begin{align*}
(U,\tilde{U})_\mathcal{H}=(u^0,\tilde{u}^0)_{L^2(\mathcal{M})}+(\nabla_gu^0,\nabla_g\tilde{u}^0)_{L^2V(\mathcal{M})}+(u^1,\tilde{u}^1)_{L^2(\mathcal{M})},
\end{align*}
for $U=(u^0,u^1)$ and $\tilde{U}=(\tilde{u}^0,\tilde{u}^1)$.\\
A norm on $D(A)$ is given by
\begin{align*}
\|U\|_{D(A)}^2=\|AU\|_{\mathcal{H}}^2=\|u^1\|^2_{H^1(\mathcal{M})}+\|-\Delta_gu^0+u^0\|^2_{L^2(\mathcal{M})}, \quad U=(u^0,u^1).
\end{align*}
Here, the equation \eqref{eqimA1} is known as the Klein–Gordon equation, but for simplicity, we continue to call it wave equation.  Note that when $m=0$, which corresponds to the classical wave equation, constant functions are eigenfunctions of the operator $\Delta_g$ with $0$ as the eigenvalue. Consequently, constant functions are solutions to the wave equation and are considered "invisible solutions" in the context of the observability property we are concerned with (see Remark \ref{invisA1}). This issue does not arise when dealing with a manifold with a boundary and homogeneous Dirichlet conditions \cite{1A1}.  In our case, with the boundary condition as in \eqref{eqimA1}, to overcome this difficulty, one could either work in a quotient space or replace the wave operator with the Klein-Gordon operator. This explain our choice of equation \eqref{eqimA1} instead of the classical wave equation.

Before stating our main result, we establish in Proposition \ref{proµµµA1} the well-posedness of equation \eqref{eqimA1} in the energy space $\mathcal{H}$. This is based on the properties of the operator $(A, D(A))$ given in Lemmas \ref{selfA1} and  \ref{semigroupeA1}. The proofs of these lemmas are provided in Appendix \ref{appenAA1}.
\begin{lemma} \label{selfA1}
The operator $(A,D(A))$ is selfadjoint on $\mathcal{H}$ and has a compact resolvent.
\end{lemma}

\begin{lemma} \label{semigroupeA1}
The operator $(A,D(A))$ is the infinitesimal generator of a strongly continuous unitary group $S(t)$ on $\mathcal{H}$.
\end{lemma}
In view of Lemmas \ref{selfA1} and  \ref{semigroupeA1}, solutions of the system \eqref{eqimA1} are given by the following  proposition, with the proof same as in Theorem 6.1 of \cite{20A1}.
\begin{proposition}
\label{proµµµA1}
For $\underline{U}=(\underline{u}^0,\underline{u}^1)\in \mathcal{H}$, there exists a unique
\begin{align*}
u\in    C\big(\mathbb{R}; H^1(\mathcal{M}) \big)\cap C^1\big(\mathbb{R};L^2(\mathcal{M})\big),
\end{align*}
that is a weak solution of \eqref{eqimA1}. This solution is given by  the first component of $S(t)\underline{U}$, that is, $ u_{\mid t=0}=\underline{u}^0$ and $\partial_tu_{\mid t=0}=\underline{u}^1$ and it satisfies
\begin{align*}
( \partial_t^2-\Delta_g+1)u = 0 \quad \text{in} \quad \mathcal{D}'(\mathbb{R}\times \mathcal{M}).
\end{align*}
\end{proposition}
We define the energy of this solution at time $t$ by\begin{align}
\mathcal{E}(u)(t)=\frac{1}{2}\|(u(t),\partial_tu(t))\|^2_\mathcal{H}.
\end{align} This energy is constant with respect to time $t$, i.e., $\mathcal{E}(u)(t)  =\mathcal{E}(u)(0)$, and for simplicity, we will denote it as $\mathcal{E}(u)$.
\subsection{Statement of the main result} \label{mainresultA1}
 Let $\omega$ be a nonempty open subset of $\mathcal{M}$. Interior observability of the wave equation \eqref{eqimA1} from $\omega$ in time $T$ is defined as follows.
\begin{definition}[Interior observability]\label{defobsA1}
Let $\omega$ be a nonempty open subset of $\mathcal{M}$.The wave equation \eqref{eqimA1} is said to be observable from $\omega$ in time
$T > 0$ if there exists $C_{\mathrm{obs}}>0$ such that for any $\underline{U}=(\underline{u}^0,\underline{u}^1)\in \mathcal{H}$,
\begin{align} \label{obsA1}
\mathcal{E}(u)\leq C_{\mathrm{obs}}\int_0^T\|\mathds{1}_{\omega}\partial_tu(t)\|_{L^2(\mathcal{M})}^2dt,
\end{align}
for the weak solution $u$ to \eqref{eqimA1}.
\end{definition}
Note that the selfadjoint operator $A$, defined in \eqref{opA1}, with compact resolvent on $\mathcal{H}$, admits a sequence of eigenvalues $(\lambda_\nu)_{\nu \in \mathbb{Z}}$, compted with their finite multiplicities, such that
\begin{align*}
    \lambda_{-\nu}=-\lambda_\nu, \quad \lambda_\nu>0 \quad\text{if}\quad \nu>0 \quad \text{and}\quad \lambda_\nu \underset{\nu \rightarrow +\infty}{\longrightarrow}+\infty,
\end{align*}
 with associated sequence of eigenfunctions $(e_\nu)_\nu=(e_\nu^0,e_\nu^1)_\nu$, forming a Hilbert basis of $\mathcal{H}$. Any $U \in \mathcal{H}$ reads $U=\sum\limits_{\nu \in \mathbb{Z}} U_\nu e_\nu$ with $U_\nu=(U,e_\nu)_\mathcal{H}$ and $(U_\nu)_\nu \in \ell^2(\mathbb{C})$.
 
 Define
 \begin{align*}
     E^+=\{U \in \mathcal{H}; \quad U_\nu=0 \quad \text{if} \quad \nu<0\} \quad \text{and} \quad E^-=\{U \in \mathcal{H}; \quad U_\nu=0 \quad \text{if} \quad \nu>0\}.
 \end{align*}
 In this article, we are concerned with interior observability as follows.
 \begin{definition}[Interior observability for solutions associated with positive time-frequencies] \label{intersiwarnew}Let $\omega$ be a nonempty open subset of $\mathcal{M}$. The wave equation \eqref{eqimA1} is said to be observable for solutions associated with positive time-frequencies from $\omega$ in time $T>0$ if one has 
 \begin{align} \label{obsPartielA1}
\mathcal{E}(u)\leq C_{\mathrm{obs}}\int_0^T\|\mathds{1}_{\omega}\partial_tu(t)\|_{L^2(\mathcal{M})}^2dt,
\end{align}
for some $C_{\mathrm{obs}}>0$ for any solution to \eqref{eqimA1} with that $(\underline{u}^0,\underline{u}^1)\in E^+$.
\end{definition}
\begin{remark} \label{invisA1}
In the case of the classical wave equation (when $m = 0$), the energy function is given by
\begin{align*}
\mathcal{E}(u)(t)=\frac{1}{2}(\|\partial_tu(t)\|_{L^2(\mathcal{M})}^2+\|\nabla_gu(t)\|_{L^2V(\mathcal{M})}^2).
\end{align*}
Constant functions $u$ are said to be invisible for an observability inequality of the form \eqref{obsA1} or \eqref{obsPartielA1}, as they have zero energy and $\int_0^T\|\mathds{1}_{\omega}\partial_tu(t)\|_{L^2(\mathcal{M})}^2dt$ vanishes as well.
\end{remark}
The following proposition states that exact controllability is equivalent to an obserbability inequality. The duality between observability and controllability has been demonstrated by several authors (see, in particular, \cite{23A1} and \cite{5A1}).
\begin{proposition} []
Let $\omega$  be an open subset of $\mathcal{M}$ and $T > 0$. The wave equation is exactly controllable
from $\omega$ at time $T$ if and only if it is observable from $\omega$ at time $T$.
\end{proposition}
Before stating our main result, we first introduce the following notion.
\begin{definition}(Interior geometric control condition) \label{interiorA1}
Let $\omega$ be an open subset of $\mathcal{M}$.
One says that $\omega$ controls geometrically the manifold $\mathcal{M}$ if there exists $T > 0$ such
that any generalized bicharacteristic (see Definition \ref{generaA1} bellow) reaches a point above $]0, T [\times \omega$. One says
that $(\omega, T )$ fulfills GCC. In this case, one sets
\begin{align*}
T_{GCC}(\omega)=\mathrm{inf} \{ T>0; \;(\omega,T)\; \text{fulfills}\; GCC\}.
\end{align*}
\end{definition}
Our main result is the following.
\begin{theorem} \label{mainA1}
Let $\omega$ be an open subset of $\mathcal{M}$ that satisfies the
interior geometric control condition associated with the infimum time $T_{GCC} (\omega)$.
Let $T > T_{GCC} (\omega)$. Then, the wave equation 
is observable from $\omega$ at time $T>0$ in the sense of Definition \ref{intersiwarnew}.
\end{theorem}
Here, we recall that Burq, Dehman and Le Rousseau \cite{1A1} established interior observability in the sense of Definition \ref{defobsA1} for the wave equation under homogeneous Dirichlet boundary conditions. The main point of Theorem \ref{mainA1} lies in the consideration of a boundary condition different from the classical Dirichlet or Neumann ones, thereby opening the way to the study of more general boundary conditions of Lopatinskii type. It is important to note that, under the boundary condition $\partial_n u + i \partial_t u = 0,$ the Lopatinskii condition holds for $\tau > 0$, but fails for $\tau < 0$, where $\tau$ denotes the dual variable associated with $t$. This restriction prevented us from obtaining observability in the classical sense through the method used in this work. This does not mean that observability is false, but only that it is not accessible through our approach. For waves associated with negative frequencies, the problem also remains open and constitutes a natural question for further investigation. A boundary condition that could satisfy the Lopatinskii condition for all $\tau\in\mathbb{R}^*$ would be $\partial_n u - |D_t|u = 0.$ In this case, however, the difficulty is that $|D_t|$ is nonlocal in time, so the associated wave equation cannot be written in semigroup form. Consequently, the approach developed in this paper does not adapt to this setting. At the same time, this also gives rise to a new open problem.
\section{Geometry}
\label{GeoA1}
In this section, we present the geometric notions required to clearly understand the GCC stated in Definition \ref{interiorA1} within our low-regularity framework. 
\subsection{Choice of local coordinates}
Throughout this article, near a boundary point, local coordinates are chosen such as in the following proposition. Such a choice ensures the simplification of some geometrical notions introduced in this work.
\begin{proposition}[Quasi-normal geodesic coordinates]
\label{pro9A1}
For any $m_0 \in \partial\mathcal{M}$, there exists a $C^2$-local chart $(O,\phi)$ such that $m_0\in O$ and $\phi(m)=(x',z)$, where $x'\in \mathbb{R}^{d-1}$ and $z\in \mathbb{R}$, with
\begin{enumerate}
    \item $\phi(O\cap \mathcal{M})=\{z\geq0\}\cap\phi(O),\quad\phi(O\cap \partial\mathcal{M})=\{z=0\}\cap\phi(O)$, and\\ $\phi(O\setminus\mathcal{M})=\{z<0\}\cap \phi(O)$.
    \item Near the boundary, the metric $g$ satisfies
    \begin{align}
    \label{6***A1}
    g^{d,d}(x',z)=1+zh^{d,d}(x',z), \quad\;\text{and}\quad\; g^{i,d}(x',z)=zh^{i,d}(x',z),
    \end{align}
  for some continuous functions $h^{i,d}$, $i=1,..,d-1$.
\end{enumerate}
\end{proposition}
Proposition \ref{pro9A1} is presented in \cite{11A1} with a generalization to other levels of regularity (see Appendix B in \cite{11A1}).

We set $\mathcal{L}=\mathbb{R}\times \mathcal{M}$. For $x=\phi(m)$ with $m\in O\cap \mathcal{M}$, we denote the associated coordinates in \(T_m\mathcal{M}\) by \(\upsilon = (\upsilon', \upsilon_d)\) and in \(T^*_m\mathcal{M}\) by \(\xi = (\xi', \xi_d)\), where \(\upsilon', \xi' \in \mathbb{R}^{d-1}\) and \(\upsilon_d, \xi_d \in \mathbb{R}\). We denote the cotangent variable \(\xi_d\) by the letter \(\zeta\), so that \(\xi = (\xi', \zeta)\).

Note that, with local charts at the boundary as expressed in Proposition \ref{pro9A1}, if \\$\upsilon \in T_x \partial \mathcal{M}$ and $x\in \partial \mathcal{M}$, then \(\upsilon = (\upsilon', 0)\). We use the bijective map \((\xi', 0) \longmapsto \xi'\) to parameterize \(T^*_x \partial \mathcal{M}\).

We set
\begin{align*}
T\mathcal{M}=\underset{x\in \mathcal{M}}{\bigcup} \{x\}\times T_x\mathcal{M} \quad\quad \text{and}\quad\quad T^*\mathcal{M}=\underset{x\in \mathcal{M}}{\bigcup} \{x\}\times T^*_x\mathcal{M}.
\end{align*}

Recalling that $\mathcal{M}$ includes its boundary $\partial\mathcal{M}$, the tangent space $T\mathcal{M}$ (respectively, $T^*\mathcal{M}$) contains \(\{x\} \times T_x\mathcal{M}\) (respectively, \(\{x\} \times T^*_x\mathcal{M}\)) for each \(x \in \partial \mathcal{M}\). We denote the boundary of $T^*\mathcal{M}$ by
$\partial(T^*\mathcal{M})$, which consists of the set of pairs $(x, \xi)$ with $x \in \partial\mathcal{M}$.

In the corresponding local chart on \(\mathcal{L}\), the representative of $(t, m) \in \mathcal{L}$ is expressed as $(t, x) = (t, x', z)$. Here we denote an element of $T^*\mathcal{L}$ by $\rho$, that is, $\varrho = (t, x, \tau, \xi)$ with $(t, x) \in \mathcal{L}$, $\tau \in \mathbb{R}$, and $\xi \in T^*_x\mathcal{M}$. Naturally, we mean by $T^*\mathcal{L} \setminus 0$ for the set of points $\varrho = (t, x, \tau, \xi)$ with $(\tau, \xi) \neq 0$. The boundary $\partial(T^*\mathcal{L})$ consists on the set of points $\varrho = (t, x, \tau, \xi)$ such that $x \in \partial\mathcal{M}$. In local coordinates, $\partial(T^*\mathcal{L})$ is given by $\{z = 0\}$, and $T^*\mathcal{L}$ corresponds to $\{z \geq 0\}$.
\subsection{Hamiltonian vector field and bicharacteristics }
In this section, we give the expressions for the wave operator and the Hamiltonian vector field, each in local coordinates. We also recall the definition of bicharacteristics. \\
 In local coordinates, the wave operator $P=\partial^2_t-\Delta_g+1$ is given by 
\begin{align*}
Pf=\partial^2_tf-(\text{det}\; g)^{-\frac{1}{2}}\sum_{1\leq i,j\leq d}\partial_{x_i}\Big((\text{det}\; g)^{\frac{1}{2}}g^{i,j}(x)\partial_{x_j}f\Big)+f .
\end{align*}
Its principal symbol is
\begin{align*}
p(\varrho)=-\tau^2 +\sum_{1\leq i,j\leq d} g^{i,j}(x)\xi_i\xi_j, \quad\quad \varrho=(t,x,\tau,\xi).
\end{align*}
We denote by $H_p$ the Hamiltonian vector field associated with $p$, that is, the unique vector
field such that $\{p, f \} = H_p f$ for any smooth function $f$, where $\{., .\}$ denotes the Poisson
bracket. In local chart, one has
\begin{align*}
\{p, f \}=\partial_\tau p\partial_tf-\partial_tp\partial_\tau f+\sum_{1\leq j\leq d}(\partial_{\xi_j}p\partial_{x_j}f-\partial_{x_j}p\partial_{\xi_j}f),
\end{align*}
yielding
\begin{align} \label{hamiltA1}
H_p(\varrho)&=\partial_\tau p(\varrho)\partial_t+\nabla_\xi p(\varrho).\nabla_x-\nabla_xp(\varrho).\nabla_\xi \nonumber
\\&=-2\tau\partial_t+2g^{i,j}(x)\xi_i\partial_{x_j}-\partial_{x_k}g^{i,j}(x)\xi_i\xi_j\partial_{\xi_k}.
\end{align}
Observe that, for a function $f$ of the variable $\varrho=(t, x, \tau, \xi)$, one has
\begin{align*}
^tH_pf(\varrho)=2\tau\partial_tf(\varrho)-2\partial_{x_j}\big(g^{i,j}(x)\xi_if(\varrho)\big)+\partial_{\xi_k}\big(\partial_{x_k}g^{i,j}(x)\xi_i\xi_jf(\varrho)\big),
\end{align*}
with which one deduces
\begin{align*}
^t H_p=-H_p.
\end{align*}
In the following definition we recall the notion of bicharacteristics. We denote by $\mathrm{Char}(p)$ the characteristic set of $p$, given by
\begin{align*}
\mathrm{Char}(p):=\{\varrho = (t, x, \tau, \xi) \in T^*\mathcal{L},\, p(\varrho)=0\}.
\end{align*}
\begin{definition} [Bicharacteristics]
Let $V$  be an open subset of $T^*\mathcal{L} \setminus \partial(T^*\mathcal{L})$, and let $J \subset \mathbb{R}$ be an
interval. A $C^1$ map $\gamma : J \rightarrow V \cap \mathrm{Char}(p)$ is called a bicharacteristic in $V$ if it satisfies
\begin{align*}
\frac{d}{ds}\gamma(s)=H_p(\gamma(s)), \quad \text{for\; all}\quad s\in J.
\end{align*}
It is said to be maximal in $V$ if it cannot be extended by another bicharacteristic also
valued in $V$.
\end{definition}
\subsection{Partition of the cotangent bundle at the boundary}
We denote by $^\parallel \partial(T^*\mathcal{L})\subset \partial(T^*\mathcal{L})$  the bundle consisting of points $$\varrho=(\varrho',0)=(t,x',z=0,\tau,\xi',0)\in T^*\mathcal{L},$$ where  $\varrho'=(t,x',z=0,\tau,\xi')\in T^*\partial\mathcal{L}$. By identifying $\varrho'$ and $(\varrho',0)$ as expressed above using the chosen local coordinates, one can write the identification  $^\parallel \partial(T^*\mathcal{L})\simeq T^*\partial\mathcal{L}$.
\\
Let $\pi_\parallel$ be the projection map from $\partial(T^*\mathcal{L})$ to $^\parallel\partial(T^*\mathcal{L})$ defined by
\begin{align*}
\pi_\parallel(t,x',z=0,\tau,\xi',\zeta)=(t,x',z=0,\tau,\xi',0).
\end{align*}
We express the vector bundle $^\parallel\partial(T^*\mathcal{L})$ as the union of the three bundles $^\parallel\mathcal{E}_\partial$,
 $^\parallel\mathcal{G}_\partial$,     
 $^\parallel\mathcal{H}_\partial$ defined as follows.
\begin{definition}[Elliptic, glancing, and hyperbolic regions]
\label{defn11A1}
The vector bundle $^\parallel\partial(T^*\mathcal{L})$ consists on three homogeneous regions.
\begin{enumerate}
    \item The elliptic region $^\parallel\mathcal{E}_\partial= \,^\parallel\partial(T^*\mathcal{L})\cap \{p>0\}$. A point $\varrho\in \;^\parallel\mathcal{E}_\partial$ is called an elliptic point.

    \item The glancing region $^\parallel\mathcal{G}_\partial=\,^\parallel\partial(T^*\mathcal{L})\cap \{p=0\}$. A point $\varrho\in \;^\parallel\mathcal{G}_\partial$ is called a glancing point.
    \item The hyperbolic region $^\parallel\mathcal{H}_\partial=\,^\parallel\partial(T^*\mathcal{L})\cap \{p<0\}$. A point $\varrho\in \;^\parallel\mathcal{H}_\partial$ is  called an hyperbolic point.
\end{enumerate}
With \eqref{6***A1}, one has $p(\varrho)=-\tau^2+\zeta^2+g(\xi',\xi') $ if $\varrho\in \partial(T^*\mathcal{L})$, and so we deduce the following properties:
\begin{enumerate}
    \item If $\varrho \in\,^\parallel\mathcal{E}_\partial$ then $\pi^{-1}_\parallel(\{\varrho\})\cap \mathrm{Char}(p)=\emptyset$.
    \item If $\varrho \in\,^\parallel\mathcal{G}_\partial$ then $\pi^{-1}_\parallel(\{\varrho\})\cap \mathrm{Char}(p)=\{\varrho\}$.
    \item If $\varrho \in\,^\parallel\mathcal{H}_\partial$ then $\pi^{-1}_\parallel(\{\varrho\})\cap \mathrm{Char}(p)=\{\varrho^-,\varrho^+\}$, where 
    \begin{align}
    \label{6£££A1}
    \varrho^\pm=(t,x',z=0,\tau,\xi',\zeta^\pm) \quad with \quad \zeta^\pm =\pm\sqrt{-p(\varrho)}.
    \end{align}
\end{enumerate}
\end{definition}
Note that based on the partition of $^\parallel\partial(T^*\mathcal{L})$ given in Definition \ref{defn11A1}, if \\$\varrho \in \mathrm{Char}(p)\cap \partial(T^*\mathcal{L})$, one has $\pi _\parallel(\varrho)\in \, ^\parallel\partial(T^*\mathcal{L})$ and $p(\pi _\parallel(\varrho)) \leq 0$. Furthermore,\\ $\varrho \in \mathrm{Char}(p)\cap \partial(T^*\mathcal{L})$ and $p(\pi _\parallel(\varrho))=0$ is equivalent to writing $\varrho\in \; ^\parallel\mathcal{G}_\partial$. This leads to the following associated partition of $\mathrm{Char}(p)$ on the boundary $ \partial(T^*\mathcal{L})$.
\begin{definition}[Partition of $\mathrm{Char}(p)$ at the boundary]
One partitions $\mathrm{Char}(p)\cap \partial(T^*\mathcal{L})$ into two  homogeneous regions $\mathcal{G}_\partial$ and $\mathcal{H}_\partial$ with
\begin{enumerate}
    \item  $\mathcal{G}_\partial=\,^\parallel\mathcal{G}_\partial$; $\varrho \in \mathcal{G}_\partial \Leftrightarrow \varrho \in \mathrm{Char}(p) $ and    $\pi_\parallel(\varrho)=\varrho$ .
    \item  $\varrho \in \mathcal{H}_\partial$ if $\varrho \in \mathrm{Char}(p) $ and $\pi_\parallel(\varrho)\in\; ^\parallel\mathcal{H}_\partial$. If $\varrho=(t,x',z=0,\tau,\xi',\zeta)$ one writes that
 \begin{enumerate}[(a)]
 \item $ \varrho \in \mathcal{H}_\partial^+$ if $\zeta>0$
  \item $ \varrho \in \mathcal{H}_\partial^-$ if $\zeta<0$.
  \end{enumerate}
  A point $\varrho\in \mathcal{H}_\partial$ is also called an hyperbolic point.
\end{enumerate}
\end{definition}
Introducing the following involution on $ \partial(T^*\mathcal{L})$
  \begin{align*}
      \Sigma(t,x',z = 0,\tau,\xi',\zeta) = (t,x',z = 0,\tau,\xi',-\zeta),
  \end{align*}
  one finds that $\Sigma(\varrho^-) = \varrho^+$  if $\varrho \in  {^\parallel\mathcal{H}_\partial}$.\\
  By \eqref{hamiltA1} and Einstein convention, one has
\begin{align}\label{hpA1}
    H_pz(\varrho) = H_pz(x,\xi) = 2g^{d,j}(x)\xi_j,
\end{align}
where we recall that $x_d=z$ and $\xi_d=\zeta$. 
Since we are working with a $C^1$ metric, it follows from \eqref{hpA1} that $ H_pz$ is a $C^1$ function. In the chosen quasi-normal geodesic coordinates, introduced in Propositon \ref{pro9A1}, we have 
\begin{align*}
    H_pz_{|z=0}= 2\zeta.
\end{align*}
Hence,
\begin{align*}
\mathcal{G}_\partial=\,^\parallel\mathcal{G}_\partial=\{z=H_pz=p=0\}\quad \text{and} \quad \mathcal{H}_\partial^\pm=\{z = p = 0,\; H_pz \gtrless 0\}.
\end{align*}
\subsection{Glancing region, gliding vector field, and generalized bicharacteristics}
As $H_p^2z$ is continuous, one can introduce the following partition of $\mathcal{G}_\partial$.
\begin{definition}[Partition of the glancing region]
The set $\mathcal{G}_\partial$ is divided into three subsets: the diffractive set $\mathcal{G}_\partial^d$, the gliding set $\mathcal{G}_\partial^g$ and the glancing set of order three $\mathcal{G}_\partial^3$, defined as follows
\begin{align*}
&\mathcal{G}_\partial^d = \{ \varrho \in  \mathcal{G}_\partial; H^2_pz(\varrho) > 0\}, 
\\&
\mathcal{G}_\partial^3 = \{\varrho \in  \mathcal{G}_\partial; H^2_pz(\varrho) = 0\}, 
\\&
\mathcal{G}_\partial^g= \{\varrho \in  \mathcal{G}_\partial; H^2_pz(\varrho) < 0\}.
\end{align*}
Note that $\mathcal{G}_\partial^3$ is called the glancing set of order three, as any bicharacteristic that goes through $\varrho_0\in \mathcal{G}_\partial^3$ has a contact
with the boundary of order greater than or equal to three.
\end{definition}
On $^\parallel\partial(T^*\mathcal{L})$, one defines the gliding vector field as
\begin{align*}
    H^{\mathcal{G}}_p(\varrho) = (H_p+\frac{H^2_pz}{ H^2_z p} H_z) (\varrho).
\end{align*}
In quasi-normal geodesic coordinates, $H^2_z p= 2$ at the boundary. Further properties of $H^{\mathcal{G}}_p(\varrho)$ can be found in Section 5.4 of \cite{11A1}.\\
To define the notion of generalized bicharacteristics, introduce the vector field on $T^*\mathcal{L}$
 \begin{equation} ^GX(\varrho)= 
    \begin{cases}
      H_p(\varrho) \quad \text{if}\quad \varrho \in T^*\mathcal{L} \setminus\mathcal{G}_\partial^g, \\
         H^{\mathcal{G}}_p(\varrho) \quad \text{if}\quad \varrho \in \mathcal{G}_\partial^g.
    \end{cases}
\end{equation}
\begin{definition}[Generalized bicharacteristic] \label{generaA1}
 Suppose $J \subset \mathbb{R}$ is an interval, $B$ a discrete subset of $J$. A generalized bicharacteristic of $p$ is a map
    \begin{align*}
    ^G\gamma: J\setminus B \rightarrow \mathrm{Char}(p) \cap T^*\mathcal{L},
    \end{align*}
    satisfying the following properties:
    \begin{enumerate}
        \item For $s \in J\setminus B$, $^G\gamma \notin \mathcal{H}_\partial$ and the map $^G\gamma$ is differentiable at $s$ with
        \begin{align*}
            \frac{d}{ds} {^G\gamma(s)}= {^GX\big(^G\gamma(s)\big)}.
        \end{align*}
        \item If $ S \in B$, then $^G\gamma(s) \in T^*\mathcal{L}\setminus \partial(T^*\mathcal{L})$ for $s \in J \setminus B$ sufficiently close to $S$ and moreover 
        \begin{enumerate}
            \item 
         if $[S -\epsilon,S] \subset J$ for some $\epsilon > 0$, then $^G\gamma(S^-) := \underset{s \rightarrow S^-} {\mathrm{lim} }^G\gamma(s) \in\mathcal{H}_\partial^- $; 
         \item if $[S,S+ \epsilon] \subset J$ for some $\epsilon > 0$, then $^G\gamma(S^+) := \underset{s \rightarrow S^+} {\mathrm{lim} }^G\gamma(s) \in\mathcal{H}_\partial^+ $; 
         \item and if $[S- \epsilon,S +\epsilon] \subset J$ for some $\epsilon > 0$, then $^G\gamma(S^+) = \Sigma ( {^G\gamma(S^-)})$.
        \end{enumerate}
    \end{enumerate}
\end{definition}
The following Theorem asserts that for every point in $T ^*\mathcal{L}$, there exists a
maximal generalized bicharacteristic passing through it.
\begin{theorem} \label{theximaxiA1}
Let $^G\gamma(s) = (t(s), x(s), \tau (s), \xi(s))$ be a generalized bicharacteristic for $s\in J\setminus B$. If $^G\gamma$ is maximal then $J = \mathbb{R}$. Furthermore, $t(\mathbb{R}) = \mathbb{R}$ if
$\tau(s) = Cst \neq 0$.\\
If $\varrho_0 \in \mathrm{Char}(p)\cap T^*\mathcal{L}$, there exists a maximal generalized bicharacteristic 
$s\longmapsto {^G\gamma(s)}$ with $s \in \mathbb{R} \setminus B$ such that 
\begin{align*}
^G\gamma(0)= \varrho_0\;\; \text{if}\; \varrho_0 \notin \mathcal{H}_\partial \quad \text{and}\quad ^G\gamma(0^\pm)= \varrho_0\;\; \text{if}\; \varrho_0 \in \mathcal{H}_\partial^\pm.
\end{align*}
\end{theorem}
For the proof of Theorem \ref{theximaxiA1}, we refer the reader to Appendix A in \cite{11A1}.

\section{Semi-classical reduction} \label{reductionA1}
In this section we recall how the interior observability of Definition \ref{intersiwarnew} can
be deduced from the observability of solutions to \eqref{eqimA1} that are localized in frequency with a dyadic scale. The frequency localization of the solutions, described in section \ref{dyaA1}, allows us to apply powerful tools from semi-classical analysis, which are generally easier to manage than those from microlocal analysis and require much less regularity of the symbols.
\subsection{Dyadic decomposition } \label{dyaA1}
The proof of our main theorem relies on the use of a dyadic decomposition. Recall that the selfadjoint operator $A$, defined in \eqref{opA1}, with compact resolvent on $\mathcal{H}$, admits a sequence of eigenvalues $(\lambda_\nu)_{\nu \in \mathbb{Z}}$, compted with their finite multiplicities, such that
\begin{align*}
    \lambda_{-\nu}=-\lambda_\nu, \quad \lambda_\nu>0 \quad \text{if}\quad \nu>0 \quad \text{and}\quad \lambda_\nu \underset{\nu \rightarrow +\infty}{\longrightarrow}+\infty,
\end{align*}
 with associated sequence of eigenfunctions $(e_\nu)_\nu=(e_\nu^0,e_\nu^1)_\nu$, forming a Hilbert basis of $\mathcal{H}$.\\

Let $0<\alpha<1$, $\rho \in ]1,\frac{1}{\alpha}[$, $k\in \mathbb{Z}^*$ and denote $h_k=\rho^{-|k|}$,
\begin{align*}
J_k=\{\nu \in \mathbb{Z};\; k\lambda_\nu>0\;\text{and}\; \alpha\leq h_k |\lambda_\nu|<\alpha^{-1}\}.
\end{align*}
Consider
\begin{align*}
E_k=\mathrm{span}\{e_\nu;\; \nu\in J_k\},
\end{align*}
equipped with the
$\mathcal{H}$-norm 
\begin{align*}
\|U\|^2_{\mathcal{H}}=\underset{\nu\in J_k}{\sum }|U_\nu|^2,
\end{align*}
for all $U=\underset{\nu\in J_k}{\sum }U_\nu e_\nu\in E_k,$ where $U_{\nu}=(U,e_{\nu})_{\mathcal{H}}$ and $(U_\nu)_{\nu}\in\ell^2(\mathbb{C}).$ One has
\begin{align*}
    E_k\subset E^+ \quad \text{if} \quad  k>0 \quad \quad\quad\text{and}\quad \quad \quad E_k  \subset E^- \quad \text{if} \quad k<0.
\end{align*}

Observe that, using that $\mathrm{Card}(J_k)<\infty$, if $U\in E_k $ then $ A^nU \in E_k $ for all $n\in \mathbb{N}$. Hence, $E_k$ is a subspace of all the iterated domains of $A$.
One can identifies $U\in E_k$ with the following solution to the equation \eqref{eq2A1} :
\begin{align*}
U=\sum_{\nu \in J_k} e^{ it\lambda_\nu}U_\nu e_\nu.
\end{align*}
\begin{lemma} \label{equiA1}
For $U \in E_k$, $r \in \mathbb{N}$ and $s \in \mathbb{R}$ the norm
\begin{align*}
h_k^{s+r}\|\partial_t^rA^sU\|_{\mathcal{H}}, 
\end{align*}
is equivalent to $\|U\|_{\mathcal{H}}$, uniformly with respect to $k \in \mathbb{Z}^*.$
\end{lemma}
\begin{proof}[Proof]
One writes 
\begin{align*}
h_k^{2(s+r)}\|\partial_t^rA^sU\|_{\mathcal{H}}^2=\sum_{\nu \in J_k} |h_k\lambda_\nu|^{2(r+s)}|U_\nu|^2\simeq \sum_{\nu \in J_k} |U_\nu|^2=\|U\|_{\mathcal{H}}^2,
\end{align*}
as $h_k |\lambda_\nu| \simeq 1$ for $\nu \in J_k$.
\end{proof}
From Lemma \ref{equiA1}, one has in particular the equivalence
\begin{align} \label{eqqqqA1}
h_k\|\partial_tU\|_{\mathcal{H}}\simeq \|U\|_{\mathcal{H}}.
\end{align} 
In the next sections, we introduce the following sets of sequence of functions
\begin{align*}
B^\pm=\{(U_k)_{k\in \pm\mathbb{N}^*};\; U_k=(u_k,\partial_tu_k)\in E_k \quad \text{and}\quad \|U_k\|_{\mathcal{H}}\leq 1\}.
\end{align*}
 \subsection{Semi-classical observation}
Observability estimates for wave equations can be derived by several methods. Some approaches utilize a multiplier technique, tracing back to the article of Lions \cite{5A1}. Others adopt microlocal methods, following the work \cite{7A1} by Bardos, Lebeau, and Rauch.  In \cite{22A1}, Burq, Dehman, and Le Rousseau demonstrated that observability can be deduced from the observation of very particular types of waves. The waves they considered are frequency-localized, as discussed in section \ref{dyaA1}. In our context, by means of Theorem 1.2.7 in \cite{ben}, which is an adaptation of the semi-classical reduction developed in \cite{22A1}, our main result in Theorem  \ref{mainA1} is a consequence of the following proposition.
\begin{proposition} [Interior semi-classical observability for positive time-frequencies implies interior observability of solutions associated with positive time-frequencies]\label{proobsA1}
Let $\omega$ be a nonempty open subset of $\mathcal{M}$. Suppose there exist constants  $C>0$, $k_0>0$, and $\delta>0$, such that 
\begin{align} \label{semiclassicalinterA1}
\mathcal{E}(u_k)=\frac{1}{2}\|U_{k_{|t=0}}\|_{\mathcal{H}}^2\leq C\int_\delta^{T-\delta}\|\mathds{1}_{\omega}\partial_tu_k(t)\|_{L^2(\mathcal{M})}^2dt,
\end{align}
holds for any $(U_k)_{k\in \mathbb{N}^*}\in B^+$ and $k>k_0$.
Then, the wave equation \eqref{eqimA1} is observable from $\omega$ in time $T>0$ in the sense of Definition \ref{intersiwarnew}.
\end{proposition}
Proposition \ref{proobsA1} states that semi-classical observability for positive time-frequencies on an interval of length
$T-2\delta$ implies observability of solutions associated with positive time-frequencies on an interval of greater length.
Its proof relies in particular on the unique continuation property of the operator 
\begin{align*}
L:U=(u,\partial_tu) \mapsto \mathds{1}_{\omega}\partial_tu,
\end{align*}
which is bounded from $D(L)=\mathcal{H}=H^1(\mathcal{M})\times L^2(\mathcal{M})$ into $L^2(\mathcal{M})$, meaning that if $U$ is an eigenfunction of $A$ such that $LU = 0$ in $]0,T[\times \omega$,
then $U = 0$ (see Theorem 2.4 in \cite{33A1}).

\subsection{Proof strategy: initiation of the contradiction argument}\label{strategyA1}
In view of Proposition \ref{proobsA1}, to prove the interior observability of solutions associated with positive time-frequencies in the time interval $]0,T[$, it suffices to
prove the semi-classical observability for positive time-frequencies inequality \eqref{semiclassicalinterA1} for the time
interval $]\delta,T-\delta[$. To do so, we perform a contradiction argument assuming  that \eqref{semiclassicalinterA1}  does not hold. Thus, by Proposition \ref{proobsA1}, there exists a sequence of solutions $(U_k)_k=(u_k,\partial_tu_k)_k\in B^+$ to the equation \eqref{eq2A1} with initial data $(\underline{U_k})_k=(\underline{u}_k^0,\underline{u}_k^1)_k\subset\mathcal{H}$, such that $h_k^2\|U_{k_{|t=0}}\|_{\mathcal{H}}^2=1$ and $\int_\delta^{T-\delta}h_k^2\|\mathds{1}_{\omega}\partial_tu_k(t)\|_{L^2(\mathcal{M})}^2dt \rightarrow 0$ as $k\rightarrow +\infty$, where $h_k$ is defined as in section \ref{dyaA1}. The functions $u_k$ are solutions to 
\begin{equation}\label{equofcontraA1}
    \begin{cases}
        (\partial_t^2-\Delta_g+1)u_k = 0 \quad \quad \quad\quad\quad\text{in} \;\mathbb{R}\times \mathcal{M}, \\
        \partial_n u_k+i\partial_tu_k= 0 \;\; \quad \quad\quad\quad\quad \quad\text{in} \;\mathbb{R}\times \mathcal{\partial M},\\
        u_{k_{\mid t=0}}=\underline{u}_k^0,\; \partial_tu_{k_{\mid t=0}}=\underline{u}_k^1 \quad\quad\quad \text{in} \;\mathcal{ M}.
    \end{cases}
\end{equation}
In section \ref{uuuA1}, we show that these sequences are associated with a semi-classical measure $\mu$ that satisfies a propagation equation at the boundary as stated in Theorem \ref{th1A1}. This equation along with the result from \cite{11A1}, allow us to conclude the proof of the semi-classical observability inequality of the form \eqref{semiclassicalinterA1} by reaching a contradiction in section \ref{conclusionA1}.

\section{Semi-classical operators and measures} \label{semiopA1}
This section is derived from \cite{1A1}. Here, we review some fundamental facts about semi-classical pseudo-differential operators and semi-classical measures. We state in particlular existence results of semi-classical measures in Propositions \ref{pro1.2µµA1} and \ref{Pro2µµA1}.
\subsection{Euclidean framework} \label{euclideA1}
In this part, we recall the concepts on $\mathbb{R}^d$, starting with the definition of symbols. We use the notation, $$\langle x\rangle  = (1 + |x|^2)^{1/2}.$$ 
\begin{definition}(Symbols)
Let $m, n \in \mathbb{N} \cup \{+\infty\}$, with $n \geq d+1$, and let $N \in \mathbb{R}^+$. We define the symbol class $\Sigma^{m,n}(\langle \xi \rangle^{-N}; \mathbb{R}^{2d})$
as the space of all functions $a(x,\xi)$, with $x, \xi \in \mathbb{R}^d$, such that
\begin{align*}
\partial_x^\alpha \partial_\xi^\beta a \in L^1_{\mathrm{loc}}(\mathbb{R}^{2d}) \quad \text{for all }\alpha, \beta \in \mathbb{N}^d \text{ with } |\alpha| \leq m, \; |\beta| \leq n,
\end{align*}
and satisfying 
\begin{align*}
M_{m,n}^{-N}(a) := \max_{\substack{|\alpha| \leq m \\ |\beta| \leq n}} \operatorname*{ess\,sup}_{(x,\xi) \in \mathbb{R}^{2d}} \left| \partial_x^\alpha \partial_\xi^\beta a(x,\xi) \right| \langle \xi \rangle^N < \infty.
\end{align*}

Moreover, we define $\Sigma^{m,n}_0(\langle\xi\rangle^{-N};\mathbb{R}^{2d})$ as the subset of $ \Sigma^{m,n}(\langle\xi\rangle^{-N};\mathbb{R}^{2d})$ consisting of all symbols $a$ such that 
\begin{align*}
\partial_x^\alpha \partial_\xi^\beta a \in C_0(\mathbb{R}^{2d}) \quad \text{for all } \alpha, \beta \in \mathbb{N}^d \text{ with } |\alpha| \leq m, \; |\beta| \leq n - 1 - d.
\end{align*}

\end{definition}
Equipped with the norm $M^{-N}_{m,n}(.)$, both spaces $\Sigma^{m,n}(\langle\xi\rangle^{-N};\mathbb{R}^{2d})$ and $\Sigma^{m,n}_0(\langle\xi\rangle^{-N};\mathbb{R}^{2d})$ are complete.\\
In what follows, we will work with the spaces
\begin{align*}
\Sigma(\mathbb{R}^{2d})=\Sigma^{0,d+1}(\langle\xi\rangle^{-(d+1)};\mathbb{R}^{2d}),
\end{align*}
and 
\begin{align*}
\Sigma_0(\mathbb{R}^{2d})=\Sigma^{0,d+1}_0(\langle\xi\rangle^{-(d+1)};\mathbb{R}^{2d}).
\end{align*}
We define the following space, which is dense in \( \Sigma^{\infty,\infty}_0(\langle \xi \rangle^{-\infty}; \mathbb{R}^{2d}) \).
\begin{definition} \label{definespaceA1}
We write $ a(y,\eta)\in \Sigma^\mathcal{H}_0(\mathbb{R}^{2d+2})$ with $y=(t,x)\in \mathbb{R}\times \mathbb{R}^d$ and $\eta=(\tau,\xi)\in \mathbb{R}\times \mathbb{R}^d$, if $a \in \Sigma^{\infty,\infty}_0(\langle\xi\rangle^{-\infty};\mathbb{R}^{2d+2})$ and verifies the following properties:
\begin{enumerate}
\item $a(y,\eta)$ is compactly supported in the $y$ variable.
\item  $a(y,\eta)$ has a compactly supported Fourier transform in the $\eta$ variable and, therefore, is holomorphic with respect to the $\eta$ variable.
\end{enumerate}
\end{definition}
We now introduce the definition of tangential symbols.
\begin{definition}(Tangential Symbols)
\label{def 3.2A1}
Let $m, n \in \mathbb{N} \cup \{+\infty\}$, with $n \geq d$, and $N \in \mathbb{R}^+$. We define $\Sigma^{m,n}_T(\langle \eta' \rangle^{-N}; \mathbb{R}^d \times \mathbb{R}^{d-1})$ as the set of functions $a(y, \eta')$, where $y \in \mathbb{R}^d$, $\eta' \in \mathbb{R}^{d-1}$, such that 
\begin{align*}
\partial^\alpha_y \partial^\beta_{\eta'} a \in L^1_{\mathrm{loc}}(\mathbb{R}^{2d-1}) \quad \text{for all } \alpha \in \mathbb{N}^d, \, \beta \in \mathbb{N}^{d-1} \text{ with } |\alpha| \leq m, \; |\beta| \leq n,
\end{align*}
and satisfying
\begin{align*}
M^{-N}_{T,m,n}(a) := \underset{\substack{|\beta| \leq n \\ |\alpha| \leq m}}{\text{max}} \underset{(y, \eta') \in \mathbb{R}^d \times \mathbb{R}^{d-1}}{\text{ess sup}} \, \left| \partial^\alpha_y \partial^\beta_{\eta'} a(y, \eta') \right| \langle \eta' \rangle^N < \infty.
\end{align*}
We denote by $
\Sigma^{m,n}_{T,0}(\langle \eta' \rangle^{-N}; \mathbb{R}^d \times \mathbb{R}^{d-1})$ the subset of $
\Sigma^{m,n}_T(\langle \eta' \rangle^{-N}; \mathbb{R}^d \times \mathbb{R}^{d-1}),$
consisting of all tangential symbols $a$ such that 
\begin{align*}
\partial^\alpha_y \partial^\beta_{\eta'} a \in C_0(\mathbb{R}^{2d-1}) \quad \text{for all } \alpha \in \mathbb{N}^d, \, \beta \in \mathbb{N}^{d-1} \text{ with } |\alpha| \leq m, \;|\beta| \leq n-d.
\end{align*}
\end{definition}
Equipped with the norm
$M^{-N}_{T,m,n}(.)$,
both spaces $\Sigma^{m,n}_T(\langle\eta'\rangle^{-N};\mathbb{R}^{d}\times \mathbb{R}^{d-1})$ and  \\$\Sigma^{m,n}_{T,0}(\langle\eta'\rangle^{-N};\mathbb{R}^{d}\times \mathbb{R}^{d-1})$ are complete. 

From now on,  we denote by $h\leq 1$ a small positive semi-classical parameter.
\begin{definition}(Semi-classical operators)
For  $ a\in \Sigma(\mathbb{R}^{2d})$, we define the associated semi-classical operator by  
\begin{align}
\label{8µA1}
\Op^h(a)u(x)=a(x, hD_x)u(x)= (2\pi)^{-d}\int _{\mathbb{R}^d} e^{ix.\xi} a(x,h\xi)\hat{u}(\xi)d\xi,  
\end{align}
for all $u\in \mathcal{S}(\mathbb{R}^d)$.
\end{definition}
The kernel of $\Op^h(a)$ is expressed as
\begin{align}
K(x,y)&=(2\pi)^{-d}\int e^{i\xi.(x-y)}a(x,h\xi)d\xi=(2\pi h)^{-d}\int e^{i\xi.\frac{(x-y)}{h}}a(x,\xi)d\xi \nonumber
\\&=h^{-d} k(x,\frac{x-y}{h}),
\end{align}
with \begin{align} \label{noyauA1}
k(x,y)=(2\pi)^{-d}\int e^{i\xi.y}a(x,\xi)d\xi.
\end{align}
\begin{remark}
Note that taking $L=(1-iy.\nabla_\xi)/\langle y \rangle ^2$, one has $L\, exp(iy.\xi)=\mathrm{exp}(iy.\xi)$ and so by integration by parts one obtains
\begin{align} \label{4.333A1}
k(x,y)=(2\pi)^{-d}\int e^{i\xi.y} (^tL)^N a(x,\xi)d\xi,
\end{align}
for $N\leq d+1$, with $^tL=(1+iy.\nabla_\xi)/\langle y\rangle^2$. Therefore,
\begin{align}
\label{10µµA1}
|k(x,y)|\leq M_{0,d+1}^{-(d+1)}(a)\langle y\rangle^{-(d+1)}, 
\end{align}
for all $x,y\in \mathbb{R}^d$.
\end{remark}
\begin{proposition}
\label{corappenA1}
Consider the operator $A$ defined on $L^2(\mathbb{R}^d)$ by   \begin{align*} Af(x)=\int  K(x,y) f(y)dy,
\end{align*}
with kernel $K(x,y)=h^{-d}k(x,\frac{x-y}{h})$ such that 
\begin{align*}
|k(x,\nu)|\leq L_0\langle \nu\rangle^{-d-\delta}, \quad \quad \quad x,\nu\in \mathbb{R}^d, 
\end{align*}
for some $\delta>0$ and $L_0 >0$. Then $A$ extends as a continuous operator on $L^2(\mathbb{R}^d)$ with
\begin{align*} \|A\|_{\mathcal{L}(L^2(\mathbb{R}^d))}\leq C_{d,\delta}L_0,
\end{align*}
for some $C_{d,\delta}>0$.
\end{proposition}
The Proof of Proposition \ref{corappenA1} is based on Schur's lemma and can be found in \cite{1A1}. 

Using \eqref{10µµA1} and Proposition \ref{corappenA1}, we obtain the following corollary.
\begin{corollary} \label{corolary4.7A1}
If  $ a\in \Sigma(\mathbb{R}^{2d})$, then $\Op^h(a):L^2(\mathbb{R}^d)\longrightarrow L^2(\mathbb{R}^d)$ continuously.
\end{corollary}
With Proposition \ref{corappenA1}, we also have the following result. For a detailed proof, we refer the reader to Lemma 5.12 in \cite{1A1}.
\begin{corollary} \label{lem4.7A1}
For $a\in \Sigma_T^{0,d}(\langle\xi\rangle^{-(d+1)};\mathbb{R}^{d}\times \mathbb{R}^{d-1})$,
the associated operator 
\begin{align*}
\Op^h(a)u(y)=a(y, hD'_y)u(z,y')= (2\pi)^{1-d}\int _{\mathbb{R}^{d-1}} e^{iy'.\eta'} a(z,y',h\eta')\hat{u}(z,\eta')d\eta', \quad u\in \mathcal{S}(\mathbb{R}^d)
\end{align*}
with $y = (y',z)\in \mathbb{R}^{d-1}\times \mathbb{R}$, $\eta' \in \mathbb{R}^{d-1}$ (where the Fourier transform is taken with respect to the $y'$ variables) extends to a uniformly bounded operator on $L^2(\mathbb{R}^d)$.
\end{corollary}

In what follows, we define a sequence of scales $H=(h_k)_k$ as a sequence of positive real numbers converging to $0$. When utilizing this type of sequence, we will use $\Op^h$ instead $\Op^{h_k}$ for the sake of simplicity.\\

We now recall the definition of a semi-classical measure on $\mathbb{R}^d$. 
\begin{definition}(Semi-classical measure) \label{semiinRA1}
Consider a bounded sequence $(u_k)_k$ in $L^2(\mathbb{R}^d)$ and a sequence of scales  $H=(h_k)_k$. Let $\mu$ be a nonnegative Radon measure on $\mathbb{R}^{2d}$. We say that the sequence  $(u_k)_k$ admits $\mu$ as a semi-classical measure at
scale $H$ if the following holds:
\begin{align*}
 (\Op^h(a)u_k, u_k )_{L^2(\mathbb{R}^d)}\underset{k\rightarrow +\infty}{\longrightarrow} \int_ {\mathbb{R}^{2d}} a(x,\xi) d\mu(x,\xi)=\langle\mu, a\rangle,
\end{align*}
for every $a\in \Sigma_0(\mathbb{R}^{2d})$.
\end{definition}
We also recall the notion of mass leakage at infinity.
\begin{definition} [Mass leakage at infinity]\label{massA1} We say that no mass leaks at infinity at scale $H$ if the following condition holds
\begin{align*}
\underset{R\rightarrow +\infty}{\mathrm{lim}}\underset{k\rightarrow +\infty}{\mathrm{lim \; sup}}\;\Big(\int_ {|x|\geq R} |u_k(x)|^2dx+\int_{h_k |\xi|\geq R}|\hat{u}_k(\xi)|^2 d\xi\Big)=0.
\end{align*}
Otherwise, we say that there is mass leakage at scale $H$ at infinity.
\end{definition}
\subsection{On a manifold}
In what follows, let $\mathcal{N}$ be a $C^1$ manifold of dimension $d$ equipped with a density measure $\rho$ that allows one to define $L^2(\mathcal{N})$. We denote by $\Sigma_c(T^*\mathcal{N})$ the space of functions $a\in C_c(T^*\mathcal{N})$ such that 
\begin{align}
\partial^\beta_\xi a\in C_c(T^*\mathcal{N})\quad for \quad |\beta|\leq d+1.
\end{align}
Let  $\mathcal{B}(\mathcal{N})$  represent the algebra of bounded families $B_h$ of operators on
$L^2(\mathcal{N})$ and $\mathcal{R}(\mathcal{N})$ the ideal of $\mathcal{B}(\mathcal{N})$ consisting of operators $B_h$ that converge
uniformly to $0$. The quotient algebra  of $\mathcal{B}(\mathcal{N})$ by $\mathcal{R}(\mathcal{N})$ is denoted by $\mathcal{Q}(\mathcal{N})$.
For $a \in  \Sigma_c(T^*\mathcal{N})$ supported in $T^*\mathcal{N}_{|O}$, where $C = (O, \phi)$ is a local chart of $\mathcal{N}$, we use $a^c$ to denote the local representative of $a$ in this chart. Consequently,
$a^c\in C(V \times \mathbb{R} )$ for some open subset $V \subset \mathbb{R}^d$, furthermore, the derivatives $\partial^\beta_\xi a^c$ are continuous for $ |\beta|\leq d+1$. Let $\theta \in C_c(O)$ be continuous function with $\theta=1$ near the projection on $\mathcal{N}$ of $\mathrm{supp}( a)$.
Consider the family of operators $(A^h(\theta, \phi))\in \mathcal{B}(\mathcal{N})$ defined by 
\begin{align}
\label{13££££A1}
(A^h(\theta, \phi)u)\circ \phi^{-1}=a^c(x,hD_x)(\theta u\circ \phi^{-1}),
\end{align}
for all $u \in L^2(\mathcal{N} )$.
According to Lemma 1.10 in \cite{10A1}, the class of the operator $(A^h(\theta, \phi))$ in $\mathcal{Q}(\mathcal{N})$ does not depend on $\theta$ and $\phi$, and so we write $[\Op^h](a)$ to denote this class. Furthermore, let $B_h$ and $\tilde{B}_h$ be two representatives of a class in $\mathcal{Q}(\mathcal{N})$,
that is, $[B_h] = [\tilde{B}_h]$. Note that, for $\varphi \in C(\mathcal{N})$, one has $[\varphi B_h] = [\varphi \tilde{B}_h]$. Thus, one writes $[\varphi B_h] = [\varphi] [B_h]$ where $[\varphi]$ stands for the image in $\mathcal{Q}(\mathcal{N})$ of the multiplication by $\varphi$. Hence, for $a\in \Sigma_c(T^*\mathcal{N})$
one has
\begin{align}
\label{13££A1}
[\Op^h](\varphi a)=[\varphi][\Op^h](a).
\end{align}
\begin{proposition}
\label{pro5555A1}
There exists a unique linear map $[\Op^h]:\Sigma_c(T^*\mathcal{N})\rightarrow \mathcal{Q}(\mathcal{N})$ extending the one defined in \eqref{13££££A1} and satisfying \eqref{13££A1} for any continuous function $\varphi$ and any $a \in \Sigma_c(T^*\mathcal{N})$.
\end{proposition}
The proof of Proposition \ref{pro5555A1} follows from the use of continuous partitions of unity (see Proposition 1.11 in \cite{10A1}).

Denote $\lambda = \ell^\infty/c_0$ the space of bounded complex sequences modulo
sequences converging to $0$. Consider $(u_k)_k$ a bounded sequence in $L^2(\mathcal{N} )$ and
$H =(h_k)_k$ sequence of scales. For $a\in \Sigma_c(T^*\mathcal{N})$, 
\begin{align*}
[([\Op^h](a)u_k,u_k)_{L^2(\mathcal{N})}]_\lambda
\end{align*}
stands for the class in $\lambda$ of the sequence $([\Op^h](a)u_k,u_k)_{L^2(\mathcal{N})}$.
In the case where $(u_k)_k$ is bounded in $L^2_{
\mathrm{loc}}(\mathcal{N})$ it is appropriate to compute
$$([\Op^h ](a)\psi u_k , u_k)_{L^2(\mathcal{N})}$$
for $a\in \Sigma_c(T^*\mathcal{N})$ and $\psi\in C^\infty_c(\mathcal{N})$ with $\psi=1$ on $\mathrm{supp}(a)$. 

Using the above notions, we define a semi-classical measure on the case of manifolds as following.
\begin{definition}[Semi-classical measure on a manifold]
\label{def11µA1}
Consider a bounded sequence $(u_k)_k$ in $L^2_{\mathrm{loc}}(\mathcal{N};\mathbb{C})$ and a sequence of scales  $H=(h_k)_k$. Let $\mu$ be a nonnegative Radon measure on $T^*\mathcal{N}$. We say that the sequence  $(u_k)_k$ admits $\mu$ as a semi-classical measure at
scale $H$ if:
\begin{align*}
 [([\Op^h](a)\psi u_k, u_k )_{L^2(\mathcal{N})}]_\lambda \underset{k\rightarrow +\infty}{\longrightarrow}\langle \mu, a\rangle,
\end{align*}
for all $a\in \Sigma_c(T^*\mathcal{N})$ and $\psi\in C^\infty_c(\mathcal{N})$ such that $\psi=1$ on $\mathrm{supp}(a)$.
\end{definition}
\begin{remark} \label{remarK 4.11A1}
The Definition \ref{def11µA1} does not
depend on the choice of the function $\psi$, and aligns with 
the Definition \ref{semiinRA1} of a semi-classical measure in the case of a $L^2(\mathbb{R}^d)$-bounded sequence: Let $\mu$ be a semi-classical measure associated with a sequence $(u_k)_k$ in $L^2_{\mathrm{loc}}(\mathcal{N})$, and let $C =(O, \phi)$ be any local chart. We denote   by $\mu^c=(\phi^{-1})^*\mu$ the local representative of $\mu$ in $C$, and
by $u^c_k$ the local representative of $u_k$, that is, $u^c_k= (\phi^{-1})^*u_k = u_k \circ \phi^{-1}$. Then, for any compact
$K \subset \phi(O)$, $a \in \Sigma_0(\mathbb{R}^{2d})$ with $\mathrm{supp}( a) \subset K \times \mathbb{R}^d$, and $\psi \in C^\infty_c(\phi(O))$ with
$\psi=1$ in a neighborhood of the x-projection of $\mathrm{supp}(a)$, we have
\begin{align*}
(\Op^h(a)\psi u_k^c, u_k^c )_{L^2(\mathbb{R}^d)}\underset{k\rightarrow +\infty}{\longrightarrow}\langle \mu^c, a\rangle.
\end{align*}
\end{remark}
\begin{definition}[Countable at infinity space] \label{countableth}
    A topological space $X$ is called countable at infinity (or $\sigma$-compact) if $X$ can be represented as the union of a sequence of compact subsets of $X$.
\end{definition}
\begin{proposition}
\label{pro1.2µµA1}
Suppose $H = (h_k)_k$ is a sequence of scales and $(u_k)_k$ a
sequence of functions on $\mathcal{N}$.
\begin{itemize}
\item If $(u_k)_k$ is bounded in $L^2(\mathcal{N})$, then there exists a semi-classical measure $\mu$ at scale $H$ associated to the sequence $(u_k)_k$.
\item If $\mathcal{N}$ is countable at infinity and $(u_k)_k$ is bounded in $L^2_
{\mathrm{loc}}(\mathcal{N} )$, then there exists a semi-classical measure $\mu$ at scale $H$ associated to the sequence $(u_k)_k$.
\end{itemize}
\end{proposition}
For a proof of Proposition \ref{pro1.2µµA1}, we refer the reader to Section 1 in \cite{10A1} and Proposition 5.25 in \cite{1A1}. \\

The Defintion \ref{def11µA1} can be extended to the case of a vector valued sequences as following.
\begin{definition}(Hermitian semi-classical measures)
For $N \in \mathbb{N}^*$, consider $(u_k)_k$ a bounded sequence in $L^2_{\mathrm{loc}}(\mathcal{N};\mathbb{C}^N)$ and $H=(h_k)_k$ a sequence of scales. Let $\mu$ be a nonnegative $N\times N$-matrix valued Radon measure on $T^*\mathcal{N}$. One says that the sequence  $(u_k)_k$ admits $\mu$ as a semi-classical measure at
scale $H$ if:
\begin{align*}
 [([\Op^h](a)\psi u_k, u_k )_{L^2(\mathcal{N})}]_\lambda\underset{k\rightarrow +\infty}{\longrightarrow}\langle \mathrm{tr}\big(a(x,\xi)\mu\big), 1\rangle=\int_{T^*\mathcal{N}}\mathrm{tr}\big(a(x,\xi)d\mu(x,\xi)\big),
\end{align*}
for all $N\times N$-matrix $a$ with entries in $ \Sigma_c(T^*\mathcal{N})$ and $\psi\in C^\infty_c(\mathcal{N})$ such that $\psi=1$ on $\mathrm{supp}(a)$.
\end{definition}
The following result is the analogue of Proposition \ref{pro1.2µµA1}.
\begin{proposition}
\label{Pro2µµA1}
Let $N \in \mathbb{N}^*$ and $H = (h_k)_k$  a sequence of scales. Consider $(u_k)_k$ a sequence of function on $\mathcal{N}$ valued in $\mathbb{C}^N $.
\begin{itemize}
\item If $(u_k)_k$ is bounded in $L^2(\mathcal{N};\mathbb{C}^N)$, then there exists a Hermitian semi-classical measure $\mu$ at scale $H$ associated to the sequence $(u_k)_k$.
\item If $\mathcal{N}$ is countable at infinity and $(u_k)_k$ is bounded in $L^2_{\mathrm{loc}}(\mathcal{N};\mathbb{C}^N)$, then there exists a Hermitian semi-classical measure $\mu$ at scale $H$ associated to the sequence $(u_k)_k$.
\end{itemize}
\end{proposition}
\section{The measure propagation equation} \label{uuuA1}
Recall that $\mathcal{L}=\mathbb{R}\times \mathcal{M}$ with boundary $\partial\mathcal{L}=\mathbb{R}\times \partial\mathcal{M}$.  Suppose that the sequence $(u_k)_k$ of weak-solutions
to the system \ref{eqimA1}, as given in Proposition \ref{proµµµA1}, is bounded in $L^2_{\mathrm{loc}}(\mathcal{L})$ and that $(u_{k_{\mid \partial \mathcal{L}}})_k$ and $(h_k \partial_n u_{k_{\mid \partial \mathcal{L}}})_k$ are bounded in $L^2_{\mathrm{loc}}(\partial\mathcal{L})$. By Proposition \ref{pro1.2µµA1}, there exists a semi-classical measure $\mu$ at scale $H$ associated to the sequence $(u_k)_k$. Furthermore, by Proposition \ref{Pro2µµA1}, there exists a Hermitian semi-classical measure on $T^*\partial \mathcal{L}$ associated to the vector $\tilde{U}_k= \,^t(u_{k_{\mid \partial \mathcal{L}}},h_k \partial_n u_{k_{\mid \partial \mathcal{L}}})$ given by
\begin{align*}
\nu=
\begin{pmatrix}
    \nu^{0,0} & \nu^{0,1} \\
    \nu^{1,0} & \nu^{1,1} \\
\end{pmatrix}.
\end{align*}
Our measure propagation equation is the following.
\begin{theorem}
\label{th1A1}
Suppose that
\begin{align*}
\mathrm{supp}(\mu)\subset \mathrm{Char}(p)\cap T^*\mathcal{L}\setminus 0 \quad\text{and}\quad \mathrm{supp}(\nu^{i,j})\subset T^*\partial\mathcal{L}\setminus 0 \quad \text{for}\quad i,j \in \{0,1\}.
\end{align*}
Then, one has
\begin{align}
\label{13£££A1}
&H_p\mu=-^t H_p\mu \nonumber
\\&=\int_{\,^\parallel \mathcal{H}_\partial \cup \,^\parallel \mathcal{G}_\partial}  (\delta_{\varrho^{+}}+\delta_{\varrho^{-}}) \;d\big(\mathrm{Im}(\nu^{1,0}) \big)+\int_{\,^\parallel \mathcal{H}_\partial \cup \,^\parallel \mathcal{G}_\partial} \frac{\delta_{\varrho^{+}}-\delta_{\varrho^{-}}}{\zeta^+-\zeta^-} d\big(\nu^{1,1}+R(\tau,x',\xi')\nu^{0,0}\big),
\end{align} 
where \begin{align*}
R(\tau,x',\xi')=\tau^2-\underset{1\leq i,j\leq d-1}{\sum}g^{i,j}(x',0)\xi_i\xi_j.
\end{align*}
Here, $^\parallel \mathcal{H}_\partial$ and  $^\parallel \mathcal{G}_\partial$ denote respectively the hyperbolic and glancing sets as introduced in Definition \ref{defn11A1}, and $ \varrho^\pm$, $ \zeta^\pm$ are given in \eqref{6£££A1}.
\end{theorem}
Note that, the identification $T^*\partial \mathcal{L}\simeq\;^\parallel \partial(T^*\mathcal{L})$ allows us to interpret the measures $\nu^{0,1}$, $\nu^{1,0}$ and $\nu^{1,1}$ originally defined on $T^*\partial \mathcal{L}$ as measures on $^\parallel \partial(T^*\mathcal{L})$. As a result, integration over the set  $^\parallel \mathcal{H}_\partial \cup \,^\parallel \mathcal{G}_\partial$ is well-defined.
\subsection{First results and observations}
In this section, we begin by proving in Lemma \ref{lem0A1} that the sequence $(u_k)_k$, introduced in the contradiction argument developed in Section \ref{strategyA1}, and $(h_k\partial_t u_k)_k$ are bounded in $L^2_{\mathrm{loc}}(\mathcal{L})$. Furthermore, we show that $(h_ku_k(t,.))_k$ and $(h^2_ku_k(t,.))_k$ are bounded in $H^1(\mathcal{M})$ and $H^2(\mathcal{M})$, respectively, for any $t \in \mathbb{R}$. The boundedness of $(u_k)_k$ leads, through Proposition \ref{pro1.2µµA1}, to the existence of a semi-classical measure 
$\mu$. Assuming that the traces of $(u_k)_k$ are bounded in $L^2_{\mathrm{loc}}(\partial\mathcal{L})$, which will be verified in Section \ref{exisµA1}, Proposition \ref{Pro2µµA1} further ensures the existence of a Hermitian measure $\nu$. Finally, by applying Propositions \ref{pro*A1} and \ref{mesurenonnulA1}, together with Lemmas \ref{lem0A1} and \ref{massleakA1} stated bellow, we derive  crucial properties of the measures $\mu$ and $\nu$, outlined in Proposition \ref{hhA1}. The results presented here are of fundamental importance to prove, in section \ref{proofpropagA1}, the propagation equation stated in Theorem \ref{th1A1}.

We recall that $n$ denotes the unitary normal inward pointing vector field to $\partial\mathcal{M}$. As mentioned in Section \ref{GeoA1}, in the quasi-normal geodesic coordinates introduced in Propositon \ref{pro9A1}, one has $\partial_n=\partial_z$ at $z=0$, where $n$ stands for the unitary normal inward pointing vector field to $\partial\mathcal{M}$ in the sense of the metric $g$.

\begin{proposition}
\label{pro2.1*A1}
For any $f\in H^{-1}(\mathcal{M})$ and $f_{\partial}\in H^{-\frac{1}{2}}(\partial\mathcal{M})$, there exists a unique $u\in H^1(\mathcal{M})$ that is a solution of
\begin{equation}
\label{10**A1}
    \begin{cases}
        -\triangle_g u+u &= f    \\ 
        \partial_nu_{\mid \mathcal{\partial M}}&= f_{\partial},
    \end{cases}
\end{equation} 
and \begin{align*}
 \|u\|_{H^1(\mathcal{M})}\leq C\Big(\|f\|_{H^{-1}(\mathcal{M})}+\|f_{\partial}\|_{H^{-\frac{1}{2}}(\partial\mathcal{M})}\Big).
\end{align*}
\end{proposition}
The proof of Proposition \ref{10**A1} can be found in \cite{24A1}.
This proposition leads to the following result.
\begin{proposition}
\label{pro*A1}
In addition to the hypothesis of Proposition \ref{pro2.1*A1} assume that $f\in L^2(\mathcal{M})$ and $f_{\partial}\in H^{\frac{1}{2}}(\partial\mathcal{M})$, 
then $u \in H^2(\mathcal{M})$  if $u$ is the  solution of \eqref{10**A1} and
\begin{align*}
    \|u\|_{H^2(\mathcal{M})}\leq C\Big(\|f\|_{L^2(\mathcal{M})}+\|f_{\partial}\|_{H^{\frac{1}{2}}(\partial\mathcal{M})}\Big).
    \end{align*}
\end{proposition}
\begin{proof}[Proof]
It suffices to obtain the result locally near the boundary. Recall that in local coordinates $\mathcal{M}$ is given by $\{z\geq0\}$.\\
Let $1\leq j\leq {d-1}$. Applying $\partial_{x_j}$ to the first equation of \eqref{10**A1} we obtain
\begin{align}\label{tildefA1}
\tilde{f}:=-\triangle_g \partial_{x_j}u+\partial_{x_j}u=\partial_{x_j} f-[\partial_{x_j},-\triangle_g]u.
\end{align}
Since $[\partial_{x_j},-\triangle_g]$ is an operator of degree 2, and $u\in H^1(\mathcal{M})$, it follows that $[\partial_{x_j},-\triangle_g]u \in H^{-1}(\mathcal{M})$. Using further $\partial_{x_j}f\in H^{-1}(\mathcal{M})$,
we deduce from \eqref{tildefA1} that
$\tilde{f}\in H^{-1}(\mathcal{M}).$\\
Applying similarly $\partial_{x_j}$ to the second equation of \eqref{10**A1} we get
\begin{align*}
\tilde{f_{\partial}}:=\partial_n\partial_{x_j}u_{\mid \partial\mathcal{M}}=\partial_{x_j}f_{\partial}.
\end{align*}
Hence $\partial_{x_j}f_{\partial}\in H^{-\frac{1}{2}}(\partial\mathcal{M})$ gives $\tilde{f_{\partial}} \in  H^{-\frac{1}{2}}(\partial \mathcal{M})$.\\
Given that $\tilde{f}\in H^{-1}(\mathcal{M})$ and $\tilde{f_{\partial}} \in  H^{-\frac{1}{2}}(\partial \mathcal{M})$, it follows from Proposition \ref{pro2.1*A1} that $\partial_{x_j}u \in H^1(\mathcal{M})$, which in turn gives $\partial_{x_i}\partial_{x_j}u \in L^2(\mathcal{M})$ for $(i,j) \neq (d,d)$.\\
Therefore, writting the first equation of \eqref{10**A1} in local coordinates,
\begin{align*}
-(\text{det}\; g)^{-\frac{1}{2}}\partial_z \Big((\text{det}\; g)^{\frac{1}{2}} &g^{d,d}(x)\partial_zu\Big)
\\&- (\text{det}\; g)^{-\frac{1}{2}}\underset{\underset{(i,j) \neq (d,d)}{1\leq i,j\leq d}}{\sum}\partial_{x_i}\Big((\text{det}\; g)^{\frac{1}{2}}g^{i,j}(x)\partial_{x_j}u\Big)+u=f\in L^2(\mathcal{M}),
\end{align*}
we conclude that $\partial_z^2u \in L^2(\mathcal{M})$, which completes the proof.

\end{proof}
In the following results, we establish certain properties of the sequence introduced in the contradiction argument initiated in Section \ref{strategyA1}. Here we employ the same notations used in sections \ref{GeoA1} and \ref{reductionA1}. We recall that $\mathcal{L}=\mathbb{R}\times \mathcal{M}$ and $\partial\mathcal{L}=\mathbb{R}\times \partial\mathcal{M}$.
\begin{lemma}
\label{lem0A1}
The sequences $(u_k)_k$ and $(h_k\partial_t u_k)_k$ are bounded in $L^2_{\mathrm{loc}}(\mathcal{L})$. Furthermore, 
 $(h_ku_k(t,.))_k$ and $(h^2_ku_k(t,.))_k$ are bounded in $H^1(\mathcal{M})$ and $H^2(\mathcal{M})$, respectively, for any $t \in \mathbb{R}$, uniformly with respect to $k\in \mathbb{N}^*$.
\end{lemma}
\begin{proof}[Proof]
For $U_k=(u_k,\partial_tu_k)\in B^+$, one writes 
\begin{align*}
U_k=\sum_{\nu \in J_k} e^{ it\lambda_\nu}U_\nu e_\nu,
\end{align*}
where $e_\nu=(e_\nu^0,e_\nu^1)$, and $u_k(t)=\sum\limits_{\nu \in J_k} e^{ it\lambda_\nu}U_\nu e_\nu^0.$
Using the hypothesis $h_k\|U_{k_{\mid t=0}}\|_{\mathcal{H}}=1$, along with the expression
$U_{k_{\mid t=0}}=\sum\limits_{\nu \in J_k} U_\nu e_\nu$,
and the fact that $(e_\nu)_\nu$ forms an Hilbert basis of $\mathcal{H}$, we obtain  \begin{align}\label{2.8888A1}
h_k^2\sum_{\nu \in J_k} |U_\nu|^2=1.
\end{align}
Define \begin{align*}
V_k(t)=(v_k(t),\partial_tv_k(t))=\sum\limits_{\nu \in J_k} e^{ it\lambda_\nu}\frac{U_\nu}{i\lambda_\nu} e_\nu.
\end{align*}
By \eqref{2.8888A1} and the equivalence $h_k |\lambda_\nu| \simeq 1$ for $\nu \in J_k$,
we get
\begin{align} \label{2.9999A1}
\|V_k(t)\|_{\mathcal{H}}^2=\|v_k(t)\|_{H^1(\mathcal{M})}^2+\|\partial_tv_k(t)\|_{L^2(\mathcal{M})}^2 = \sum_{\nu \in J_k} |\frac{U_\nu}{\lambda_\nu}|^2 \lesssim h_k^2\sum_{\nu \in J_k} |U_\nu|^2= 1.
\end{align}
From \eqref{2.9999A1} and $\partial_tv_k(t)=\sum\limits_{\nu \in J_k} e^{ it\lambda_\nu}U_\nu e_\nu^0=u_k(t)$, we obtain
\begin{align}\label{int1A1}
\|u_k(t)\|_{L^2(\mathcal{M})}^2\lesssim 1, \quad \text{uniformly in } k.
\end{align}
Let $I$ be a bounded time interval. Integrating \eqref{int1A1} over $I$, we conclude that
\begin{align*}
\int_I\|u_k(t,.)\|^2_{L^2(\mathcal{M})}dt\lesssim |I|,
\end{align*}
where $|I|$ is the length of $I$. Hence, $(u_k)_k$ is bounded in $L^2_{\mathrm{loc}}(\mathcal{L})$, uniformly with respect to $k\in \mathbb{N}^*$.\\
Now, using the fact that 
\begin{align}\label{5.888sA1}
h_k^2\mathcal{E}(u_k)= \frac{1}{2}\Big(\|h_k u_k(t)\|_{H^1(\mathcal{M})}^2+\|h_k\partial_tu_k(t)\|_{L^2(\mathcal{M})}^2\Big)\lesssim h_k^2\|U_{k_{|t=0}}\|_{\mathcal{H}}^2=1,
\end{align}
we deduce that 
\begin{align}\label{int2A1}
 \|h_k\partial_tu_k(t)\|_{L^2(\mathcal{M})}^2\lesssim 1, \quad \text{uniformly in } k.
\end{align} Integrating \eqref{int2A1} over the time interval $I$, we conclude that
\begin{align*}
 \int_I\|h_k\partial_t u_k(t,.)\|^2_{L^2(\mathcal{M})}dt\lesssim |I|.
\end{align*}
In addition, \eqref{5.888sA1} gives 
\begin{align}
    \|h_k u_k(t)\|_{H^1(\mathcal{M})}^2\lesssim 1, \quad \text{for any } t \in \mathbb{R},\text{ uniformly in } k.
\end{align}
By Lemma \ref{equiA1} and \eqref{5.888sA1}, one has
\begin{align}
\|h_k^2 AU_k\|_{\mathcal{H}}\simeq \|h_k U_k\|_{\mathcal{H}}\lesssim 1, \quad \text{ uniformly in } k.
\end{align}
This gives
\begin{align}\label{5.111111sA1}
\|h_k^2\partial_tu_k(t,.)\|_{H^1(\mathcal{M})}\lesssim 1 \quad \text{and}\quad \|h_k^2 (-\Delta_g+1) u_k(t,.)\|_{L^2(\mathcal{M})}\lesssim 1,
\end{align}
for all $t \in \mathbb{R}$, uniformly in $k$.
The first inequality in \eqref{5.111111sA1}, together with the boundary condition in \eqref{equofcontraA1}, yields
\begin{align}\label{5.1222222222sA1}
 \|h_k^2\partial_n u_k(t,.)\|_{H^{\frac{1}{2}}(\partial \mathcal{M})}=\|h_k^2\partial_tu_k(t,.)\|_{H^{\frac{1}{2}}(\partial \mathcal{M})} \lesssim 1, \quad\text{ uniformly in } k.
\end{align}
Combining the second inequality in \eqref{5.111111sA1} with \eqref{5.1222222222sA1} and applying Proposition \ref{pro*A1}, we deduce
\begin{align*}
\|h_k^2u_k(t,.)\|_{H^2(\mathcal{M})}\lesssim 1, \quad \text{for any } t \in \mathbb{R},\text{ uniformly in } k.
\end{align*} 
This concludes the proof of Lemma \ref{lem0A1}.
\end{proof}
 From Lemma \ref{lem0A1} one deduces the following consequence.
\begin{lemma} \label{massleakA1}
No mass leaks at infinity at scale $H$ for $(\psi(t) u_k)_k$, for any $\psi \in C^\infty_c(\mathbb{R})$.
\end{lemma}
\begin{proof}[Proof]
By Lemma \ref{lem0A1}, in local coordinates on $\{z\geq 0\}$, the sequences $(\psi(t)u_k)_k$ and $(h_k \psi(t) u_k)_k$ are uniformly bounded with respect to $k\in \mathbb{N}^*$ in $L^2(\mathbb{R}^{d+1}_+)$ and $H^1(\mathbb{R}^{d+1}_+)$ (with respect to the variables $t$ and $x$), respectively. By interpolation, it follows that $(h_k^s \psi(t) u_k)_k$ is bounded in $H^s(\mathbb{R}^{d+1}_+)$ for $s\in [0,\frac{1}{2}[$ and thus by extension $(h_k^s \psi(t) u_k)_k \in H^s(\mathbb{R}^{d+1})$. Then, one writes 
\begin{align}\label{5.133333345A1}
\int_{h_k |\xi|\geq R}|\widehat{\psi u_k}(\xi)|^2 d\xi \leq R^{-2s}\int_{\mathbb{R}^d} h_k^{2s} |\xi|^{2s} |\widehat{\psi u_k}(\xi)|^2 d\xi \lesssim R^{-2s}.
\end{align}
 Letting $R \rightarrow +\infty$, the left-hand side of \eqref{5.133333345A1} tends to zero, which shows that no mass escapes at infinity in the sense of Definition \ref{massA1}. This concludes the proof.
\end{proof}

As proven in Lemma \ref{lem0A1}, the sequence $(u_k)_k$ is bounded in $L^2_{\mathrm{loc}}(\mathcal{L})$. Consequently, by Proposition \ref{pro1.2µµA1}, there exists a semi-classical measure $\mu$ at scale $H=(h_k)_k$ associated with the sequence $(u_k)_k$.
\begin{proposition} \label{boundnessaubordA1}
     The sequences $(u_{k_{\mid \partial \mathcal{L}}})_k$ and $(h_k \partial_n u_{k_{\mid \partial \mathcal{L}}})_k$ are bounded in $L^2_{\mathrm{loc}}(\partial\mathcal{L})$.
\end{proposition}
The proof of Proposition \ref{boundnessaubordA1} is done in section \ref{exisµA1}.  Due to  Propositions \ref{boundnessaubordA1} and \ref{Pro2µµA1}, there exists a Hermitian semi-classical measure on $T^*\partial \mathcal{L}$ associated to the vector $\tilde{U}_k= \,^t(u_{k_{\mid \partial \mathcal{L}}},h_k \partial_n u_{k_{\mid \partial \mathcal{L}}})$ of the form
\begin{align*}
\nu=
\begin{pmatrix}
    \nu^{0,0} & \nu^{0,1} \\
    \nu^{1,0} & \nu^{1,1} \\
\end{pmatrix}.
\end{align*}
\begin{proposition}\label{mesurenonnulA1} Let $\psi\in C^\infty_c(\mathbb{R})$.
One has 
\begin{align*}
\|\psi(t)h_k \nabla_gu_k\|_{L^2V(\mathcal{L})}\simeq \|\psi(t) h_k\partial_tu_k\|_{L^2(\mathcal{L})}, \quad \text{as}\; k \rightarrow +\infty.
\end{align*}
\end{proposition}
\begin{proof}[Proof]
Let $\psi\in C^\infty_c(\mathbb{R})$. Integrating by parts and using Proposition \ref{boundnessaubordA1}, one has
\begin{align} \label{5.13338ççA1}
h_k^2\|\psi(t)\nabla_gu_k\|_{L^2V(\mathcal{L})}^2&=(\psi^2(t)h_k\nabla_gu_k,h_k \nabla_gu_k)_{L^2V(\mathcal{L})}\nonumber\\&=\big(\psi^2(t)h_k^2 (-\Delta_gu_k),u_k\big)_{L^2(\mathcal{L})}+(\psi(t)h_k^2\partial_nu_{k_{\mid \partial \mathcal{L}}},\psi(t)u_{k_{\mid \partial \mathcal{L}}})_{L^2(\partial\mathcal{L})}\nonumber
\\&=\big(\psi^2(t)h_k^2 (-\Delta_gu_k),u_k\big)_{L^2(\mathcal{L})} +O(h_k).
\end{align}
Using rescpectivly the fact that $(u_k)_k$ is bounded in $L^2_{\mathrm{loc}}(\mathcal{L})$, uniformly with respect to $k\in \mathbb{N}^*$, \eqref{5.13338ççA1} and the equation $(\partial_t^2-\Delta_g+1)u_k=0$, we get for $k$ large
\begin{align*}
h_k^2\|\psi(t)\nabla_gu_k\|_{L^2V(\mathcal{L})}^2&\simeq h_k^2\|\psi(t)u_k\|_{L^2(\mathcal{L})}^2+h_k^2\|\psi(t)\nabla_gu_k\|_{L^2V(\mathcal{L})}^2
\\&\simeq h_k^2\|\psi(t)u_k\|_{L^2(\mathcal{L})}^2+\big(\psi^2(t)h_k^2 (-\Delta_gu_k),u_k\big)_{L^2(\mathcal{L})}\\&\simeq h_k^2 \big(\psi^2(t)(1-\Delta_g)u_k,u_k\big)_{L^2(\mathcal{L})}\simeq h_k^2 \big(\psi^2(t)(-\partial_t^2u_k),u_k\big)_{L^2(\mathcal{L})}.
\end{align*}
By integration by parts and Lemma \ref{lem0A1}, one writes
\begin{align*}
h_k^2 \big(\psi^2(t)(-\partial_t^2u_k),u_k\big)_{L^2(\mathcal{L})}&=h_k^2\|\psi(t)\partial_tu_k\|_{L^2(\mathcal{L})}^2 +2h_k\big((\partial_t\psi(t))\psi(t)h_k\partial_tu_k,u_k)_{L^2(\mathcal{L})}\\&=h_k^2\|\psi(t)\partial_tu_k\|_{L^2(\mathcal{L})}^2+O(h_k).
\end{align*}
This concludes the proof of Proposition \ref{mesurenonnulA1}.
\end{proof}
From Propositions \ref{pro*A1} and \ref{mesurenonnulA1} along with Lemmas \ref{lem0A1} and \ref{massleakA1} we deduce the following  properties of the measures $\mu$ and $\nu$.
\begin{proposition}
\label{hhA1}
The three subsequent properties hold.
\begin{enumerate}
\item If $J \subset \mathbb{R}$ is a bounded nonempty open interval, one has $\mu\big(T^*(J\times \mathcal{M})\big)>0$.
\item Let $0<\alpha<1$,
\begin{align}
\label{24*A1}
&\mathrm{supp}(\mu)\subset \mathrm{Char}(p)\cap T^*\mathcal{L}\cap \{\alpha \leq \tau \leq \alpha ^{-1}\},
\\&
\label{25*A1}
\mathrm{supp}(\nu^{i,j})\subset  T^*\partial\mathcal{L}\cap \{\alpha \leq \tau \leq \alpha ^{-1}\} \quad \text{for}\quad i,j \in \{0,1\}.
\end{align}
\item The measure $\mu$ vanishes on $T^*(]\delta,T-\delta[ \times \omega)$.
\end{enumerate}
\end{proposition}
\begin{proof}[Proof]
Let $\psi\in C^\infty_c(\mathbb{R})$. By Proposition \ref{mesurenonnulA1}, Lemma \ref{massleakA1} and Proposition 5.21 of \cite{1A1}, we have
\begin{align*}
1\simeq \underset{k\rightarrow +\infty}{\mathrm{lim}}\|\psi(t)h_kU_k(t)\|_{\mathcal{H}}^2\simeq 2\underset{k\rightarrow +\infty}{\mathrm{lim}}\|\psi(t)h_k\partial_tu_k\|_{L^2(\mathcal{L})}^2\simeq \langle \mu,|\psi(t)|^2\tau^2\rangle.
\end{align*}
Hence the first item in Proposition \ref{hhA1}.
The result \eqref{24*A1} is proven in Proposition 6.4 of \cite{1A1}. It remains to prove \eqref{25*A1}. For that,
consider $\varphi \in C^\infty_c(\mathbb{R})$ such that $\varphi \equiv 1$ in a neighborhood of $[\alpha,\alpha^{-1}]$. For $U_k=(u_k,\partial_tu_k)\in B^+$, we write
\begin{align}
\label{3*A1}
\big(1-\varphi(h_kD_t)\big)\psi(t) U_k=\sum_{\nu \in J_k}\big(1-\varphi(h_kD_t)\big)(\psi(t)e^{ it\lambda_\nu})U_\nu  e_\nu.
\end{align}
For any $\ell^{'}\geq0$ and $\ell^{''}\geq0$, one has
\begin{align*}
A(1-\varphi(h_kD_t))(\partial_t^{\ell'}&\psi)(t)\partial_t^{\ell''}U_k=(1-\varphi(h_kD_t))(\partial_t^{\ell'}\psi)(t) \partial_t^{\ell''}A U_k
\\&=(1-\varphi(h_kD_t))(\partial_t^{\ell'}\psi)(t)\sum_{\nu \in J_k} ( i\lambda_\nu)^{\ell^{''}}\lambda_\nu e^{ it\lambda_\nu}U_\nu  e_\nu.
\end{align*}
Therefore, arguing as in the proof of Proposition 6.4 in \cite{1A1} (see formula (6.11)), for any $N\in \mathbb{N}$, we have due to the fact that $(h_k^2\partial_t^2u_k(t,.))_k$ is bounded in $L^2(\mathcal{M})$ for any $t \in \mathbb{R}$, uniformly with respect to $k\in \mathbb{N}^*$, by Lemma \ref{equiA1}, $h_k^2\partial_n u_k(t,.)$ is bounded in $H^\frac{1}{2}(\partial\mathcal{M})$ for any $t \in \mathbb{R}$, uniformly with respect to $k\in \mathbb{N}^*$ (see formula \eqref{5.1222222222sA1}), and to Proposition \ref{pro*A1} along with the boundary condition and the trace formula
\begin{align*}
\|\partial^\ell_t\big(1-\varphi(h_kD_t)\big)\psi(t)  u_k\|_{L^2(\mathbb{R};H^2(\mathcal{M}))}^2&\lesssim \|\big(1-\varphi(h_kD_t)\big)(\partial_t^{\ell'}\psi)(t)\big(-\Delta_g+1\big) \partial_t^{\ell''} u_k\|_{L^2(\mathcal{L})} ^2
\\&\quad+ \|\big(1-\varphi(h_kD_t)\big)(\partial_t^{\ell'}\psi)(t) \partial_t^{\ell''+1} u_k\|_{L^2(\mathbb{R};H^1(\mathcal{M}))} ^2
\\&
\lesssim \|\big(1-\varphi(h_kD_t)\big)(\partial_t^{\ell'}\psi)(t) \partial_t^{\ell''} AU_k\|_{L^2(\mathbb{R};\mathcal{H})} ^2
\\&
\lesssim \sum_{\nu \in J_k} \lambda_\nu^{2\ell^{''}}\lambda_\nu^2|U_\nu|^2 \|\big(1-\varphi(h_kD_t)\big)(\partial_t^{\ell'}\psi)(t)e^{ it\lambda_\nu}\|_{L^2(\mathbb{R})}^2
\\&\lesssim  h_k^{N-(\ell^{''}+4)}h_k^2\sum_{\nu \in J_k}|u_\nu |^2
\\&\lesssim  h_k^{N-(\ell^{''}+4)},
\end{align*}
for any $\ell \geqslant0$ such that $\ell^{'}+\ell^{''}=\ell$, 
using that $h_k |\lambda_\nu| \simeq 1$ for $\nu \in J_k$ and 
 $h_k^2\sum\limits_{\nu \in J_k}| U_\nu|^2=1$.\\
Thus, with $\ell=0$, one has
\begin{align}
\label{4*A1}
\|\big(1-\varphi(h_kD_t)\big)\psi(t) u_k\|_{L^2(\mathbb{R};H^2(\mathcal{M}))}=O(h_k^N),
\end{align}
implying, by the trace formula,
\begin{align}
\label{4**A1}
\|\big(1-\varphi(h_kD_t)\big)\psi(t)h_k \partial_n u_{k_{| \partial \mathcal{L}}}\|_{L^2(\partial\mathcal{L})}=O(h_k^N),
\end{align}
for any $N\in \mathbb{N}$. Hence, by Lemma \ref{massleakA1}, Proposition 5.21 of \cite{1A1} and the fact that the measure $\nu^{1,1}$ associated to the sequence $(h_k \partial_n u_{k_\mid \partial \mathcal{L}})_k$ is nonnegative, we obtain $\langle \nu^{1,1}, 
|(1-\varphi)(\tau)\psi(t)|^2\rangle=0 $, which leads to the result in \eqref{25*A1} for the measure $\nu^{1,1}$. By the boundary condition  in \eqref{eqimA1}, the same conclusion holds for the measures $\nu^{0,0}$, $\nu^{1,0}$, and $\nu^{0,1}$.\\
For the third item in Proposition \ref{hhA1}, we have by Proposition 5.21 in \cite{1A1}
\begin{align} \label{5.199---A1}
(\mathds{1}_{]\delta,T-\delta[ \times\omega}h_k\partial_tu_k,h_k\partial_tu_k)_{L^2(\mathcal{L})}
\underset{k\rightarrow +\infty}{\rightarrow} \langle \mu, \tau^2 \mathds{1}_{]\delta,T-\delta[ \times\omega}\rangle.
\end{align}
Using the contradiction argument initiated in section \ref{strategyA1}, one has
\begin{align}\label{5.20---A1}
(\mathds{1}_{]\delta,T-\delta[ \times\omega}h_k\partial_tu_k,h_k\partial_tu_k)_{L^2(\mathcal{L})}
\underset{k\rightarrow +\infty}{\rightarrow} 0.
\end{align}
 Hence, with \eqref{24*A1}, we obtain the desired result.
\end{proof}
\subsection{Existence of semi-classical measures near a boundary point}
\label{exisµA1}
In this section, we prove the boundedness properties of the traces of $(u_k)_k$. This is carried out in the hyperbolic, elliptic and glancing regions introduced in Definition \ref{defn11A1}.
Here, we will often work in local coordinates, using a local chart $C=(O,\phi_{\mathcal{L}})$, with $O$ is a neighborhood of a point $\varrho^0 \in \partial \mathcal{L}$. We recall that $\partial\mathcal{L}$ is given by $\{z=0\}$ and $\partial_n=\partial_z$ at $z=0$. For simplicity, we choose to keep the notation $u_k$ instead of  its local representative $u_k^c$ in $C$.

In the sequel we use the following definition. 
\begin{definition}[Boundary operator]
\label{def1A1}
One says that the sequence $(f_k)_k$ is bounded in $L^2_{\varrho_0}(\partial\mathcal{L})$ if there exists a semi-classical operator $\Op^h(e)$ of degree $0$, elliptic at $\varrho_0 \in T^*\partial\mathcal{L}$, with compact support in $\partial\mathcal{L}$, such that the sequence $(\Op^h(e) f_k)_k$ is bounded in $L^2(\partial\mathcal{L})$.
\end{definition}
 In what follows, we state the following results that gives $(h_k^2\partial_z u_{k_{\mid z=0}})_k$ is bounded in \\$L^2_{\mathrm{loc}}(\mathbb{R}, H^{-\frac{1}{2}}(\partial \mathcal{M}))$, uniformly in $k\in \mathbb{N}^*$.
\begin{lemma}
\label{2.1***A1}
Let $u\in H^1(\mathcal{M})$ and $\Delta_g u \in L^2(\mathcal{M})$. There exists  $C>0$
such that 
\begin{align*}
\|[\Delta_g,\rho_\epsilon^*]u\|_{L^2(\mathcal{M})}\leq C \|u\|_{H^1(\mathcal{M})},
\end{align*}
where $\rho_\epsilon^*$ is the convolution operator by $\rho_\epsilon=\epsilon^{-d}\rho(\frac{x}{\epsilon})$ with $\rho\in C^\infty_c(\mathbb{R}^d)$.
Additionally, $[\Delta_g,\rho_\epsilon^*]u$ converges to 0 in $L^2(\mathcal{M})$ as $\epsilon\rightarrow 0$.
\end{lemma}
\begin{proof}[Proof]
Recall that in local coordinates, the Laplace-Beltrami operator on a manifold $\mathcal{M}$ is given by
\begin{align*}
\Delta_gu=(\text{det}\; g)^{-\frac{1}{2}}\sum_{1\leq i,j\leq d}\partial_{x_i}\big( \tilde{g}^{i,j}\partial_{x_j}u\big),
\end{align*}
where $\tilde{g}^{i,j}=(\text{det}\; g)^{\frac{1}{2}}g^{i,j}$.
Let
\begin{align} \label{5.2000000A1}
[\Delta_g,\rho_\epsilon^*]u=\sum_{1\leq i,j\leq d}G^\epsilon_{i,j},
\end{align}
with
\begin{align*}
G^\epsilon_{i,j}&:=(\text{det}\; g)^{\frac{1}{2}}\Big(\rho_\epsilon \ast \partial_{x_i}\big( \tilde{g}^{i,j}\partial_{x_j}u\big)-\partial_{x_i}\big( \tilde{g}^{i,j}\partial_{x_j}(\rho_\epsilon \ast u)\big)\Big).
\end{align*}
Using $\rho_\epsilon:=\epsilon^{-d}\rho(\frac{x}{\epsilon})$ and the fact that $\frac{\partial^2}{\partial_{y_i}\partial_{y_j}}\rho(\frac{x-y}{\epsilon})=\frac{\partial^2}{\partial_{x_i}\partial_{x_j}}\rho(\frac{x-y}{\epsilon})$, we get 
\begin{align*}
G^\epsilon_{i,j}&=\frac{1}{\epsilon^d}\int_{\mathcal{M}}\rho(\frac{x-y}{\epsilon})\partial_{y_i}\big( \tilde{g}^{i,j}\partial_{y_j}u\big)(y)dy
\\&\quad-\partial_{x_i}(\tilde{g}^{i,j})\frac{1}{\epsilon^d}\int_{\mathcal{M}} \partial_{x_j}\rho(\frac{x-y}{\epsilon})u(y)dy
-\tilde{g}^{i,j}(x)\frac{1}{\epsilon^d}\int_{\mathcal{M}} \frac{\partial^2}{\partial_{y_i}\partial_{y_j}}\rho(\frac{x-y}{\epsilon})u(y)dy.
\end{align*}
By integartion by parts, we obtain
\begin{align*}
G^\epsilon_{i,j}&=-\frac{1}{\epsilon^d}\int_{\mathcal{M}}\partial_{y_i}\rho(\frac{x-y}{\epsilon}) \tilde{g}^{i,j}(y)\partial_{y_j}u(y)dy\\&\quad\,-\partial_{x_i}(\tilde{g}^{i,j})\frac{1}{\epsilon^d}\int_{\mathcal{M}} \partial_{x_j}\rho(\frac{x-y}{\epsilon})u(y)dy
+\tilde{g}^{i,j}(x)\frac{1}{\epsilon^d}\int_{\mathcal{M}} \partial_{y_j}\rho(\frac{x-y}{\epsilon})\partial_{y_j}u(y)dy
\\&
=\frac{1}{\epsilon^d}\int_{\mathcal{M}}\partial_{y_i}\rho(\frac{x-y}{\epsilon})\big( \tilde{g}^{i,j}(x)-\tilde{g}^{i,j}(y)\big)\partial_{y_j}u(y)dy
-\partial_{x_i}(\tilde{g}^{i,j})\frac{1}{\epsilon^d}\int_{\mathcal{M}} \partial_{x_j}\rho(\frac{x-y}{\epsilon})u(y)dy.
\end{align*}
Then
\begin{align*}
\|[\Delta_g,\rho_\epsilon^*]u\|_{L^2(\mathcal{M})}\lesssim \underset{1\leq i,j\leq d}{\mathrm{sup}}\|\nabla \tilde{g}^{i,j}\|_{L^\infty}  \|u\|_{H^1(\mathcal{M})}\lesssim \|u\|_{H^1(\mathcal{M})},
\end{align*}
where in the last estimate we use $\tilde{g}^{i,j}\in C^1\cap W^{1,\infty}$.
Therefore, since $\partial_{x_i}\big( \tilde{g}^{i,j}\partial_{x_j}u\big) \in L^2(\mathcal{M}) $, one has
\begin{align*}
    \rho_\epsilon \ast \partial_{x_i}\big( \tilde{g}^{i,j}\partial_{x_j}u\big)\rightarrow \partial_{x_i}\big( \tilde{g}^{i,j}\partial_{x_j}u\big)  \quad \text{as} \quad \epsilon \rightarrow 0 \quad \text{in} \quad L^2(\mathcal{M}),
\end{align*}
and 
\begin{align*}
   \partial_{x_i}\big( \tilde{g}^{i,j}\partial_{x_j}(\rho_\epsilon \ast u)\big) \rightarrow  \partial_{x_i}\big( \tilde{g}^{i,j}\partial_{x_j}u\big) \quad \text{as} \quad \epsilon \rightarrow 0 \quad \text{in} \quad L^2(\mathcal{M}).
\end{align*}
Thus, from \eqref{5.2000000A1} we deduce that $[\Delta_g,\rho_\epsilon^*]u$ converges to 0 in $L^2(\mathcal{M})$ as $\epsilon$ goes to 0.
\end{proof}
With Lemma \ref{2.1***A1} one obtains the following proposition.
\begin{proposition}
\label{lem3*A1}
If  $v \in H^1(\mathcal{M})$ and $\Delta_g v \in L^2(\mathcal{M})$, then $\partial_nv \in H^{-\frac{1}{2}}(\partial\mathcal{M})$ and 
\begin{align*}
\|\partial_nv\|_{H^{-\frac{1}{2}}(\partial\mathcal{M})}\lesssim \big(\|\nabla_gv\|_{L^2V(\mathcal{M})}+\|\Delta_gv\|_{L^2(\mathcal{M})}\big).
\end{align*}
\end{proposition}
\begin{proof}[Proof]
It suffices to prove the result locally near the boundary. We recall that in local coordinates $\mathcal{M}=\{z\geq0\}$ and  $\partial\mathcal{M}=\{z=x_d=0\}$.\\
Consider the extension $v^0$ of $v$ by 0 for $z<0$. Let $(\rho_n)_n=(n^{-d}\rho(\frac{x}{n}))_n$ be a regularizing sequence with support included in $\{z\leq 0\}$. Denote $w_n=(v^0\ast \rho_n)_{\mid \mathcal{M}}$, then $(\nabla_gw_n)_n$ converges to $\nabla_g v$ in $L^2V(\mathcal{M})$.\\
One has 
\begin{align*}
\Delta_gw_n=(\rho_n \ast \Delta_gv^0)_{\mid \mathcal{M}}+([\Delta_g,\rho_n^\ast]v^0)_{\mid \mathcal{M}},
\end{align*}
where $\rho_n^\ast$ is the convolution operator by $\rho_n$. Since $\Delta_gv^0-(\Delta_gv)^0$ is supported in $\{z=0\}$, then  $(\rho_n \ast \Delta_gv^0)_{\mid \mathcal{M}}=(\rho_n \ast (\Delta_gv)^0)_{\mid \mathcal{M}}$ by the support theorem (see, e.g., Theorem 4.2.4 in \cite{15A1}) which converges to $\Delta_gv$ in $L^2(\mathcal{M})$. Moreover, by  Lemma \ref{2.1***A1}, $[\Delta_g,\rho_n^\ast]v^0$ converges to 0 in $L^2(\mathcal{M})$. Thus $(\Delta_gw_n)_n$ converges to $\Delta_gv$ in $L^2(\mathcal{M})$.
\\
Consider $\widetilde{\Theta}\in H^{\frac{1}{2}}(\partial\mathcal{M})$ and $\Theta$ solution of
 \begin{equation*}
\label{10*A1}
    \begin{cases}
        -\triangle_g \Theta &= 0  \\ 
        \Theta_{\mid \mathcal{\partial M}}&= \widetilde{\Theta},
    \end{cases}
\end{equation*} 
which satisfies the following inequality:
\begin{align}
\label{14**A1}
\|\Theta\|_{H^1(\mathcal{M})}\lesssim \|\widetilde{\Theta}\|_{H^{\frac{1}{2}}(\partial\mathcal{M})}.
\end{align}
By integration by parts, we write
\begin{align}
\label{13**A1}
(-\Delta_gw_n,\Theta)_{L^2(\mathcal{M})}=(\nabla_gw_n,\nabla_g\Theta)_{L^2V(\mathcal{M})}+(\partial_nw_n{_{\mid\partial\mathcal{M}}},\Theta_{\mid\partial\mathcal{M}})_{L^2(\partial \mathcal{M})}.
\end{align}
Using \eqref{13**A1} and \eqref{14**A1}, one finds
\begin{align*}
|(\partial_nw_n{_{\mid\partial\mathcal{M}}},\widetilde{\Theta})_{L^2(\partial \mathcal{M})}|&\lesssim \|\nabla_gw_n\|_{L^2V(\mathcal{M})}\|\nabla_g\Theta\|_{L^2V(\mathcal{M})}+\|\Delta_gw_n\|_{L^2(\mathcal{M})}\|\Theta\|_{L^2(\mathcal{M})}
\\&
\lesssim \big(\|\nabla_gw_n\|_{L^2V(\mathcal{M})}+\|\Delta_gw_n\|_{L^2(\mathcal{M})}\big)\|\Theta\|_{H^1(\mathcal{M})}
\\&
\lesssim \big(\|\nabla_gw_n\|_{L^2V(\mathcal{M})}+\|\Delta_gw_n\|_{L^2(\mathcal{M})}\big)\|\widetilde{\Theta}\|_{H^\frac{1}{2}(\partial\mathcal{M})}.
\end{align*}
Hence, we get
\begin{align}
\|\partial_nw_n\|_{H^{-\frac{1}{2}}(\partial\mathcal{M})}\lesssim \big(\|\nabla_gw_n\|_{L^2V(\mathcal{M})}+\|\Delta_gw_n\|_{L^2(\mathcal{M})}\big).
\end{align}
By passing to the limits, if  $v \in H^1(\mathcal{M})$ and $\Delta_g v \in L^2(\mathcal{M})$ then $\partial_nv$ has a well-defined meaning in  $H^{-\frac{1}{2}}(\partial\mathcal{M})$ and 
\begin{align*}
\|\partial_nv\|_{H^{-\frac{1}{2}}(\partial\mathcal{M})}\lesssim \big(\|\nabla_gv\|_{L^2V(\mathcal{M})}+\|\Delta_gv\|_{L^2(\mathcal{M})}\big).
\end{align*}
\end{proof}
\begin{remark}
Note that in the proof of the Proposition \ref{lem3*A1}, there is no need of boundary conditions.
\end{remark}
Applying Proposition \ref{lem3*A1} along with Lemma \ref{lem0A1}, one concludes that $(h_k^2\partial_z u_{k_{\mid z=0}})_k$ is bounded in $L^2_{\mathrm{loc}}(\mathbb{R}, H^{-\frac{1}{2}}(\partial \mathcal{M}))$, uniformly in $k\in \mathbb{N}^*$. This result is not sufficient to prove the existence of the measure $\nu$, for which we need the sequences $(u_{k_{\mid z=0}})_k$ and $(h_k\partial_zu_{k_{\mid z=0}})_k$ to be bounded in $L^2_{\mathrm{loc}}(\partial\mathcal{L})$. This will be established in the following three sections.
\subsubsection{Glancing region}
The properties of the traces of $(u_k)_k$ in the glancing region ${^\parallel\mathcal{G_\partial}}$ are stated in Theorem \ref{lopA1} and Corollary \ref{cor1**A1} bellow.
\begin{theorem}
\label{lopA1}
The sequence $( u_{k_{\mid z=0}})_k$ is bounded in $L^2_{\varrho_0}(\partial\mathcal{L})$ for $\varrho_0\in {^\parallel\mathcal{G_\partial}}$.
\end{theorem}
A consequence of Theorem \ref{lopA1}, is the following result.
\begin{corollary}
\label{cor1**A1}
The sequence $(h_k\partial_zu_{k_{\mid z=0}})_k$ is bounded in $L^2_{\varrho_0}(\partial \mathcal{L})$ for $\varrho_0\in {^\parallel\mathcal{G_\partial}}$.
\end{corollary}
\begin{proof}[Proof of Corollary \ref{cor1**A1}]
Let $e(t,x,\tau,\xi') \in C^\infty_c(\phi_{\mathcal{L}}(O)\times\mathbb{R}^{d})$ be elliptic at $\varrho_0$ with $e(\varrho_0) =1$. Let $\psi\in C^\infty_c(\phi_{\mathcal{L}}(O))$ with $\psi$ equal to $1$ in a neighborhood of the $(t,x)$-projection of $\mathrm{supp}(e)$. Using the boundary condition  $h_k\partial_zu_{k_{\mid z=0}}=\Op^h(\tau)u_{k_{\mid z=0}}$ and Theorem \ref{lopA1}, it follows that 
\begin{align*}
\|\Op^h(e)\psi h_k\partial_zu_{k_{\mid z=0}}\|_{L^2(\partial\mathcal{L})}\leq\|\Op^h(e\tau)\psi u_{k_{\mid z=0}}\|_{L^2(\partial\mathcal{L})}+\|\Op^h(e)h_k[\psi,\partial_t] u_{k_{\mid z=0}}\|_{L^2(\partial\mathcal{L})}\lesssim 1.
\end{align*}
By Definition \ref{def1A1}, this implies that the sequence $(h_k\partial_zu_{k_{\mid z=0}})_k$ is bounded in $L^2_{\varrho_0}(\partial \mathcal{L})$, which gives the result. 
\end{proof}
\begin{proof}[Proof of Theorem \ref{lopA1}]
Let $q(t,x,\tau,\xi')\in C^\infty_c(\phi_{\mathcal{L}}(O)\times\mathbb{R}^{d})$ elliptic near $\varrho_0$. Consider $\psi\in C^\infty_c(\phi_{\mathcal{L}}(O))$ with $\psi$ equal to $1$ in a neighborhood of the $(t,x)$-projection of $\mathrm{supp}(q)$. In local coordinates, one has
\begin{align}\label{exppA1}
Pu_k&=-g^{d,d}(x)\partial_{z}^2u_k-R(t,x,D)u_k,
\end{align}
where \begin{align} \label{exp of RA1}
R(t,x,D)&:= -\partial_t^2+(\text{det}\; g)^{-\frac{1}{2}}\sum_{1\leq i\leq d-1}\partial_{x_i}\Big((\text{det}\; g)^{\frac{1}{2}}g^{i,d}(x)\partial_{z}\Big)  \nonumber
\\&\quad
+(\text{det}\; g)^{-\frac{1}{2}}\partial_{z}\Big((\text{det}\; g)^{\frac{1}{2}}g^{d,d}(x)\Big)\partial_{z}
+(\text{det}\; g)^{-\frac{1}{2}}\sum_{1\leq j\leq d-1}\partial_{z}\Big((\text{det}\; g)^{\frac{1}{2}}g^{d,j}(x)\partial_{x_j}\Big)\nonumber
\\&\quad
+\underset{1 \leq i,j \leq d-1}{\sum}(\text{det}\; g)^{-\frac{1}{2}}\partial_{x_i}\Big((\text{det}\; g)^{\frac{1}{2}}g^{i,j}(x)\partial_{x_j}\Big)-1.\end{align}
Using \eqref{exppA1}, integrating by parts and taking into account the conditions $g^{d,d}(x)=1$ and $g^{j,d}(x)=g^{d,j}(x)=0 $ at $z=0$ for $j\neq d$ in the chosen local coordinates, we obtain
\begin{align*}
&\big(h_k^2Pu_k,\Op^h(q)\psi\partial_zu_k\big)_{L^2(z\geq 0)}+\big(\partial_z(\psi\Op^h(q)^*u_k),h_k^2Pu_k\big)_{L^2(z\geq0)}
\\&\quad=\big(u_k,h_k^2[P,\Op^h(q)\psi\partial_z]u_k\big)_{L^2(z\geq 0)}
+\big(h_k \partial_zu_{k_{\mid z=0}},\Op^h(q)\psi h_k\partial_zu_{k_{\mid z=0}}\big)_{L^2(\partial \mathcal{L})}
\\&\quad \quad-\big(h_k^2u_{k_{\mid z=0}},\partial_z\big(\Op^h(q)\psi\partial_zu_k\big)_{\mid z=0}\big)_{L^2(\partial \mathcal{L})}
-\big(\psi\Op^h(q)^*u_{k_{\mid z=0}},h_k^2(Pu_k)_{\mid z=0}\big)_{L^2(\partial \mathcal{L})}.
\end{align*}
Therefore, due to $Pu_k=0$, we  get
\begin{align}
\label{2A1}
\big(h_k^2u_{k_{\mid z=0}},\partial_z\big(\Op^h(q)\psi\partial_zu_k\big)_{\mid z=0}&\big)_{L^2(\partial \mathcal{L})} \nonumber
-\big(h_k^2 \partial_zu_{k_{\mid z=0}},\Op^h(q)\psi\partial_zu_{k_{\mid z=0}}\big)_{L^2(\partial \mathcal{L})} 
\\&
=\big(u_k,h_k^{-1}[h_k^2P,\Op^h(q)\psi h_k\partial_z]u_k\big)_{L^2(z\geq 0)}.
\end{align}
Then, one writes
\begin{align}\label{5.26SSA1}
h_k^{-1}[h_k^2P,\Op^h(q)\psi h_k\partial_z]&=h_k^{-1}[h_k^2P,\Op^h(q)h_k\partial_z\psi]+ h_k^{-1}[h_k^2P,\Op^h(q)h_k[\psi,\partial_z]]\nonumber
\\&= h_k^{-1}[h_k^2P,\Op^h(q)h_k\partial_z\psi] \nonumber -h_k^{-1}[h_k^2P,\Op^h(q)h_k(\partial_z\psi)]
\\&=h_k^{-1}[h_k^2P,\Op^h(q)h_k\partial_z\psi] -[h_k^2P,\Op^h(q)](\partial_z\psi) \nonumber
\\&\quad-h_k^2\Op^h(q)[P,(\partial_z\psi)]\nonumber
\\&
=h_k^{-1}[h_k^2P,\Op^h(q)h_k\partial_z\psi] -[h_k^2P,\Op^h(q)] (\partial_z\psi)+ h_k^2O(1)_{\mathcal{L}(H^1,L^2)},
\end{align}
where in the last line we use the fact that $[P,(\partial_z\psi)]$ is a differential operator of order one with continuous coefficients.
By symbolic calculus, we obtain a sum of this form
\begin{align}\label{5.28dA1}
h_k^{-1}[h_k^2P,\Op^h(q)h_k\partial_z\psi]-[h_k^2P,\Op^h(q)] (\partial_z\psi)&=\Op^h(q_0)\psi_0+ \Op^h(q_1)\psi_1h_kD_z  \nonumber
\\&\quad +\Op^h(q_2)\psi_2 h_k^2D_z^2,
\end{align}
where $q_0$, $q_1$ and $q_2$ are continuous functions with compact support in the $x$-variable and smooth with compact support in the variables $(t,\tau,\xi')$  and where $\psi_0$, $\psi_1$ and $\psi_2$ are smooth functions such that $\psi_i=1$ on  $\mathrm{supp}(\psi)$ for $i=0,1,2$.\\
Inserting \eqref{5.28dA1} in \eqref{5.26SSA1} then combining the result with \eqref{2A1} and using Lemma \ref{lem0A1}, we get
\begin{align*}
|\big(u_k,h_k^{-1}[h_k^2P,\Op^h(q)\psi h_k\partial_z] u_k\big)_{L^2(z\geq 0)}|\leq C(q).
\end{align*}
Hence, with \eqref{2A1}
\begin{align}
\label{6*A1}
|\big(h_k^2u_{k_{\mid z=0}},\partial_z\big(\Op^h(q)\psi \partial_zu_k\big)_{\mid z=0}\big)_{L^2(\partial \mathcal{L})} 
-\big(h_k^2 \partial_zu_{k_{\mid z=0}},\Op^h(q)\psi\partial_zu_{k_{\mid z=0}}\big)_{L^2(\partial \mathcal{L})}|\leq C(q).
\end{align}
One writes
\begin{align*}
h_k^2\partial_z(\Op^h(q)\psi\partial_zu_k)_{\mid z=0}=-h_k^2\Op^h(q)\psi R(t,x,D)u_{k_{\mid z=0}}+h_k^2g_k,
\end{align*}
with
\begin{align*}
h_k^2g_k&=[\partial_z,\Op^h(q)\psi]h_k^2\partial_zu_{k_{\mid z=0}}-h_k^2\Op^h(q)\psi(Pu_k)_{\mid z=0}\nonumber
\\&=[\partial_z,\Op^h(q)]\psi h_k^2\partial_zu_{k_{\mid z=0}}+\Op^h(q)[\partial_z,\psi]h_k^2\partial_zu_{k_{\mid z=0}},
\end{align*}
where in the last equation we use $(Pu_k)_{\mid z=0}=0$. Using the following semi-classical traces inequalities
\begin{align*}
 \|h_k^\frac{3}{2}\partial_zu_{k_{\mid z=0}}\|_{L^2(\partial\mathcal{L})}&\lesssim h_k \|\partial_zu_k\|_{H^1_{sc}(\mathbb{R}^{d+1}_+)}\\&\lesssim  \|h_k\partial_z u_k\|_{L^2(\mathbb{R}^{d+1}_+)}+ \|h_k^2\nabla  \partial_z u_k\|_{L^2(\mathbb{R}^{d+1}_+)}+\|h_k^2\partial_t\partial_z u_k\|_{L^2(\mathbb{R}^{d+1}_+)},
\end{align*}
and
\begin{align}\label{5.3111111A1}
h_k^\frac{1}{2}\|u_{k_{\mid z=0}}\|_{L^2(\partial\mathcal{L})}&\lesssim \|u_k\|_{H^1_{sc}(\mathbb{R}^{d+1}_+)} \nonumber
\\&\lesssim \|u_k\|_{L^2(\mathbb{R}^{d+1}_+)}+ \|h_k\nabla u_k\|_{L^2(\mathbb{R}^{d+1}_+)}+\|h_k\partial_tu_k\|_{L^2(\mathbb{R}^{d+1}_+)},
\end{align}
together with Lemma \ref{lem0A1} and the first inequality in \eqref{5.111111sA1}, we deduce that the sequences $(h_k^\frac{3}{2}g_k)_k$ and $(h_k^\frac{1}{2}u_{k_{\mid z=0}})_k$ are bounded in $L^2_{\mathrm{loc}}(\partial\mathcal{L})$. Then, it follows that 
\begin{align*}
\big|\big(u_{k_{\mid z=0}},h_k^2g_k\big)_{L^2(\partial\mathcal{L})}\big|=\big|\big(h_k^\frac{1}{2}u_{k_{\mid z=0}},h_k^\frac{3}{2}g_k\big)_{L^2(\partial\mathcal{L})}\big|\leq C'(q).
\end{align*}
Thus with \eqref{6*A1}, one obtains
\begin{align}
\label{3iA1}
\big|\big(h_k^2u_{k_{\mid z=0}}&,\Op^h(q)\psi R(t,x,D)u_{k_{\mid z=0}}\big)_{L^2(\partial \mathcal{L})} \nonumber
\\&+\big(h_k \partial_zu_{k_{\mid z=0}},\Op^h(q)\psi h_k\partial_zu_{k_{\mid z=0}}\big)_{L^2(\partial \mathcal{L})}\big| \leq C^{''}(q).
\end{align}
Using the boundary condition $h_k\partial_zu_{k_{\mid z=0}}=\Op^h(\tau)u_{k_{\mid z=0}}$, we obtain 
\begin{align}
\big(h_k^2&u_{k_{\mid z=0}},\Op^h(q)\psi R(t,x,D)u_{k_{\mid z=0}}\big)_{L^2(\partial \mathcal{L})} 
+\big( h_k\partial_zu_{k_{\mid z=0}},\Op^h(q)\psi h_k\partial_zu_{k_{\mid z=0}}\big)_{L^2(\partial \mathcal{L})} \nonumber
\\&=\big(u_{k_{\mid z=0}}, \big(\Op^h(q)\psi h_k^2 R(t,x,D)+\Op^h(\tau)^*\Op^h(q)\psi \Op^h(\tau)\big)u_{k_{\mid z=0}}\big)_{L^2(\partial \mathcal{L})}\nonumber
\\&=\Big(u_{k_{\mid z=0}},\Op^h(q)\psi \big(h_k^2 R(t,x,D)+\Op^h(\tau^2)\big)u_{k_{\mid z=0}}\Big)_{L^2(\partial \mathcal{L})}\nonumber
\\&\quad+i \Big(h_k^\frac{1}{2}u_{k_{\mid z=0}},[\partial_t,\Op^h(q)\psi \Op^h(\tau)]h_k^\frac{1}{2}u_{k_{\mid z=0}}\Big)_{L^2(\partial \mathcal{L})}\nonumber
\\&
=\Big(u_{k_{\mid z=0}},\Op^h(q)\psi \big(h_k^2 R(t,x,D)+\Op^h(\tau^2)\big)u_{k_{\mid z=0}}\Big)_{L^2(\partial \mathcal{L})}+ O(1), \label{11A1}
\end{align}
where in \eqref{11A1} we use \eqref{5.3111111A1} along with Lemma \ref{lem0A1}.\\
Replacing $\Op^h(q)\psi $ by $\psi \Op^h(q)^*\Op^h(q)\psi$ in \eqref{11A1}, we deduce according to \eqref{3iA1}
\begin{align*}
\Big|\Big(\Op^h(q)\psi u_{k_{\mid z=0}},\Op^h(q)\psi\big(h_k^2 R(t,x,D)+\Op^h(\tau^2)\big)u_{k_{\mid z=0}}\Big)_{L^2(\partial \mathcal{L})}\Big|\leq C''(q).
\end{align*}
Writing
\begin{align}
&\Big(\Op^h(q)\psi u_{k_{\mid z=0}},\Op^h(q)\psi \big(h_k^2 R(t,x,D)+\Op^h(\tau^2)\big)u_{k_{\mid z=0}}\Big)_{L^2(\partial \mathcal{L})}\nonumber
\\&
=\Big(\Op^h(q)\psi u_{k_{\mid z=0}},\big(h_k^2 R(t,x,D)+\Op^h(\tau^2)\big)\Op^h(q)\psi u_{k_{\mid z=0}}\Big)_{L^2(\partial \mathcal{L})}\nonumber
\\&\quad
+\Big(\Op^h(q)\psi u_{k_{\mid z=0}},[\Op^h(q)\psi ,h_k^2 R(t,x,D)+\Op^h(\tau^2)]u_{k_{\mid z=0}}\Big)_{L^2(\partial \mathcal{L})}\nonumber
\\&=\Big(\Op^h(q)\psi u_{k_{\mid z=0}},\big(h_k^2 R(t,x,D)+\Op^h(\tau^2)\big)\Op^h(q)\psi u_{k_{\mid z=0}}\Big)_{L^2(\partial \mathcal{L})}+ O(1),\label{22A1}
\end{align}
where in the last line we use \eqref{5.3111111A1} and Lemma \ref{lem0A1},
one gets
\begin{align}\label{5.34444A1}
\Big|\Big(\Op^h(q)\psi u_{k_{\mid z=0}},\big(h_k^2 R(t,x,D)+\Op^h(\tau^2)\big)\Op^h(q)\psi u_{k_{\mid z=0}}\Big)_{L^2(\partial \mathcal{L})}\Big|\leq C^{''}(q).
\end{align}
Recall that
\begin{align} \label{5.35555sssA1}
h_k^2R(t,x,D)=\underset{(i,j)\neq (d,d)}{\underset{1\leq i,j\leq d}{\sum}}g^{i,j}(x)h_k^2\partial_{x_i} \partial_{x_j}-h_k^2\partial_t^2+\sum_{1\leq j\leq d} h_k^2r_0^j(x)\partial_{x_j}-h_k^2,
\end{align}
where $r_0^j(x)=(\text{det}\; g)^{-\frac{1}{2}}\sum\limits_{1\leq i\leq d}\partial_{x_i}\Big((\text{det}\; g)^{\frac{1}{2}}g^{i,j}(x)\Big)$ is a continuous function. Then, by \eqref{5.3111111A1} and Lemma \ref{lem0A1}, we get
\begin{align}\label{5.35555A1}
\Big(\Op^h(q)\psi u_{k_{\mid z=0}},\big(\sum_{1\leq j\leq d}h_k^2r_0^j(x)\partial_{x_j}-h_k^2\big)\Op^h(q)\psi u_{k_{\mid z=0}}\Big)_{L^2(\partial \mathcal{L})}=O(1).
\end{align}
From \eqref{5.34444A1}, \eqref{5.35555sssA1} and \eqref{5.35555A1}, we obtain
\begin{align}\label{5.3777777A1}
\Big|\Big(\Op^h(q)\psi u_{k_{\mid z=0}},\Op^h\big(r(x,\tau,\xi)\big)\Op^h(q)\psi u_{k_{\mid z=0}}\Big)_{L^2(\partial \mathcal{L})}\Big|\leq C^{'''}(q),
\end{align}
where $r(x,\tau,\xi)=2\tau^2-\underset{(i,j)\neq (d,d)}{\underset{1\leq i,j\leq d}{\sum}}g^{i,j}(x) \xi_i \xi_j$ is a $C^1$ function on the $x$-variable. Applying Taylor’s formula to the function $r$ at a point $x^0$, we obtain
\begin{align}\label{taylA1}
r(x,\tau,\xi)=r(x^0,\tau,\xi)+\sum\limits_{1\leq i \leq d}(x_i-x_i^0)\tilde{r}_i(x,\tau,\xi),
\end{align}
where $\tilde{r}_i$ is a continuous function on $x$.
Using \eqref{taylA1}, we write
\begin{align} \label{G_1+G_2A1}
(\Op^h(q)\psi u_{k_{\mid z=0}}&, \Op^h(r(x,\tau,\xi))\Op^h(q)\psi u_{k_{\mid z=0}})_{L^2(\partial \mathcal{L})}=G_1+G_2,
\end{align}
with 
\begin{align*}
G_1=(\Op^h(q)\psi u_{k_{\mid z=0}}&, \Op^h(r(x^0,\tau,\xi))\Op^h(q)\psi u_{k_{\mid z=0}})_{L^2(\partial \mathcal{L})},
\end{align*}
and 
\begin{align*}
G_2=(\Op^h(q)\psi u_{k_{\mid z=0}}&, \sum\limits_{1\leq i \leq d}\Op^h\big((x_i-x_i^0)\tilde{r}_i(x,\tau,\xi)\big)\Op^h(q)\psi u_{k_{\mid z=0}})_{L^2(\partial \mathcal{L})}.
\end{align*}
Noting that there exist $c_0,c_1>0$ such that $$|r(x^0,\tau,\xi)|\geq c_0(\tau^2+|\xi|^2)\quad \text{for}\quad |(\tau,\xi)|\geq c_1,$$
 $q$ is compactly supported in the $(\tau,\xi')$-variables, it follows from the semi-classical microlocal Gårding inequality (see Theorem 2.29 in \cite{20A1}) that there exists a constant $C_0>0$ such that
\begin{align}  \label{G_1A1}
|G_1|\geq C_0 \|\Op^h(q)\psi u_{k_{\mid z=0}}\|^2_{L^2(\partial \mathcal{L})}.
\end{align}
For the second term, consider $\chi \in C^\infty_c(\mathbb{R}^{d+1})$ with $\chi(\tau,\xi)=1$ in the $(\tau,\xi')$-projection of $\mathrm{supp}(q)$  such that
\begin{align}\label{G_2 sA1}
G_2=G_2'+G_2'',
\end{align}
with 
\begin{align*}
G_2'=(\Op^h(q)\psi u_{k_{\mid z=0}}, \sum\limits_{1\leq i \leq d}\Op^h\big((x_i-x_i^0)\chi(\tau,\xi)\tilde{r}_i(x,\tau,\xi)\big)\Op^h(q)\psi u_{k_{\mid z=0}})_{L^2(\partial \mathcal{L})},
\end{align*}
and 
\begin{align*}
G_2''=(\Op^h(q)\psi u_{k_{\mid z=0}}, \sum\limits_{1\leq i \leq d}\Op^h\big((x_i-x_i^0)\big(1-\chi\big)\tilde{r}_i(x,\tau,\xi)\big)\Op^h(q)\psi u_{k_{\mid z=0}})_{L^2(\partial \mathcal{L})}.
\end{align*}
Now, let us consider the first term $G_2'$. By Cauchy Schwarz inequality, one obtains
\begin{align*}
|G_2'|&\leq \sum\limits_{1\leq i \leq d} \Big(\underset{ \mathrm{supp}(q)}{\mathrm{sup}}|x_i-x_i^0| \;\|\Op^h(q)\psi u_{k_{\mid z=0}}\|_{L^2(\partial \mathcal{L})} 
\\& \quad\quad\quad\quad\quad
\times\| \Op^h\big(\chi(\tau,\xi)\tilde{r}_i(x,\tau,\xi)\big)\Op^h(q)\psi u_{k_{\mid z=0}}\|_{L^2(\partial \mathcal{L})} \Big)
\\&
\leq C_1\sum\limits_{1\leq i \leq d} \underset{ \mathrm{supp}(q)}{\mathrm{sup}}|x_i-x_i^0| \|\Op^h(q)\psi u_{k_{\mid z=0}}\|_{L^2(\partial \mathcal{L})}^2.
\end{align*}
Taking $\mathrm{supp}(q)$ sufficiently small so that
\begin{align*}
C_1\sum\limits_{1\leq i \leq d} \underset{ \mathrm{supp}(q)}{\mathrm{sup}}|x_i-x_i^0| \leq \frac{C_0}{2},
\end{align*}
one gets
\begin{align}  \label{G_2A1}
|G_2'|\leq \frac{C_0}{2} \|\Op^h(q)\psi u_{k_{\mid z=0}}\|^2_{L^2(\partial \mathcal{L})}.
\end{align}
For $G_2''$,  using  \eqref{5.3111111A1} with Lemma \ref{lem0A1} and the fact that $\chi(\tau,\xi)=1$ in the $(\tau,\xi')$-projection of $\mathrm{supp}(q)$, we obtain
\begin{align}\label{G_2''A1}
G_2''= O(h_k^\infty)C(q).
\end{align}
Thus, combining \eqref{G_2 sA1}, \eqref{G_2A1} and \eqref{G_2''A1}, we deduce
\begin{align} \label{G_22A1}
|G_2|\leq \frac{C_0}{2} \|\Op^h(q)\psi u_{k_{\mid z=0}}\|^2_{L^2(\partial \mathcal{L})}+O(h_k^\infty)C(q).
\end{align}
Collecting  \eqref{5.3777777A1}, \eqref{G_1+G_2A1}, \eqref{G_1A1} and \eqref{G_22A1}, we conclude
\begin{align*}
\|\Op^h(q)\psi u_{k_{\mid z=0}}\|^2_{L^2(\partial \mathcal{L})}\leq \frac{2}{C_0}C^{'''}(q)+O(h_k^\infty)\frac{2C(q)}{C_0}.
\end{align*}
Hence, by Definition \ref{def1A1}, $( u_{k_{\mid z=0}})_k$  is bounded in $L^2_{\varrho_0}(\partial \mathcal{L})$.
\end{proof}

\subsubsection{Elliptic region}
The properties of the traces of $(u_k)_k$ in the elliptic  region ${^\parallel\mathcal{E_\partial}}$ are presented in Theorem \ref{ellipA1} and Corollary \ref{corllellipA1}. To prove Theorem \ref{ellipA1} we need the following elementary
results. 
\begin{lemma} \label{lem1Th}
   Let $a(x',z,\xi',\zeta)$ be compactly supported in the $(x',z,\xi')$-variables, $C^0(\mathbb{R}^d)$ in $x=(x',z)$, and $C^\infty(\mathbb{R}^{d})$ in $\xi=(\xi',\zeta)$ with decay $\langle \zeta\rangle ^{-2}$. For any bounded function $u$ in $L^2(\mathbb{R}^d)$, one has
   \begin{align*}
       \|(\Op^h(a)u)_{|z=0}\|_{L^2(\mathbb{R}^{d-1})} \lesssim h^{-\frac{1}{2}}.
   \end{align*}
\end{lemma}
\begin{lemma} \label{lem2Th}
  Let $a(x',z,\xi',\zeta)$ be compactly supported in the $(x',z,\xi')$-variables, $C^0(\mathbb{R}^d)$ in $x=(x',z)$, and $C^\infty(\mathbb{R}^{d})$ in $\xi=(\xi',\zeta)$ with decay $\langle \zeta\rangle ^{-2}$. Let $\theta \in C^1(\mathbb{R}^d)$ and set
   \begin{align*}
      \mathcal{K}:u \mapsto ([\Op^h(a),\theta]u)_{|z=0}. 
   \end{align*}
  Then, $$\|\mathcal{K}\|_{L^2(\mathbb{R}^d)\rightarrow L^2(\mathbb{R}^{d-1})} \lesssim h^\frac{1}{2}.$$
\end{lemma}
\begin{lemma} \label{lem3Th}
    Let $a(x',z,\xi',\zeta)$ be compactly supported in the $(x',z,\xi')$-variables, $C^0(\mathbb{R}^d)$ in $x=(x',z)$, and $C^\infty(\mathbb{R}^{d})$ in $\xi=(\xi',\zeta)$ with decay $\langle \zeta\rangle ^{-2}$. Define
    \begin{align*}
    \mathcal{K}:u \mapsto (\Op^h(a)u\; \delta)_{|z=0},
    \end{align*}
    where $\delta$ is the Dirac function on $z=0$. One has  \begin{align*}
    \|\mathcal{K}u\|_{L^2(\mathbb{R}^{d-1})} \lesssim h^{-1}\|u_{|z=0}\|_{L^2(\mathbb{R}^{d-1})}.
    \end{align*}
    \end{lemma}
Note that in all the above lemmas, the traces at $z=0^+$, $z=0^-$ and $z=0$ coincide. The proofs of the Lemmas \ref{lem1Th}, \ref{lem2Th} and \ref{lem3Th} are provided in Appendix \ref{appenAA1}. 
\begin{theorem}
\label{ellipA1}
The sequence $( u_{k_{\mid z=0}})_k$ is bounded in $L^2_{\varrho_0}(\partial\mathcal{L})$ for $\varrho_0\in {^\parallel\mathcal{E_\partial}}$.
\end{theorem}
As a direct consequence of Theorem \ref{ellipA1} one has the following result.
\begin{corollary} \label{corllellipA1}
The sequence $(h_k\partial_zu_{k_{\mid z=0}})_k$ is bounded in $L^2_{\varrho_0}(\partial \mathcal{L})$ for $\varrho_0\in {^\parallel\mathcal{E_\partial}}$.
\end{corollary}
The proof of Corollary \ref{corllellipA1} follows similarly as in the proof of Corollary \ref{cor1**A1}.
\begin{proof} [Proof of Theorem \ref{ellipA1}]
In local coordinates, the wave equation can be expressed near $\varrho_0\in {^\parallel\mathcal{E_\partial}}$ as follows
\begin{align}
\label{36££A1}
Pu_k=-\Big(g^{d,d}(x)\partial_z^2+R(t,x,D)\Big)u_k=0,
\end{align}
where $R(t,x,D)$ is given in  \eqref{exp of RA1}. The principal symbol of the operator associated to the equation
\eqref{36££A1} is given by
 \begin{align*}
p(x,\tau,\xi)=g^{d,d}(x)\zeta^2+2\underset{1\leq i\leq d-1}{\sum}g^{d,i}(x)\xi_i\zeta-r_0(x,\tau,\xi'),
\end{align*}
with
\begin{align*}
r_0(x,\tau,\xi')=\tau^2-\underset{1\leq i,j\leq d-1}{\sum}g^{i,j}(x)\xi_i\xi_j.
\end{align*} 
Near $\varrho_0\in {^\parallel\mathcal{E_\partial}}$, one has
 \begin{align}\label{5;32'A1}
r_0(x,\tau,\xi')<0 \quad  \text{and} \quad p(x,\tau,\xi)>0.
\end{align}
 Let $f(t,x,\tau,\xi') \in C^\infty_c(\phi_{\mathcal{L}}(O)\times\mathbb{R}^{d})$ supported near $\varrho_0$ such that $f(\varrho_0)=1$ and set 
 \begin{align} \label{qqqA1}
     q(t,x,\tau,\xi)=\frac{e(t,x,\tau,\xi')}{p(x,\tau,\xi)}, \quad \text{with} \quad e(t,x,\tau,\xi')=\sqrt{-r_0(x,\tau,\xi') }f(t,x,\tau,\xi').
 \end{align}
 The symbol $q$ is compactly supported in $(t,x,\tau,\xi')$-variables, $C^1(\mathbb{R}^d)$ in $x$-variable and $C^\infty(\mathbb{R}^{d+2})$ in $(t,\tau,\xi)$-variables with decay $\langle \zeta\rangle ^{-2}$. By Schur's lemma, the operator $\Op^h(q)$ is bounded on $L^2(\mathbb{R}^{d+1})$.
 Denote $\underline{u}$ the extension by $0$ of $u_k$ on $\{z<0\}$. Consider $\psi\in C^\infty_c(\phi_{\mathcal{L}}(O))$ with $\psi$ equal to $1$ in a neighborhood of the $(t,x)$-projection of $\mathrm{supp}(q)$ such that $\psi(t,x)=\psi_1(t,x')\psi_2(z)$ with $\psi_2(0)=1$ and $\partial_z \psi_2(0)=0$. We fix the following notations:
\begin{align*}
    \underline{v}=\psi\underline{u}, \quad g^{\pm \frac{1}{2}}=(\text{det}\;g)^{\pm\frac{1}{2}}, \quad \tilde{g}^{i,j}=g^\frac{1}{2} g^{i,j}, \quad  u=u_k \quad \text{and}\quad h=h_k.
\end{align*}
Writing
 \begin{align*}  P=\partial_t^2-\Delta_g'-\underset{\underset{(i,j)\neq (d,d)}{1\leq i,j\leq d}}{\sum}g^{-\frac{1}{2}}\partial_{i}\Big( g^\frac{1}{2}g^{i,j}\partial_{j}\Big)+1,
 \end{align*}
 with $\Delta_g'= \underset{1\leq i,j\leq d-1}{\sum}g^{-\frac{1}{2}}\partial_{i}\Big( g^\frac{1}{2}g^{i,j}\partial_{j}\Big)$,
 one has
 \begin{align} \label{expressTh}
    h^2P\underline{v}= \psi h^2P\underline{u} +h^2[P,\psi] \underline{u},
    \end{align}
where
    \begin{align*}
     P\underline{u}=\underline{(\partial_t^2-\Delta_g')u}-g^{-\frac{1}{2}} \Big(\partial_i\big(\tilde{g}^{i,d} \partial_{z} \underline{u}\big)+\partial_z\big(\tilde{g}^{d,j} \partial_{j}\underline{u}\big)+ \partial_z\big(\tilde{g}^{d,d} \partial_{z}\underline{u}\big) \Big) +\underline{u}.
 \end{align*}
 Using 
 \begin{align*}
    &\partial_z\underline{u}= u_{|z=0} \;\delta+\underline{\partial_zu},
 \end{align*}
 and the fact that $g^{i,d}=g^{d,j}=0$ for $i,j\neq d$, $g^{d,d}=1$ in $z=0$, one gets
 \begin{align*}
      &\quad g^\frac{1}{2} g^{i,d}\partial_z\underline{u}= \underline{g^\frac{1}{2}g^{i,d}\partial_zu}, \quad \quad \text{for}\quad i\neq d, \quad \text{and}\quad g^\frac{1}{2} g^{d,d} \partial_z\underline{u}=( g^\frac{1}{2} u)_{|z=0} \; \delta+ \underline{g^\frac{1}{2}g^{d,d}\partial_zu}.
      \end{align*}   
 Thus, 
 \begin{align*}
     g^{-\frac{1}{2}}\partial_i\big(\tilde{g}^{i,d} \partial_{z} \underline{u}\big)=g^{-\frac{1}{2}}\partial_i\big(\tilde{g}^{i,d} )\partial_z\underline{u}+g^{i,d}\partial_i\partial_z \underline{u}
     =\underline{g^{-\frac{1}{2}}\partial_i(\tilde{g}^{i,d})\partial_zu} +  \underline{g^{i,d}\partial_i\partial_z u}=\underline{g^{-\frac{1}{2}}\partial_i\big(\tilde{g}^{i,d} \partial_{z}u\big)},
     \end{align*}
     \begin{align*}
     g^{-\frac{1}{2}}\partial_z\big(\tilde{g}^{d,j} \partial_{j}\underline{u}\big)=\underline{g^{-\frac{1}{2}}\partial_z(\tilde{g}^{d,j}) \partial_ju}+ g^{d,j} \partial_j\partial_z\underline{u}=\underline{g^{-\frac{1}{2}}\partial_z(\tilde{g}^{d,j}) \partial_ju}+ \underline{g^{d,j} \partial_j\partial_zu} =\underline{g^{-\frac{1}{2}}\partial_z\big(\tilde{g}^{d,j}\partial_{j}u\big)},
     \end{align*}
     \begin{align*}
     g^{-\frac{1}{2}}\partial_z\big(\tilde{g}^{d,d} \partial_{z}\underline{u}\big)&= g^{-\frac{1}{2}} \partial_z\Big(\underline{\tilde{g}^{d,d}\partial_z u}+ (g^\frac{1}{2}u)_{|z=0}\;\delta\Big)= \underline{g^{-\frac{1}{2}}\partial_z(\tilde{g}^{d,d}\partial_z u)}+ (\partial_zu)_{|z=0}\delta+ g^{-\frac{1}{2}}(g^{\frac{1}{2}} u)_{|z=0}\; \delta'
     \\& =\underline{g^{-\frac{1}{2}}\partial_z(\tilde{g}^{d,d}\partial_z u)}+ (\partial_zu)_{|z=0}\delta+  u_{|z=0}\; \delta'-(\partial_zg^{-\frac{1}{2}})_{|z=0}(g^{\frac{1}{2}} u)_{|z=0} \delta.
 \end{align*}
 Note that the traces of $u$ and $\partial_zu$ are taken at $z=0^+$.\\
One then obtains,
 \begin{align*}
     P\underline{u}&=\underline{Pu}-u_{|z=0}\; \delta'+ \Big((\partial_zg^{-\frac{1}{2}})_{|z=0}(g^{\frac{1}{2}} u)_{|z=0}-(\partial_zu)_{|z=0}\Big)\delta
     \\&=- u_{|z=0}\; \delta'+ \Big((\partial_zg^{-\frac{1}{2}})_{|z=0}(g^{\frac{1}{2}} u)_{|z=0}-(\partial_zu)_{|z=0}\Big)\delta.
 \end{align*}
 Consequently, using $\psi(t,x)=\psi_1(t,x')\psi_2(z)$ with $\psi_2(0)=1$ and $\partial_z \psi_2(0)=0$, we deduce
 \begin{align} \label{275555TH}
     \psi h^2P \underline{u}&=-h^2  (\psi u)_{|z=0} \;\delta'+h^2\Big( (\partial_zg^{-\frac{1}{2}})_{|z=0}(g^{\frac{1}{2}} \psi u)_{|z=0}-( \psi \partial_zu)_{|z=0}\Big)\;\delta \nonumber
     \\&=-h^2 \psi_1 u_{|z=0}\; \delta'+h^2\Big( (g^{\frac{1}{2}}\partial_zg^{-\frac{1}{2}})_{|z=0} \psi_1 u_{|z=0}-\psi_1 (\partial_zu)_{|z=0}\Big)\;\delta.
 \end{align}
 One writes 
      \begin{align} \label{1TH}
    \Op^h(q)h^2P\underline{v} &
    =\Op^h(-q \tau^2)\underline{v}+ \Op^h(q) h^2 \underline{v}
    +I^{(1)}+I^{(2)}+I^{(3)}+ I^{(4)}, 
     \end{align}
     with 
     \begin{align*}
       &I^{(1)} = -\Op^h(q) g^{-\frac{1}{2}}h^2\partial_{i}\big( g^{\frac{1}{2}}g^{i,j}\partial_{j}\underline{v}\big), \quad I^{(2)}=-\Op^h(q) g^{-\frac{1}{2}} h \partial_z \big(g^\frac{1}{2} g^{d,j} h\partial_j\underline{v}\big),
       \\& I^{(3)}=-\Op^h(q) g^{-\frac{1}{2}} h \partial_i\big(g^\frac{1}{2} g^{i,d} h\partial_z \underline{v}\big), \quad  I^{(4)}=-\Op^h(q) g^{-\frac{1}{2}} h \partial_z\big(g^\frac{1}{2} g^{d,d} h\partial_z v\big).
     \end{align*}
     By Lemmas \ref{lem0A1} and \ref{lem1Th}, one obtains
\begin{align}\label{4Th}
   \Big( \Op^h(q) h^2 \underline{v}\Big)_{|z=0^-}=h^\frac{3}{2} O(1)_{L^2(\partial\mathcal{L})}.
\end{align}
For $(i,j)\neq (d,d)$, one has
     \begin{align}
       I_{|z=0^-}^{(1)}
    &=\Big(-\Op^h(q) g^{-\frac{1}{2}}h[\partial_i,g^{\frac{1}{2}}g^{i,j}]h \partial_j\underline{v}\Big)_{|z=0^-}-\Big(\Op^h(q)g^{i,j}h^2 \partial_i\partial_j \underline{v} \Big)_{|z=0^-} \nonumber
    \\&=h^\frac{1}{2} O(1)_{L^2(\partial \mathcal{L})}-\Big([\Op^h(q),g^{i,j}] h^2 \partial_i\partial_j \underline{v}\Big)_{|z=0^-}-\Big(g^{i,j} \Op^h(q) h^2\partial_i\partial_j\underline{v} \Big)_{|z=0^-} \label{2.76i1TH}
    \\&=\Big(\Op^h\big(q g^{i,j}\xi_i\xi_j\big) \underline{v}\Big)_{|z=0^-}+h^\frac{1}{2} O(1)_{L^2(\partial\mathcal{L})}, \label{2.77i1TH}
     \end{align}
     where in \eqref{2.76i1TH} we use Lemmas \ref{lem0A1} and \ref{lem1Th}, then in \eqref{2.77i1TH}, we use the fact that $g^{i,j}\in C^1$ along with Lemmas \ref{lem0A1} and \ref{lem2Th}. \\
     For $i=d$, $j\neq d$, one has
     \begin{align}
         I_{|z=0^-}^{(2)}&= \Big(-\Op^h(q) h g^{-\frac{1}{2}}[\partial_z,g^\frac{1}{2}g^{d,j}]h \partial_j \underline{v} \Big)_{|z=0^-}- \Big( \Op^h(q) g^{d,j} h \partial_z h \partial_j \underline{v}\Big)_{|z=0^-} \nonumber
         \\&= h^\frac{1}{2} O(1)_{L^2(\partial \mathcal{L})}-\Big([\Op^h(q), g^{d,j}] h\partial_z h \partial_j \underline{v}\Big)_{|z=0^-}+\Big( \Op^h(q g^{d,j} \zeta \xi_j) \underline{v} \Big)_{|z=0^-} \label{2.78I2TH}
         \\&
         =h^\frac{1}{2} O(1)_{L^2(\partial\mathcal{L})}+\Big( \Op^h(q g^{d,j} \zeta \xi_j) \underline{v}\Big)_{|z=0^-}- \Big([\Op^h(q), g^{d,j}]  h^2  \underline{\partial_j\partial_z (\psi u})\Big)_{|z=0^-} \nonumber
         \\&
         \quad- \Big([\Op^h(q), g^{d,j}]  h^2 \partial_j (\psi u)_{|z=0} \delta \Big)_{|z=0^-} \nonumber
         \\&=h^\frac{1}{2} O(1)_{L^2(\partial\mathcal{L})}+\Big( \Op^h(q g^{d,j} \zeta \xi_j) \underline{v}\Big)_{|z=0^-}+ h^\frac{1}{2} O(1)_{L^2(\partial\mathcal{L})} \label{2.799I2TH}\\&\quad
         - \Big([\Op^h(q), g^{d,j}]  h^2\partial_j (\psi u)_{|z=0} \delta \Big)_{|z=0^-} 
        \nonumber \\&=h^\frac{1}{2} O(1)_{L^2(\partial\mathcal{L})}+\Big( \Op^h(q g^{d,j} \zeta \xi_j) \underline{v}\Big)_{|z=0^-}+ h^\frac{1}{2} O(1)_{L^2(\partial\mathcal{L})}\nonumber\\&\quad- \Big(\Op^h(q)g^{d,j} h^2 \partial_j (\psi u)_{|z=0} \delta \Big)_{|z=0^-} + \Big(g^{d,j} \Op^h(q)h^2 \partial_j(\psi u)_{|z=0} \delta \Big)_{|z=0^-} \nonumber
         \\&=\Big( \Op^h(q g^{d,j} \zeta \xi_j) \underline{v}\Big)_{|z=0^-}+ h^\frac{1}{2} O(1)_{L^2(\partial\mathcal{L})},  \label{2.80I2TH}
     \end{align}
where in \eqref{2.78I2TH} we apply Lemmas \ref{lem0A1} and \ref{lem1Th}, \eqref{2.799I2TH} follows from Lemmas \ref{lem0A1} and \eqref{lem2Th} and in the last line we use the fact that $g^{d,j}=0$ in $z=0$ for $j\neq d$. \\
For $i\neq d$, $j=d$, one has
\begin{align}
    I_{|z=0^-}^{(3)}&=-\Big(\Op^h(q)g^{-\frac{1}{2}} [\partial_i, g^\frac{1}{2}g^{i,d}] h^2\partial_z\underline{v} \Big)_{|z=0^-}- \Big(\Op^h(q) g^{i,d}h^2 \partial_i \partial_z\underline{v} \Big)_{|z=0^-} \nonumber
    \\&
    =-\Big(\Op^h(q)g^{-\frac{1}{2}} [\partial_i, g^\frac{1}{2}g^{i,d}] h^2 \partial_z\underline{v} \Big)_{|z=0^-}+ \Big(\Op^h(q g^{i,d} \xi_i \zeta) \underline{v}\Big)_{|z=0^-}- \Big([\Op^h(q),g^{i,d}] h^2 \partial_i \partial_z\underline{v} \Big)_{|z=0^-} \nonumber
    \\&
    =-\Big(\Op^h(q)g^{-\frac{1}{2}} [\partial_i, g^\frac{1}{2}g^{i,d}] h^2 \underline{\partial_z v}\Big)_{|z=0^-}- \Big(h^\frac{3}{2} \Op^h(q) [\partial_i, g^\frac{1}{2}g^{i,d}] (\psi h^\frac{1}{2} u)_{|z=0} \delta \Big)_{|z=0^-} \nonumber
\\& \quad
+ \Big(\Op^h(q g^{i,d} \xi_i \zeta) \underline{v}\Big)_{|z=0^-}
- \Big([\Op^h(q),g^{i,d}] h^2 \partial_i \partial_z\underline{v} \Big)_{|z=0^-} 
\nonumber
\\&
=h^\frac{1}{2} O(1)_{L^2(\partial\mathcal{L})}+ h^\frac{1}{2}O(1)_{L^2(\partial\mathcal{L})} + \Big(\Op^h(q g^{i,d} \xi_i \zeta) \underline{v} \Big)_{|z=0^-} 
- \Big([\Op^h(q),g^{i,d}] h^2 \underline{\partial_i \partial_zv} \Big)_{|z=0^-}\label{2.80I3TH}
\\& \quad - \Big([\Op^h(q),g^{i,d}] h^2 \partial_i (\psi u)_{|z=0} \delta \Big)_{|z=0^-}
\nonumber
\\&=  \Big(\Op^h(q g^{i,d} \xi_i \zeta) \underline{v}\Big)_{|z=0^-}+ h^\frac{1}{2} O(1)_{L^2(\partial\mathcal{L})}, \label{2.81I3th}
\end{align}
where \eqref{2.80I3TH} follows from Lemmas \ref{lem0A1}, \ref{lem1Th} and \ref{lem3Th} along with \eqref{5.3111111A1} and in the last line we use Lemmas \ref{lem0A1}, \ref{lem2Th} as well as $g^{i,d}=0$ in $z=0$ for $i\neq d$.\\
     For $(i,j)=(d,d)$, one has 
     \begin{align}
         I_{|z=0^-}^{(4)}&=\Big(-\Op^h(q) g^{d,d}  h^2 \partial_z^2 \underline{v}\Big)_{|z=0^-}-\Big(\Op^h(q) g^{-\frac{1}{2}}[\partial_z,g^\frac{1}{2} g^{d,d}] h^2 \partial_z\underline{v}\Big)_{|z=0^-} \nonumber
         \\&=\Big(\Op^h(q g^{d,d}\zeta^2)\underline{v}\Big)_{|z=0^-}-\Big([\Op^h(q),g^{d,d}]h^2 \partial_z^2\underline{v}\Big)_{|z=0^-}-\Big(\Op^h(q) g^{-\frac{1}{2}}[\partial_z,g^\frac{1}{2} g^{d,d}] h^2 \partial_z\underline{v}\Big)_{|z=0^-}
         \nonumber
         \\&
         =\Big(\Op^h(q g^{d,d}\zeta^2)\underline{v}\Big)_{|z=0^-}-\Big([\Op^h(q),g^{d,d}]h^2 \partial_z^2\underline{v}\Big)_{|z=0^-}
         -\Big(\Op^h(q) g^{-\frac{1}{2}}[\partial_z,g^\frac{1}{2} g^{d,d}] h^2 \psi \underline{\partial_z u} \Big)_{|z=0^-} \nonumber
         \\& \quad
         -\Big(h^\frac{3}{2}\Op^h(q) g^{-\frac{1}{2}}[\partial_z,g^\frac{1}{2} g^{d,d}]  (\psi h^\frac{1}{2} u)_{|z=0} \delta \Big)_{|z=0^-} \nonumber
         - \Big(\Op^h(q) g^{-\frac{1}{2}}[\partial_z,g^\frac{1}{2} g^{d,d}] h^2 (\partial_z\psi) \underline{ u} \Big)_{|z=0^-} \nonumber
         \\&
         =\Big(\Op^h(q g^{d,d}\zeta^2)\underline{v}\Big)_{|z=0^-}-\Big([\Op^h(q),g^{d,d}]h^2 \partial_z^2\underline{v}\Big)_{|z=0^-} \nonumber
         \\& \quad
         +h^\frac{1}{2}O(1)_{L^2(\partial\mathcal{L})} +h^\frac{1}{2}   O(1)_{L^2(\partial\mathcal{L})}+h^\frac{3}{2} O(1)_{L^2(\partial\mathcal{L})} \label{2.88444TH}
         \\&
         = \Big(\Op^h(q g^{d,d}\zeta^2)\underline{v}\Big)_{|z=0^-}+h^\frac{1}{2} O(1)_{L^2(\partial\mathcal{L})}-\Big([\Op^h(q),g^{d,d}]h^2 \partial_z^2\underline{v}\Big)_{|z=0^-}, \label{2.8885TH}
     \end{align}
     where in \eqref{2.88444TH} we apply Lemmas \ref{lem0A1}, \ref{lem1Th} and \ref{lem3Th} along with \eqref{5.3111111A1}.
     Using
     \begin{align*}
     \partial_z^2\underline{v}=(\partial_z^2\psi)\underline{u}+(\partial_z\psi) \big(u \;\delta + \underline{\partial_z u} \big)+(\partial_z\psi) \underline{\partial_z u}+ \psi \big(\partial_zu \;\delta+ \underline{\partial_z^2u} \big)+(\psi u)_{|z=0} \;\delta',
     \end{align*}
     and $(\partial_z \psi)_{|z=0}=0$,  we rewrite the last term in $I_{|z=0^-}^{(4)}$ as
     \begin{align}
     \Big([\Op^h(q),g^{d,d}]&h^2 \partial_z^2\underline{v}\Big)_{|z=0^-}\\&=  \Big([\Op^h(q),g^{d,d}]h^2  (\partial_z^2\psi)\underline{u}\Big)_{|z=0^-}+ 2\Big([\Op^h(q),g^{d,d}]h^2(\partial_z\psi) \underline{\partial_z u} \Big)_{|z=0^-} \nonumber
     \\& \quad
     +\Big([\Op^h(q),g^{d,d}]h^2\psi \big(\partial_zu\; \delta+ \underline{\partial_z^2u} \big)\Big)_{|z=0^-}+ \Big([\Op^h(q),g^{d,d}]h^2(\psi u)_{|z=0} \delta'\Big)_{|z=0^-} \nonumber
     \\&=h^\frac{5}{2}O(1)_{L^2(\partial\mathcal{L})}+h^\frac{3}{2}O(1)_{L^2(\partial\mathcal{L})} +h^\frac{1}{2}O(1)_{L^2(\partial\mathcal{L})}\label{2.8666}
     \\& \quad
     + \Big([\Op^h(q),g^{d,d}]h^2\psi \big(\partial_zu) \delta\Big)_{|z=0^-} + \Big([\Op^h(q),g^{d,d}]h^2(\psi u)_{|z=0} \delta'\Big)_{|z=0^-} \nonumber
     \\&=h^\frac{1}{2}O(1)_{L^2(\partial\mathcal{L})}+\Big(\Op^h(q) g^{d,d} h^2 \psi \big(\partial_zu) \delta\Big)_{|z=0^-}-\Big( g^{d,d} \Op^h(q)  h^2 \psi \big(\partial_zu) \delta\Big)_{|z=0^-} \nonumber
     \\&\quad + \Big([\Op^h(q),g^{d,d}]h^2(\psi u)_{|z=0} \delta'\Big)_{|z=0^-} \nonumber
     \\&=h^\frac{1}{2}O(1)_{L^2(\partial\mathcal{L})}+ \Big([\Op^h(q),g^{d,d}]h^2(\psi u)_{|z=0} \delta'\Big)_{|z=0^-}
     \label{2.877Th}
 \end{align}
 where \eqref{2.8666} is a consequence of Lemmas \ref{lem0A1} and \ref{lem2Th} and \eqref{2.877Th} follows from $g^{d,d}=1$ in $z=0$. Using further that $g^{d,d}=1$ in $z=0$, together with Lemmas \ref{lem0A1} and \ref{lem3Th}, one has 
\begin{align}
     \Big([\Op^h(q)&,g^{d,d}]h^2(\psi u)_{|z=0} \delta'\Big)_{|z=0^-}\nonumber
     \\&= \Big(\Op^h(q) h^2 g^{d,d} (\psi u)_{|z=0} \delta'\Big)_{|z=0^-}- \Big(g^{d,d}\Op^h(q) h^2 (\psi u)_{|z=0} \delta'\Big)_{|z=0^-} \nonumber
     \\&= \Big(\Op^h(q) h^2 (\psi u)_{|z=0} g^{d,d}_{|z=0}\delta'\Big)_{|z=0^-} 
     -\Big(\Op^h(q) h^2 (\psi u)_{|z=0} 
     (\partial_z g^{d,d})_{|z=0} \delta \Big)_{|z=0^-} \nonumber
     \\& \quad
     - \Big(\Op^h(q) h^2(\psi u)_{|z=0} \delta'\Big)_{|z=0^-}  \nonumber
     \\&=  -\Big( h^\frac{3}{2}\Op^h(q)  h^2 (\psi h^\frac{1}{2} u)_{|z=0} 
     (\partial_z g^{d,d})_{|z=0} \delta \Big)_{|z=0^-}= h^\frac{1}{2} O(1)_{L^2(\partial \mathcal{L})}, \label{2.8999999TH}
\end{align}
Combining  \eqref{2.8885TH}, \eqref{2.877Th} and \eqref{2.8999999TH} yields
\begin{align}\label{2Th}
    I_{|z=0^-}^{(4)}= \Big(\Op^h(q g^{d,d}\zeta^2)\underline{v}\Big)_{|z=0^-}+h^\frac{1}{2} O(1)_{L^2(\partial\mathcal{L})}.
\end{align}
Collecting \eqref{1TH},\eqref{4Th}, \eqref{2.77i1TH}, \eqref{2.80I2TH}, \eqref{2.81I3th} and \eqref{2Th}, we deduce
\begin{align}
     \Big(\Op^h(q)h^2P\underline{v}\Big)_{|z=0^-}&= \underset{1\leq i,j\leq d}{\sum}\Big(\Op^h(q(\tau^2-g^{i,j}\xi_i\xi_j)\underline{v}\Big)_{|z=0^-}+h^\frac{1}{2} O(1)_{L^2(\partial\mathcal{L})} \nonumber
     \\&
     = \Big(\Op^h(qp) \underline{v}\Big)_{|z=0^-}+h^\frac{1}{2} O(1)_{L^2(\partial\mathcal{L})}= \Big(\Op^h(e) \underline{v}\Big)_{|z=0^-}+h^\frac{1}{2} O(1)_{L^2(\partial\mathcal{L})}.
     \label{28666Th}
\end{align}
To deal with the second term in \eqref{expressTh}, we argue as following. Since $[P,\psi]$ is a first-order differential operator, we distinguish two cases: if the derivatives act with respect to the variable $ z$, we apply Lemmas \ref{lem0A1}, \ref{lem1Th} and \ref{lem3Th}; otherwise, we use Lemmas \ref{lem0A1} and \ref{lem1Th}. Consequently, we obtain
     \begin{align} \label{3Th}
        \Big(\Op^h(q)h^2[P,\psi]\underline{u} \Big)_{|z=0^-}= h^\frac{1}{2} O(1)_{L^2(\partial\mathcal{L})}.
     \end{align}
From  \eqref{expressTh}, \eqref{275555TH}, \eqref{28666Th} and \eqref{3Th}, we obtain the following semi-classical pseudo-differential trace equation on $u$,
\begin{align} \label{traceequaTh}
   \Big(\Op^h(e) \underline{v}\Big)_{|z=0^-}+h^\frac{1}{2} O(1)_{L^2(\partial\mathcal{L})} \nonumber
 &= \Big(\Op^h(q) \big(-h^2 \psi_1 u_{|z=0}\; \delta'\big) \Big)_{|z=0^-} \nonumber
    \\& \quad
    +\Big(\Op^h(q) h^2 \big(\big( (g^{\frac{1}{2}}\partial_zg^{-\frac{1}{2}})_{|z=0} \psi_1 u_{|z=0}-\psi_1 (\partial_zu)_{|z=0}\big)\delta \big)\Big)_{|z=0^-}.
\end{align}
Note that
\begin{align*}
    \Op^h(e) \underline{v}&= \frac{1}{(2\pi)^{d+1}} \int e^{i(t-s).\tau+i(x'-y').\xi'+i(z-y_d).\zeta} e(t,x,h\tau,h\xi') \underline{v}(s,y) \; ds\; dy \;d\xi' d\tau\; d\zeta
    \\&=\frac{1}{(2\pi)^{d+1}} \int e^{i(t-s).\tau+i(x'-y').\xi'} e(t,x,h\tau,h\xi') \underline{v}(s,y) \delta(z-y_d) \; ds\; dy \;d\xi' d\tau
    \\&=\frac{1}{(2\pi)^{d+1}} \int e^{i(t-s).\tau+i(x'-y').\xi'} e(t,x,h\tau,h\xi') \underline{v}(s,y) \delta(z-y_d) \; ds\; dy \;d\xi' d\tau
    \\&=\frac{1}{(2\pi)^{d+1}} \int e^{i(t-s).\tau+i(x'-y').\xi'} e(t,x,h\tau,h\xi') \underline{v}(s,y',z) \; ds\; dy \;d\xi' d\tau.
\end{align*}
Thus, \begin{align} \label{thirdoneTH}
\Big(\Op^h(e) \underline{v}\Big)_{|z=0-}=0.
\end{align}
Since $p>0$ near $\varrho_0$, it follows from \eqref{qqqA1} that the symbol $q$ has two complex poles 
\begin{align} \label{polesTH}
    \zeta_\pm = \frac{
-\displaystyle\sum_{i=1}^{d-1} g^{d,i}(x)\,\xi_i
\;\pm i\;
\sqrt{
-\left(\sum_{i=1}^{d-1} g^{d,i}(x)\,\xi_i\right)^2
- g^{d,d}(x)\, r_0(x,\tau,\xi')
}
}{
g^{d,d}(x)
}.
\end{align}
At $z=0$, these simplify to $\zeta_\pm=\pm i \sqrt{-r_0(x,\tau,\xi')_{|z=0}}$. Denoting 
\begin{align*}
w=(g^{\frac{1}{2}}\partial_zg^{-\frac{1}{2}})_{|z=0} \psi_1 u_{|z=0}-\psi_1 (\partial_zu)_{|z=0},
\end{align*}
 one writes
\begin{align}
    \Op^h(q)\big(h^2 w \;\delta\big)&= \frac{h^2}{(2\pi h)^{d+1}} \int_{\mathbb{R}^{2d+1}} e^{\frac{i}{h}\big((t-s)\tau+ (x'-y').\xi' +z. \zeta \big)} q(t,x,\tau,\xi',\zeta)  w(s,y')\; d\zeta\;d\xi'\;d\tau \;dy' \; ds \nonumber
    \\&=\frac{h}{(2\pi h)^{d}} \int_{\mathbb{R}^{2d}} e^{\frac{i}{h}\big((t-s)\tau+ (x'-y').\xi' \big)} t_1(t,x,\tau,\xi')  w(s,y')\;d\xi'\;d\tau \;dy' \; ds, \label{theniwa7daTh}
\end{align}
where
\begin{align*}
    t_1(t,x,\tau,\xi')= \frac{1}{2 \pi} \int_{\mathbb{R}} e^{\frac{i}{h}z\zeta} q(t,x,\tau,\xi',\zeta) d\zeta.
\end{align*}
Let $\Gamma_M$ be the lower semicircle in the region $\mathrm{Im}(\zeta) \leq 0$. If $\zeta \in \Gamma_M$, then $|e^{\frac{i}{h}z\zeta}|=e^{|z|\frac{\mathrm{Im}\zeta}{h}} \leq 1$. Since $q$ is with decay $\langle \zeta\rangle ^{-2}$, one obtains 
\begin{align} \label{29777TH}
    \underset{M \rightarrow +\infty}{ \mathrm{lim}}  \int_{\Gamma_M }e^{\frac{i}{h}z\zeta} q(t,x,\tau,\xi',\zeta) d\zeta =0.
\end{align}
Hence,
\begin{align*}
t_1(t,x,\tau,\xi')= \underset{M \rightarrow +\infty}{ \mathrm{lim}} \frac{1}{2 \pi} \int_{\Gamma_M \cup [-M,M]} e^{\frac{i}{h}z\zeta} q(t,x,\tau,\xi',\zeta) d\zeta.
\end{align*}
For large $M$, the pole $\zeta_-$ lies inside the contour.  Since $q$ is holomorphic in $\zeta$ except at the pole $\zeta_-$, which lies inside the contour, the residue theorem gives
\begin{align*}
 t_1(t,x,\tau,\xi')= -i e^{\frac{i}{h}z \zeta_-} \frac{e(t,x,\tau,\xi')}{g^{d,d}(\zeta_--\zeta_+)}.
\end{align*}
Thus, for $z<0$,
\begin{align}
    \Op^h(q)\big(h^2 w \;\delta\big)= \frac{-i h}{(2\pi h)^{d}} e^{\frac{i}{h}z \zeta_-} \int_{\mathbb{R}^{2d}} e^{\frac{i}{h}\big((t-s)\tau+ (x'-y').\xi' \big)}  \frac{e(t,x,\tau,\xi')}{g^{d,d}(\zeta_--\zeta_+)} w(s,y')\;d\xi'\;d\tau \;dy' \; ds.
\end{align}
Then taking the trace for $z=0^-$ and using that $g^{d,d}=1$ along with $\zeta_-=- \zeta_+=-i \sqrt{-r_0(x,\tau,\xi')_{|z=0}}$ in $z=0$, one has
\begin{align} \label{secononeTh}
   \Big(\Op^h(q)&\big(h^2 w \;\delta\big) \Big)_{|z=0} =\Op^h(r_1)h w,
\end{align}
with
\begin{align} \label{2.9999TH}
    r_1= \frac{f(t,x',z=0,\tau,\xi')}{2}.
\end{align}
Similarly, we write
\begin{align} 
\Op^h&(q)\big(-h^2 \psi_1 u_{|z=0} \;\delta'\big)\nonumber
\\&=\frac{-h}{(2\pi h)^{d+1}}\int_{\mathbb{R}^{2d+1}} e^{\frac{i}{h}\big((t-s)\tau+ (x'-y').\xi'+z\zeta\big)} q(t,x,\tau,\xi',\zeta) \;i\zeta \;  (\psi_1 u_{|z=0})(s,y')\;d\zeta \; d\xi'\;d\tau \;dy' \; ds \nonumber
\\&
=\frac{-1}{(2\pi h)^{d}}\int_{\mathbb{R}^{2d}} e^{\frac{i}{h}\big((t-s)\tau+ (x'-y').\xi'+z\zeta\big)} t_2(t,x,\tau,\xi') \;  (\psi_1 u_{|z=0})(s,y') \; d\xi'\;d\tau \;dy' \; ds, \label{2.9444TH}
\end{align}
where
\begin{align*}
   t_2(t,x,\tau,\xi') =\frac{1}{2\pi}\int_{\mathbb{R}} e^{\frac{i}{h}z \zeta} q(t,x,\tau,\xi',\zeta)i \zeta\; d\zeta
   = \Big(\frac{h}{2 \pi} \partial_{\widehat{z}} \int_{\mathbb{R}}
e^{\frac{i}{h}\widehat{z} \zeta} q(t,x,\tau,\xi',\zeta) d\zeta \Big)_{|\widehat{z}=z}.
\end{align*}
For $\widehat{z}<0$, by the residue theorem, one has
\begin{align*}
  \frac{1}{2 \pi} \int_{\mathbb{R}}
e^{\frac{i}{h}\widehat{z} \zeta} q(t,x,\tau,\xi',\zeta) d\zeta = -i e^{\frac{i}{h}\widehat{z} \zeta_-} \frac{e(t,x,\tau,\xi')}{g^{d,d}(\zeta_--\zeta_+)}.
\end{align*}
Hence 
\begin{align*}
    t_2(t,x,\tau,\xi')= \Big(\zeta_-e^{\frac{i}{h}\widehat{z} \zeta_-} \frac{e(t,x,\tau,\xi')}{g^{d,d}(\zeta_--\zeta_+)} \Big)_{|\widehat{z}=z}=\zeta_-e^{\frac{i}{h}z \zeta_-} \frac{e(t,x,\tau,\xi')}{g^{d,d}(\zeta_--\zeta_+)}.
\end{align*}
Consequently, for $z<0$,
\begin{align*}
    \Op^h(q)&\big(-h^2 \psi_1 u_{|z=0} \;\delta'\big)\\&=\frac{-1}{(2\pi h)^{d}}e^{\frac{i}{h}z \zeta_-}\int_{\mathbb{R}^{2d}} e^{\frac{i}{h}\big((t-s)\tau+ (x'-y').\xi'+z\zeta\big)} \zeta_- \frac{e(t,x,\tau,\xi')}{g^{d,d}(\zeta_--\zeta_+)} \;  (\psi_1 u_{|z=0})(s,y') \; d\xi'\;d\tau \;dy' \; ds, 
\end{align*}
Taking the trace at $z=0^-$ and using that $g^{d,d}=1$ along with $\zeta_-=- \zeta_+$ in $z=0$, one gets
\begin{align} \label{firstonaTH}
   \Big(\Op^h(q) \big(-h^2 \psi_1 u_{|z=0}\; \delta'\big) \Big)_{|z=0}= \Op^h(r_2)(\psi_1 u_{|z=0}),
\end{align}
with
\begin{align}
   r_2= \frac{-e(t,x',z=0,\tau,\xi')}{2}. \label{2.955Th}
\end{align}
Using the boundary condition $h(\partial_zu)_{|z=0}=h D_t u_{|z=0}$, we get
\begin{align} \label{fourthoneTH}
-\Op^h(r_1 )h\psi_1( \partial_zu)_{|z=0}=-\Op^h(r_1 \tau)\psi_1 u_{|z=0}+\Op^h(r_1)(D_t\psi_1) hu_{|z=0}.
\end{align}
Collecting \eqref{traceequaTh}, \eqref{thirdoneTH}, \eqref{secononeTh}, \eqref{firstonaTH} and \eqref{fourthoneTH}, we deduce
\begin{align*}
-\Op^h(r_1\tau)&\psi_1 u_{|z=0}+\Op^h(r_1)(D_t\psi_1) hu_{|z=0}\\&+\Op^h(r_1)(g^{\frac{1}{2}} \partial_zg^{-\frac{1}{2}})_{|z=0} \psi_1 h u_{|z=0}+\Op^h(r_2) \psi_1 u_{|z=0}=h^\frac{1}{2} O(1)_{L^2(\partial\mathcal{L})}.
\end{align*}
With the trace formula \eqref{5.3111111A1}
along with Lemma \ref{lem0A1}, we conclude
\begin{align} \label{exx1th}
    \Op^h(r_1)(D_t\psi_1) hu_{|z=0}+\Op^h(r_1) (g^{\frac{1}{2}} \partial_zg^{-\frac{1}{2}})_{|z=0} \psi_1 h u_{|z=0} =h^\frac{1}{2} O(1)_{L^2(\partial\mathcal{L})}.
\end{align}
Let $\varphi \in C^\infty_c(\mathbb{R})$ such that $\varphi \equiv 1$ in a neighborhood of $[\alpha,\alpha^{-1}]$. Writing 
\begin{align}
-h^2\Delta_g &\big(1-\varphi(h D_t)\big) \psi_1(t,x')u+ h^2 \big(1-\varphi(h D_t)\big) \psi_1(t,x')u\nonumber\\&=-\big(1-\varphi(h D_t)\big) \psi_1(t,x')h^2\partial_t^2u-\big(1-\varphi(h D_t)\big)[\Delta_g,\psi_1] h^2u,
\end{align}
then using that $[\Delta_g,\psi_1]$ is a differential operator of order one with continuous coefficients along with the fact that $h^2 u(t,.)$ is bounded in $H^2(\mathcal{M})$, for any $t \in \mathbb{R}$ by Lemma \ref{lem0A1} and 
arguying as in \eqref{4*A1},  we deduce
\begin{align} \label{exx2th}
    \psi_1 u_{|z=0}=\varphi(h D_t) \psi_1 u_{|z=0} +O(h^N),
\end{align}
for all $N\in \mathbb{N}$.
From \eqref{exx1th} and \eqref{exx2th}, we deduce
\begin{align*}
   h^\frac{1}{2} O(1)_{L^2(\partial\mathcal{L})}&= \Op^h(-r_1\tau)\varphi(h_kD_t) \psi_1 u_{|z=0}+\Op^h(r_2)\psi_1 u_{|z=0}\\&= \Op^h(-r_1\tau \varphi(\tau)+r_2 )\psi_1 u_{|z=0}\\&=\Op^h\Big(-\frac{f}{2}\big(\tau \varphi(\tau)+\sqrt{-r_0(x,\tau,\xi')_{|z=0}}\big)\Big) \psi_1 u_{|z=0}.
\end{align*}
 Taking $\varphi$ supported in $[\frac{\alpha}{2}, 2 \alpha^{-1}]$ so that $]-\infty,0] \cap \mathrm{supp}(\varphi)=\varnothing$, one has $\tau \varphi(\tau)\geq 0.$ Then, recalling that $-r_0(x,\tau,\xi')_{|z=0}>0$ near $\varrho_0$, we conclude that $(u_{k_{|z=0}})_k$ is bounded in $L^2_{\varrho_0}(\partial \mathcal{L})$ in the sense of Definition \ref{def1A1}. 
\end{proof}

\subsubsection{Hyperbolic region}
The properties of the traces of $(u_k)_k$ in the hyperbolic region ${^\parallel\mathcal{H_\partial}}$ are outlined in the following two results.
\begin{theorem}
\label{hyppA1}
The sequence $(h_k\partial_zu_{k_{\mid z=0}})_k$ is bounded in $L^2_{\varrho_0}(\partial\mathcal{L})$ for $\varrho_0\in {^\parallel\mathcal{H_\partial}}$.
\end{theorem}
With Theorem \ref{hyppA1} one obtains the following result.
\begin{corollary} \label{corrlhyperA1}
The sequence $(u_{k_{\mid z=0}})_k$ is bounded in $L^2_{\varrho_0}(\partial \mathcal{L})$ for $\varrho_0\in {^\parallel\mathcal{H_\partial}}$.
\end{corollary}
For the proof of Corollary \ref{corrlhyperA1} we argue as in the proof of Corollary \ref{cor1**A1} using the boundary condition.
\begin{proof}[Proof of Theorem \ref{hyppA1}]

Locally near $\varrho_0 \in {^\parallel\mathcal{H_\partial}}$, the wave equation is given by
\begin{align} \label{hyyypA1}
Pu_k=-\Big(g^{d,d}(x)\partial_z^2+R(t,x,D)\Big)u_k=0,
\end{align}
where $R(t,x,D)$ is given in  \eqref{exp of RA1}.  The principal symbol of the operator associated with equation \eqref{hyyypA1} is given by
\begin{align*}
p(x,\tau,\xi)=g^{d,d}(x)\zeta^2+2\underset{1\leq i\leq d-1}{\sum}g^{d,i}(x)\xi_i\zeta-r_0(x,\tau,\xi'),
\end{align*}
with
\begin{align}\label{5.61111r0A1}
r_0(x,\tau,\xi')=\tau^2-\underset{1\leq i,j\leq d-1}{\sum}g^{i,j}(x)\xi_i\xi_j.
\end{align}
Near $\varrho_0 \in {^\parallel\mathcal{H_\partial}}$, we have 
$(g^{d,d}(x))^{-1} r_0(x,\tau,\xi')>0.$\\
Let $Z>0$. Consider $\theta \in C^\infty(\mathbb{R})$  such that
\begin{equation}
    \theta(z)= \begin{cases}
       1 \quad \text{for} \quad z\leq\frac{Z}{2}   \\ 
        0 \quad \text{for} \quad z\geq Z.
    \end{cases}
\end{equation} Set
\begin{align} \label{5.58expV_kA1}
v_k=\theta(z)u_k.
\end{align}
Locally, near $\varrho_0 \in {^\parallel\mathcal{H_\partial}}$ with $ \tilde{g}^{i,j}(x)= (\mathrm{det}g)^\frac{1}{2} g^{i,j}(x)$, we write
\begin{align} h_k^2(&g^{d,d}(x))^{-1}P=h_k^2\partial_z^2+h_k^2\underset{1\leq i\leq d-1}{\sum}(\tilde{g}^{d,d}(x))^{-1}\partial_{x_i}\Big(\tilde{g}^{i,d}(x)\partial_z\Big)+h_k^2(\tilde{g}^{d,d}(x))^{-1} \Big(\partial_z\tilde{g}^{d,d}(x)\Big)\partial_z\nonumber\\&\quad\quad\quad+h_k^2\underset{1\leq j\leq d-1} {\sum}(\tilde{g}^{d,d}(x))^{-1}\partial_z \Big(\tilde{g}^{d,j}(x)\partial_{x_j}\Big)
+ h_k^2\underset{1\leq i,j\leq d-1}{\sum}[(\tilde{g}^{d,d}(x))^{-1},\partial_{x_i}]\tilde{g}^{i,j}(x) \partial_{x_j}\nonumber \\&\hspace{8cm}-h_k^2(g^{d,d}(x))^{-1}+\tilde{R}(t,x,h_kD_t,h_kD'),
\end{align}
with 
\begin{align}\label{tildeRA1}
\tilde{R}(t,x,h_kD_t,h_kD')=-h_k^2(g^{d,d}(x))^{-1}\partial_t^2&+h_k^2\underset{1\leq i,j\leq d-1}{\sum}  \partial_{x_i}\Big((g^{d,d}(x))^{-1} g^{i,j}(x) \partial_{x_j}\Big). \nonumber\\&
\end{align}
This gives
\begin{align} \label{expressiondepA1}
& h_k^2\partial_z^2v_k
+2h_k^2\underset{1\leq i\leq d-1}{\sum} (g^{d,d}(x))^{-1}g^{i,d}(x)\partial_{x_i}\partial_zv_k+\tilde{R}(t,x,h_kD_t,h_kD') v_k=h_k^2f,
\end{align}
where 
\begin{align*}
    f&=\theta''(z)u_k+2\theta'(z)\partial_zu_k+\theta'(z)(\tilde{g}^{d,d}(x))^{-1}\Big(\partial_z\tilde{g}^{d,d}(x) \Big)u_k
    \\&\quad+\theta'(z)\underset{1\leq i\leq d-1}{\sum} (\tilde{g}^{d,d}(x))^{-1} \partial_{x_i} \Big(\tilde{g}^{i,d}(x) u_k\Big)+\theta'(z)\underset{1\leq j\leq d-1}{\sum} (g^{d,d}(x))^{-1}g^{d,j}(x)\partial_{x_j}u_k
   \\& \quad-\underset{1\leq j\leq d-1} {\sum}(\tilde{g}^{d,d}(x))^{-1} \Big(\partial_z\tilde{g}^{d,j}(x)\Big)\partial_{x_j}v_k
    -\underset{1\leq i\leq d-1}{\sum}(\tilde{g}^{d,d}(x))^{-1}\Big(\partial_{x_i}\tilde{g}^{i,d}(x)\Big)\partial_zv_k\\&\quad-(\tilde{g}^{d,d}(x))^{-1} \Big(\partial_z\tilde{g}^{d,d}(x)\Big) \partial_zv_k- h_k^2\underset{1\leq i,j\leq d-1}{\sum}[(\tilde{g}^{d,d}(x))^{-1},\partial_{x_i}]\tilde{g}^{i,j}(x) \partial_{x_j}v_k+h_k^2(g^{d,d}(x))^{-1}v_k.
\end{align*}
Let $e(t,x',\tau,\xi')$ be a $C^\infty_c(\mathbb{R}^{2d})$ elliptic symbol near $\varrho_0$, compactly supported on the $(t,x')$-projection of $\phi_{\mathcal{L}}(O)$ and $\chi(t,x',\tau,\xi')\in C^\infty_c(\mathbb{R}^{2d})$ with $\chi=1$  on $\mathrm{supp}(e)$. The size of the supports will be adjusted below. Define the semi-classical operator
\begin{align}
    Q(t,x,h_kD_t,h_kD')=\Op^h(\chi)^* \tilde{R}(t,x,h_kD_t,h_kD') \Op^h(\chi)+\Op^h(1-\chi)^*\Lambda^2\Op^h(1-\chi),
\end{align}
with $\Lambda^2=\Op^h(\lambda^2)$ and $\lambda^2=\tau^2+|\xi'|^2$. Note that $Q(t,x,h_kD_t,h_kD')$ is symmetric and its  real part symbol is given by
\begin{align}\label{qqqqA1}
\mathrm{Re}(q)=\chi^2 (g^{d,d}(x))^{-1}r_0+(1-\chi)^2\lambda^2 \geq 0,
\end{align}
with $r_0$ as in \eqref{5.61111r0A1}. The symbol $q$ is tangential and elliptic on $\mathrm{supp}(\theta)$. Let $\psi\in C^\infty_c(\mathbb{R}_t\times\mathbb{R}_{x'}^{d-1})$ with $\psi $ equal to $1$ in a neighborhood of the $(t,x')$-projection of $\mathrm{supp}(e)$.
Set
\begin{align}\label{ptildeA1}
    h_k^2\tilde{P}=h_k^2\partial_z^2+Q(t,x,h_kD_t,h_kD') \quad \text{and} \quad W_k=\begin{pmatrix}
    w_k \\
    h_k \partial_zw_k
\end{pmatrix}=\begin{pmatrix}
    \Op^h(e)\psi v_k \\
     \Op^h(e)\psi h_k\partial_z v_k
\end{pmatrix}.
\end{align}
Since $\chi=1$ on $\mathrm{supp}(e)$, using \eqref{expressiondepA1}, one can write
\begin{align}\label{tildepA1}
h_k^2\tilde{P}w_k&=\Op^h(e)\psi \Big(h_k^2\partial_z^2+\tilde{R}(t,x,h_kD_t,h_kD')\Big)v_k\nonumber
\\&\quad+\Op^h(\chi)^*[\tilde{R}(t,x,h_kD_t,h_kD'),\Op^h(e)\psi]v_k
+O(h_k^\infty)_{\mathcal{L}(L^2)}\psi v_k\nonumber
\\&\quad +\Op^h(r_1)\psi\tilde{R}(t,x,h_kD_t,h_kD')v_k+\Op^h(\chi)^* \tilde{R}(t,x,h_kD_t,h_kD')\Op^h(r_2)\psi v_k \nonumber
\\&=\Op^h(e)\psi\Big(-2h_k^2 \underset{1\leq i\leq d-1}{\sum} (g^{d,d}(x))^{-1}g^{i,d}(x)\partial_{x_i}\partial_zv_k+ h_k^2f\Big)\nonumber
\\&\quad+\Op^h(\chi)^*[\tilde{R}(t,x,h_kD_t,h_kD'),\Op^h(e)]\psi v_k +O(h_k^\infty)_{\mathcal{L}(L^2)}\psi v_k \nonumber
\\&\quad+\Op^h(r_1)\psi\tilde{R}(t,x,h_kD_t,h_kD')v_k+\Op^h(\chi)^*\Op^h(e)[\tilde{R}(t,x,h_kD_t,h_kD'),\psi]v_k,
\end{align}
where $\Op^h(r_1), \Op^h(r_2) \in h_k^\infty \Psi^{-\infty}$.
From \eqref{tildeRA1}, one has
\begin{align}\label{5.62222A1}
     [\tilde{R}(t,x,h_kD_t,h_kD'),\Op^h(e)]=J_1+J_2,
\end{align}
with 
\begin{align*}
    &J_1= [-h_k^2(g^{d,d}(x))^{-1}\partial_t^2,\Op^h(e)],
    \\&J_2= \underset{1\leq i,j\leq d-1}{\sum}[h_k^2 \partial_{x_i}\Big((g^{d,d}(x))^{-1} g^{i,j}(x) \partial_{x_j}\Big),\Op^h(e)].
\end{align*}
One rewrite $J_1$ as follows
\begin{align}
    J_1=[\Op^h(e),(g^{d,d}(x))^{-1}]h_k^2\partial_t^2+(g^{d,d}(x))^{-1}[\Op^h(e),h_k^2\partial_t^2].
\end{align}
By symbolic calculus, one gets
\begin{align}\label{5.28888A1}
 (g^{d,d}(x))^{-1}[h_k^2\partial_t^2,\Op^h(e)]=i(g^{d,d}(x))^{-1}h_k\Op^h(2\tau \partial_te)+o(h_k)_{\mathcal{L}(L^2)}.
\end{align}
Using formula (5.12) from Proposition 5.10 in \cite{1A1}, we obtain
\begin{align}\label{J111A1}
    [&\Op^h(e),(g^{d,d}(x))^{-1}]h_k^2\partial_t^2\nonumber\\&=-ih_k^3\sum\limits_{1 \leq k \leq d-1}\Big(\partial_{x_k} (g^{d,d}(x))^{-1}\Big)\Op^h(\partial_{\xi_k}e)\partial_t^2\nonumber+o(h_k^3)_{\mathcal{L}(L^2)}\partial_t^2\nonumber
\\&=ih_k\sum\limits_{1 \leq k \leq d}(\partial_{x_k} g^{i,j})\Op^h(\tau^2\partial_{ \xi_k}e)+o(h_k^3)_{\mathcal{L}(L^2)}\partial_t^2.
\end{align}
Thus, from \eqref{5.28888A1} and \eqref{J111A1}, we deduce
\begin{align}\label{J111111A1}
    J_1&=i(g^{d,d}(x))^{-1}h_k\Op^h(2\tau \partial_te)+ih_k\sum\limits_{1 \leq k \leq d}(\partial_{x_k} g^{i,j})\Op^h(\tau^2\partial_{ \xi_k}e)\nonumber
   \\&\quad+o(h_k)_{\mathcal{L}(L^2)}+o(h_k^3)_{\mathcal{L}(L^2)}\partial_t^2.
    \end{align}
Setting $\alpha^{i,j}=(g^{d,d}(x))^{-1}g^{i,j}(x)$ and $\varrho^{i,j}=[\partial_{x_i},\alpha^{i,j}]$, one writes
\begin{align} \label{soommkkkA1}
  J_2=K_1+K_2+K_3+K_4,
\end{align}
with
\begin{align*}
&K_1=h_k^2\varrho^{i,j}[\partial_{x_j},\Op^h(e)], \quad \quad\quad K_3=h_k^2\alpha^{i,j}[\partial_{x_i}\partial_{x_j},\Op^h(e)],
\\&K_2=h_k^2[\varrho^{i,j},\Op^h(e)]\partial_{x_j},\quad\quad\quad K_4= h_k^2[\alpha^{i,j},\Op^h(e)]\partial_{x_i}\partial_{x_j}.
\end{align*}
We have
\begin{align} \label{kk1A1}
    K_1=h_k^2\varrho^{i,j}\Op^h(\partial_{x_j}e)=O(h_k^2)_{\mathcal{L}(L^2)},
    \end{align}
    and
    \begin{align}\label{kk3A1}
  K_3&=ih_k\alpha^{i,j}\Op^h(\xi_i\partial_{x_j}e+\xi_j\partial_{x_i}e)+h_k^2\alpha^{i,j}\Op^h(\partial_{x_i}\partial_{x_j}e)\nonumber
    \\&=h_k\alpha^{i,j} \Op^h(\partial_{x_j}e)h_k\partial_{x_i}+h_k\alpha^{i,j}\Op^h(\partial_{x_i}e)h_k\partial_{x_j}+O(h_k^2)_{\mathcal{L}(L^2)}.
\end{align}
Since $\varrho^{i,j}\in C^0\cap L^\infty$, we deduce from Proposition 5.10 (see formula (5.8)) in \cite{1A1},
\begin{align}\label{kk2A1}
    K_2=h_k^2\sum\limits_{1 \leq j \leq d}o(1)_{\mathcal{L}(L^2)}\partial_{x_j}=o(h_k^2)_{\mathcal{L}(H^1,L^2)}.
\end{align}
Moreover, using formula (5.12) from Proposition 5.10 in \cite{1A1}, we get
\begin{align}\label{kk4A1}
K_4&=ih_k^3\sum\limits_{1 \leq k \leq d-1} \Big(\partial_{x_k}\alpha^{i,j}\Big)\Op^h(\partial_{\xi_k}e)\partial_{x_i}\partial_{x_j}+o(h_k^3)_{\mathcal{L}(H^2,L^2)}\nonumber
\\&=-ih_k\sum\limits_{1 \leq k \leq d-1} \Big(\partial_{x_k}\alpha^{i,j}\Big)\Op^h(\xi_i \xi_j\partial_{ \xi_k}e)+o(h_k^3)_{\mathcal{L}(H^2,L^2)}.
\end{align}
Furthermore, arguing as for \eqref{kk1A1} and \eqref{kk4A1}, one has
\begin{align}\label{5.7555555A1}
-2&\Op^h(e)\psi(h_k^2 \underset{1\leq i\leq d-1}{\sum} \alpha^{i,d}\partial_{x_i}\partial_zv_k)\nonumber\\&=-2\underset{1\leq i\leq d-1}{\sum}\alpha^{i,d}\Op^h(e)\psi h_k^2\partial_{x_i}\partial_zv_k
\nonumber-2\underset{1\leq i\leq d-1}{\sum}[\Op^h(e)\psi,\alpha^{i,d}]h_k^2\partial_{x_i}\partial_zv_k \nonumber\nonumber\\&=
-2\underset{1\leq i\leq d-1}{\sum}\alpha^{i,d}[\Op^h(e)\psi,\partial_{x_i}]h_k^2\partial_zv_k
 -2\underset{1\leq i\leq d-1}{\sum}\alpha^{i,d}h_k^2\partial_{x_i}\partial_zw_k +O(h_k)_{\mathcal{L}(L^2)}\psi_1v_k\nonumber
 \\&=-2\underset{1\leq i\leq d-1}{\sum}\alpha^{i,d}h_k^2\partial_{x_i}\partial_zw_k+O(h_k)_{\mathcal{L}(L^2)}\psi_2v_k,
\end{align}
where $\psi_i\in C^\infty_c(\mathbb{R}^d)$ such that $\psi_i=1$ on $\mathrm{supp}(\psi)$ for $i\in \{1,2\}$.\\
Collecting \eqref{5.62222A1}, \eqref{J111111A1}, \eqref{soommkkkA1}, \eqref{kk1A1}, \eqref{kk3A1}, \eqref{kk2A1},\eqref{kk4A1} and \eqref{5.7555555A1}, inserting the result in \eqref{tildepA1}, then using Lemma \ref{lem0A1}, we deduce
\begin{align*}
   & h_k^2\tilde{P}w_k+2\underset{1\leq i\leq d-1}{\sum}\alpha^{i,d}h_k^2\partial_{x_i}\partial_zw_k
    \\&= \Op^h(e)\psi h_k^2f+O(h_k)_{\mathcal{L}(L^2)}\psi_3v_k
    + \Op^h(r_1)\psi \tilde{R}(t,x,h_kD_t,h_kD')v_k=O(h_k)_{L^2(\mathcal{L})},
\end{align*}
where $\psi_3\in C^\infty_c(\mathbb{R}^d)$ such that $\psi_3=1$ on $\mathrm{supp}(\psi)$.
This implies along with \eqref{ptildeA1},
\begin{align}\label{exprematricA1}
  h_k\partial_zW_k = \tilde{B}W_k+\tilde{F},
\end{align}
where  
\begin{align*}
    \tilde{B}=\begin{pmatrix}
    0&\quad 1\\
 -Q(t,x,h_kD_t,h_kD')& \quad -2 \underset{1\leq i\leq d-1}{\sum} \alpha^{i,d} h_k\partial_{x_i} 
\end{pmatrix}
\quad \text{and}\quad \tilde{F}=\begin{pmatrix}
    0 \\
 O(h_k)_{L^2(\mathcal{L})}
\end{pmatrix}.
\end{align*}
Let $\mathcal{L'}=\mathbb{R}_t\times\mathbb{R}_{x'}^{d-1}$ and consider the measure $\mu'=dtdx_1...dx_{d-1}$ on $\mathcal{L'}$. Define $N(.)$ and $\tilde{N}(.,.)$, as follows
\begin{align} \label{expnormA1}
    & \tilde{N}(W,Y)=(Q(t,x,h_kD_t,h_kD')w^0,y^0)_{H^{-1}(\mathcal{L}'),H^1(\mathcal{L}')}+(w^1,y^1)_{L^2(\mathcal{L}')}^2,\nonumber
     \\&N(W)=\tilde{N}(W,W),
\end{align}
for $W=(w^0,w^1)$ and $Y=(y^0,y^1)$ in $H^1(\mathcal{L'})\times L^2(\mathcal{L'})$. Note that, by integration by parts, the quadratic form $(Q(t,x,h_kD_t,h_kD')w^0,w^0)_{H^{-1}(\mathcal{L}'),H^1(\mathcal{L}')}$ is symmetric.\\
One writes
\begin{align}\label{G1plusG2A1}
  &(Q(t,x,h_kD_t,h_kD')w_k,w_k)_{H^{-1}(\mathcal{L}'),H^1(\mathcal{L}')} =G_1+G_2,
\end{align}
with \begin{align*}
    &G_1=(\tilde{R}(t,x,h_kD_t,h_kD') \Op^h(\chi)w_k,\Op^h(\chi)w_k)_{L^2(\mathcal{L}')},
    \\& G_2=(\Lambda^2\Op^h(1-\chi)w_k,\Op^h(1-\chi)w_k)_{L^2(\mathcal{L}')}.
\end{align*}
From \eqref{tildeRA1} and by integration by parts, one gets
\begin{align*}
    G_1&=\big((g^{d,d}(x))^{-1}h_k\partial_t\Op^h(\chi)w_k,h_k\partial_t\Op^h(\chi)w_k\big)_{L^2(\mathcal{L}')}\\& \quad-\underset{1\leq i,j\leq d-1}{\sum}(h_k^2 \alpha^{i,j} \partial_{x_j}\Op^h(\chi)w_k,  \partial_{x_i}\Op^h(\chi)w_k)_{L^2(\mathcal{L}')}.  
\end{align*}
Applying Taylor’s formula to the functions $(g^{d,d})^{-1}$ and $\alpha^{i,j}$ at a point $x^0$, we obtain
\begin{align*}
    (g^{d,d})^{-1}(x)=(g^{d,d})^{-1}(x^0)+\sum\limits_{1\leq k \leq d}(x_k-x_k^0)(e^{d,d})^{-1}_k(x),
    \end{align*}
and
    \begin{align*}
    \alpha^{i,j}(x)=\alpha^{i,j}(x^0)+\sum\limits_{1\leq k \leq d}(x_k-x_k^0)e^{i,j}_k(x),
\end{align*}
where $(e^{d,d})_k^{-1}$ and $e^{i,j}_k$ are continuous functions.
Thus, 
\begin{align}\label{G'1plusG'2A1}
    G_1&=G_1'+G''_1,
\end{align}
with 
\begin{align*}
    G_1'&=\big((g^{d,d}(x^0))^{-1}h_k\partial_t\Op^h(\chi)w_k,h_k\partial_t\Op^h(\chi)w_k\big)_{L^2(\mathcal{L}')}\\& \quad-\underset{1\leq i,j\leq d-1}{\sum}(h_k^2 \alpha^{i,j}(x^0) \partial_{x_j}\Op^h(\chi)w_k,  \partial_{x_i}\Op^h(\chi)w_k)_{L^2(\mathcal{L}')}
\end{align*}
and
\begin{align*}
    G_1''&=\sum\limits_{1\leq k \leq d}\Big(\big((x_k-x_k^0)(e^{d,d})^{-1}_k(x)h_k\partial_t\Op^h(\chi)w_k,h_k\partial_t\Op^h(\chi)w_k\big)_{L^2(\mathcal{L}')}\\& \quad-\underset{1\leq i,j\leq d-1}{\sum}\big( (x_k-x_k^0)e^{i,j}_k(x)h_k\partial_{x_j}\Op^h(\chi)w_k,  h_k\partial_{x_i}\Op^h(\chi)w_k\big)_{L^2(\mathcal{L}')}\Big)
\end{align*}
Then, one obtains
\begin{align*}
    G'_1+G_2=(Q(t,x^0,h_kD_t,h_kD')w_k,w_k)_{H^{-1}(\mathcal{L}'),H^1(\mathcal{L}')}.
\end{align*}
Since the symbol of $Q(t,x^0,h_kD_t,h_kD')$ is elliptic with smooth coefficients and positive real part, it follows from the semi-classical Gårding inequality that there exists a constant $C_0>0$ such that 
\begin{align}\label{G'1plusG2A1}
    G'_1+G_2\geq C_0\|w_k\|_{H^1_{sc}(\mathcal{L}')}^2.
\end{align}
Let $\tilde{\chi}(t,x,\tau,\xi')\in C^\infty_c(\mathbb{R}^{2d+1})$ with $\tilde{\chi}=1$  on $\mathrm{supp}(\chi)$. Using that $\chi=1$ on $\mathrm{supp}(e)$, one has
\begin{align*}
\Op^h(\chi)w_k=\Op^h(\tilde{\chi})\Op^h(e)\psi v_k+O(h_k^\infty)_{\mathcal{L}(L^2)}\psi v_k,
\end{align*}
and using Cauchy Schwarz along with Lemma \ref{lem0A1}, one gets
\begin{align*}
    G_1''\leq C_1\sum\limits_{1\leq k \leq d} \underset{ \mathrm{supp}(\tilde{\chi})}{\mathrm{sup}}|x_k-x_k^0|\; \Big(\|w_k\|_{H^1_{sc}(\mathcal{L}')}^2+O(h_k^\infty)\Big).
\end{align*}
Taking $\mathrm{supp}(\tilde{\chi})$ sufficiently small so that
\begin{align*}
C_1\sum\limits_{1\leq k \leq d} \underset{ \mathrm{supp}(\tilde{\chi})}{\mathrm{sup}}|x_k-x_k^0| \leq \frac{C_0}{2},
\end{align*}
one gets
\begin{align} \label{G''1A1}
    G_1''\leq \frac{C_0}{2}\|w_k\|_{H^1_{sc}(\mathcal{L}')}^2+O(h_k^\infty).
\end{align}
Collecting \eqref{G1plusG2A1}, \eqref{G'1plusG'2A1}, \eqref{G'1plusG2A1} and \eqref{G''1A1}, we conclude
\begin{align} 
    (Q(t,x,h_kD_t,h_kD')w_k,w_k)_{H^{-1}(\mathcal{L}'),H^1(\mathcal{L}')}&\geq \frac{C_0}{2}\|w_k\|_{H^1_{sc}(\mathcal{L}')}^2-O(h_k^\infty).\nonumber
\end{align}
Hence, \eqref{expnormA1} gives
\begin{align}\label{5.877777777A1}
    N(W_k)\geq C\Big(\|w_k\|_{H^1_{sc}(\mathcal{L}')}^2+\|h_k \partial_zw_k\|_{L^2(\mathcal{L}')}^2\Big)-O(h_k^\infty).
\end{align}
From \eqref{exprematricA1}, we obtain
\begin{align*}
\frac{d}{dz}N(h_kW_k)&=2 \mathrm{Re}\;\tilde{N}(\frac{d}{dz}h_kW_k,h_kW_k)
\\&
=2 \mathrm{Re}\;\tilde{N}(\tilde{B}W_k,h_kW_k)
+\mathrm{Re}\big(O(h_k)_{L^2(\mathcal{L})},h_k^2\Op^h(e)\psi\partial_zv_k\big)_{L^2(\mathcal{L'})}
\\&
=I_1+I_2+I_3+I_4.
\end{align*}
where 
\begin{align*}
    &I_1=2\mathrm{Re}(Q(t,x,h_kD_t,h_kD')\Op^h(e)\psi h_k\partial_zv_k,h_k\Op^h(e)\psi v_k)_{L^2(\mathcal{L'})},
    \\&I_2=-2\mathrm{Re}(Q(t,x,h_kD_t,h_kD')\Op^h(e)\psi v_k,h_k^2\Op^h(e)\psi \partial_zv_k)_{L^2(\mathcal{L'})},
    \\&I_3=-4\underset{1\leq i\leq d-1}{\sum} \mathrm{Re}(\alpha^{i,d}h_k^2\partial_{x_i}\Op^h(e)\psi\partial_zv_k, h_k^2\Op^h(e)\psi\partial_zv_k)_{L^2(\mathcal{L'})},
    \\&I_4= \mathrm{Re}(O(h_k)_{L^2(\mathcal{L})},h_k^2\Op^h(e)\psi\partial_zv_k)_{L^2(\mathcal{L'})}.
\end{align*}
Using the symmetry of $Q$, one obtains $I_1+I_2=0$.\\
One has
\begin{align*}
 I_3&=-2\underset{1\leq i\leq d-1}{\sum}(\alpha^{i,d}h_k^2\partial_{x_i}\Op^h(e)\psi\partial_zv_k, h_k^2\Op^h(e)\psi\partial_zv_k)_{L^2(\mathcal{L'})}\\&\quad-2\underset{1\leq i\leq d-1}{\sum}(h_k^2\Op^h(e)\psi\partial_zv_k,\alpha^{i,d}h_k^2\partial_{x_i}\Op^h(e)\psi\partial_zv_k)_{L^2(\mathcal{L'})}
 \\&=2\underset{1\leq i\leq d-1}{\sum}\big(h_k^2\Op^h(e)\psi\partial_zv_k, h_k^2(\partial_{x_i}\alpha^{i,d})\Op^h(e)\psi\partial_zv_k\big)_{L^2(\mathcal{L'})}.
 \end{align*}
Therefore,
\begin{align*}
  \frac{d}{dz}&N(h_kW_k)=2\underset{1\leq i\leq d-1}{\sum}\big(h_k^2\Op^h(e)\psi\partial_zv_k, h_k^2(\partial_{x_i}\alpha^{i,d})\Op^h(e)\psi\partial_zv_k\big)_{L^2(\mathcal{L'})}+ I_4.
\end{align*}
Thus, using that $(h_k\partial_zv_k)_k$ is bounded in $L^2_{\mathrm{loc}}(\mathcal{L})$ by Lemma \ref{lem0A1} and $\partial_{x_i}\alpha^{i,d}\in C^0\cap L^\infty$, one gets
\begin{align} \label{5.62hypA1}
 h_k^2N(W_k)(z)&\lesssim\int_Z^z  \Big(h_k^2\|\Op^h(e)\psi h_k\partial_{z}v_k\|^2_{L^2(\mathcal{L'})}(z')+h_k^2 \nonumber\\&\quad+
\underset{1\leq i\leq d-1}{\sum}h_k^2\|(\partial_{x_i}\alpha^{i,d})\Op^h(e)\psi h_k\partial_{z}v_k\|^2_{L^2(\mathcal{L'})}(z')
\Big) dz'\nonumber
 \\&\lesssim h_k^2(1+|z-Z|).
\end{align}
Taking $z=0$ in \eqref{5.62hypA1}, one has 
\begin{align}
    N(W_k)(0) \lesssim 1+|Z|.
\end{align}
Hence, by \eqref{ptildeA1} and \eqref{5.877777777A1}, we obtain
\begin{align*}
\|\Op^h(e)\psi h_k\partial_zv_{k_{|z=0}}\|_{L^2(\partial\mathcal{L})}^2\lesssim 1+|Z|+O(h_k^\infty).
    \end{align*}
This concludes the proof since $\partial_zv_{k_{|z=0}}= \partial_zu_{k_{|z=0}}$ by \eqref{5.58expV_kA1}.
\end{proof}
\subsection{Proof of the propagation equation} \label{proofpropagA1}
The measure equation \eqref{13£££A1} is local. Consequently,
its proof can be carried out in local charts. So,
let $C=(O,\phi_\mathcal{L})$ be a local chart, where $O$ is a neighborhood of a point $\varrho^0 \in \partial \mathcal{L}$. We recall that in local coordinates $\mathcal{L}$ is given by $\{z\geq0\}$. Let $\chi \in C^\infty_c([\alpha,\alpha^{-1}])$ such that $\chi \equiv 1$ in a neighborhood of $[\alpha,\alpha^{-1}]$. We consider a function $c(\varrho)=c(t,x,\tau,\xi)$ of the space $ \Sigma_0^\mathcal{H} (\mathbb{R}^{2d+2}) $. By Definition \ref{definespaceA1}, the symbol $c$ has compact support in the $y=(t,x)$ variable, denoted as $\mathrm{supp}(c) \subset K \times \mathbb{R}^{d+1}$, where K is a compact subset of $\phi_\mathcal{L}(O)$, and exhibits rapid decay in the $\eta=(\tau,\xi)$ variable.
A decomposition of the symbol $c$, presented in Proposition 8.2 of \cite{1A1}, makes tangential symbols appear, so that one can writes,
\begin{align}\label{decomA1}
\chi(\tau)c(\varrho)=c_0(t,x,\tau,\xi')+c_1(t,x,\tau,\xi')\zeta+q(\varrho)p(\varrho),
\end{align}
where $c_0$ and $c_1$ are tangential symbols satisfying 
\begin{align}\label{5;37''(A1}
|&\partial_y^\alpha\partial_{\eta'}^\beta c_i(y,\eta')|\leq C_{N,\beta} \langle \eta'\rangle^{-N},
\end{align}
for $ N\in \mathbb{N},\; \alpha \in \mathbb{N}^{d+1},\; |\alpha|\leq 1,\; \beta \in  \mathbb{N}^d,\; i=0,1,\; y\in\mathbb{R}^{d+1},\; \eta'=(\tau,\xi')\in \mathbb{R}\times \mathbb{R}^{d-1}$
and $q$ fulfills property (8.4) of Proposition 8.2 in \cite{1A1}.
We recall that in local coordinates the Hamiltonian vector field associated with the principal symbol $p$ of the wave operator is given by
\begin{align} \label{hpcA1}
H_p(\varrho) =-2\tau\partial_t+2 g^{ij}(x)\xi_i\partial_{x_j}-\partial_{x_k} g^{i,j}(x)\xi_i\xi_j\partial_{\xi_k}.
\end{align}
By the definition of $\chi$ and formula \eqref{24*A1} of Proposition \ref{hhA1}, one has $\chi=1$ on $\mathrm{supp}(\mu)$. Thus, one writes
\begin{align} \label{5.39effA1}
\langle ^t H_p\mu,c\rangle=\langle \mu,H_p c\rangle=\langle \mu,\chi H_p c\rangle=\langle \mu, H_p (\chi c)\rangle.
\end{align}
Inserting \eqref{decomA1} in \eqref{5.39effA1} and using $H_p(p)=\{p,p\}=0$ and the fact that by Proposition \ref{hhA1}, $\mathrm{supp}(\mu)\subset \mathrm{Char}(p)$, we get
\begin{align}\label{5.38''(A1}
\langle ^t H_p\mu,c\rangle&=\langle \mu, H_p (c_0)\rangle+\langle \mu, H_p (c_1)\zeta\rangle+ \langle \mu, H_p (q)p\rangle \nonumber
\\&=\langle \mu, H_p (c_0)+H_p (c_1)\zeta\rangle.
\end{align}
In view of \eqref{5.38''(A1}, we deduce that the symbol $c$ can be chosen of the form 
\begin{align*}
c(\varrho)= c_0(t,x,\tau,\xi')+c_1(t,x,\tau,\xi')\zeta.
\end{align*}
With this form one computes \begin{align} \label{5.40RRA1}
H_pc(\varrho)=\{p,c\}(\varrho)= \tilde{c}_0(y,\tau,\xi')+\tilde{c}_1(y,\tau,\xi')\zeta+\tilde{c}_2(y,\tau,\xi')\zeta^2,
\end{align}
where, by \eqref{5;37''(A1}, the tangential symbols $\tilde{c}_0,\tilde{c}_1$  and  $\tilde{c}_2$ are in $\Sigma^{0,n}_{T,0}(\langle\eta'\rangle^{-N};\mathbb{R}^{d+1}\times \mathbb{R}^{d})$, for any $n \in \mathbb{N} \cup \{+\infty\}$, with $n \geq d$ and $N \in \mathbb{N}$ as defined in Definition \ref{def 3.2A1}.\\
Note that $-ih_kH_pc=-ih_k\{p,c\}$ is the principal symbol of the commutator $[\Op^h(p),\Op^h(c)]$.
Hence, in order to evaluate $\langle\mu,H_pc\rangle$, we compute the limit as $k$ tends to infinity of 
\begin{align*}
i(  h_k [P, \Op^h(c)]\psi u_k,u_k)_{L^2(\mathcal{L})},
\end{align*}
where $\psi \in C^\infty_c(\phi_\mathcal{L}(O))$ such that $\psi \equiv 1$ in a neighborhood of $K$. Note that $\psi=1$ in a
neighborhood of the $y$-projection of $\mathrm{supp}(\tilde{c}_i)$, for $i=0,1,2$. One Writes 
\begin{align}
     [h_k^2 P, \Op^h(c)\psi]&= [h_k^2P,\Op^h(c)]\psi+h_k^2\Op^h(c)[P,\psi]\nonumber
     \\&=[h_k^2P,\Op^h(c)]\psi+h_k^2O(1)_{\mathcal{L}(H^1,L^2)}, \label{operlimiteTh}
\end{align}
where in the last line we use the fact that $[P,\psi]$ is a differential operator of order one with continuous coefficients.
The commutator $[h_k^2P,\Op^h(c)]$ is given by the following lemma.
\begin{lemma}
\label{lem4.1££A1}
One has
\begin{align*}
[h_k^2P,\Op^h(c)]=-ih_k \Op^h(H_pc)+ o(h_k)_{\mathcal{L}(L^2)}.
\end{align*}
\end{lemma}
\begin{proof}[Proof]
 Recalling that $P=\partial_t^2-\Delta_g+1$, one can write
\begin{align}\label{5;44dvvxgA1}
    [h_k^2P,\Op^h(c)]=[h_k^2\partial_t^2,\Op^h(c)]-[h_k^2\Delta_g,\Op^h(c)].
\end{align}
By symbolic calculus, one has
\begin{align*}
[h_k^2\partial_t^2,\Op^h(c)]=ih_k\Op^h(2\tau \partial_tc)+o(h_k)_{\mathcal{L}(L^2)}.
\end{align*}
Next, with repeated indices convention, we write the Laplace–Beltrami operator
\begin{align*}
\Delta_g= g^{i,j}(x) \partial_{x_i}\partial_{x_j}+\varrho^{i,j}\partial_{x_j},
\end{align*}
where  $\tilde{g}^{i,j}=(\text{det}\; g)^{\frac{1}{2}}g^{i,j}$ and $\varrho^{i,j}=(\text{det}\; g)^{-\frac{1}{2}}[\partial_{x_i},\tilde{g}^{i,j}]$.
We then compute the commutator
\begin{align}\label{commuA1}
[h_k^2\Delta_g,\Op^h(c)]=K_1+K_2+K_3+K_4,
\end{align}
with
\begin{align*}
&K_1=h_k^2\varrho^{i,j}[\partial_{x_j},\Op^h(c)], \quad\quad\quad\quad \quad K_2=h_k^2[\varrho^{i,j},\Op^h(c)]\partial_{x_j},
\\&K_3=h_k^2g^{i,j}(x)[\partial_{x_i}\partial_{x_j},\Op^h(c)],\quad\quad K_4= h_k^2[g^{i,j}(x),\Op^h(c)]\partial_{x_i}\partial_{x_j}.
\end{align*}
Since $\partial_{x_j}c\in \Sigma^{\infty,\infty}_0(\langle\xi\rangle^{-\infty};\mathbb{R}^{2d+2})$, applying Lemma 7.1 in \cite{1A1} we get
\begin{align}\label{k_1A1}
K_1&=h_k^2\varrho^{i,j} \Op^h(\partial_{x_j}c)+O(h_k^4)_{\mathcal{L}(L^2)}
=O(h_k^2)_{\mathcal{L}(L^2)}.
\end{align}
 Moreover, using the fact that $\varrho^{i,j}\in C^0\cap L^\infty$, we deduce from Proposition 5.10 (see formula (5.8)) in \cite{1A1},
\begin{align}\label{K_2A1}
K_2=h_k^2\sum\limits_{1 \leq j \leq d}o(1)_{\mathcal{L}(L^2)}\partial_{x_j}=o(h_k^2)_{\mathcal{L}(H^1,L^2)}.
\end{align}
By symbolic calculus, we have
\begin{align}\label{K_3A1}
K_3=ih_k g^{i,j}(x)\Op^h(\{\xi_i\xi_j,c\})+O(h_k^2)_{\mathcal{L}(L^2)}.
\end{align}
Furthermore, using formula (5.12) from Proposition 5.10 in \cite{1A1}, we get
\begin{align}\label{K_4A1}
K_4&=ih_k^3\sum\limits_{1 \leq k \leq d}(\partial_{x_k} g^{i,j})\Op^h(\partial_{\xi_k}c)\partial_{x_i}\partial_{x_j}+o(h_k^3)_{\mathcal{L}(H^2,L^2)}\nonumber
\\&=-ih_k\sum\limits_{1 \leq k \leq d}(\partial_{x_k} g^{i,j})\Op^h(\xi_i \xi_j\partial_{ \xi_k}c)+o(h_k^3)_{\mathcal{L}(H^2,L^2)}.
\end{align}
Inserting \eqref{k_1A1}, \eqref{K_2A1}, \eqref{K_3A1} and \eqref{K_4A1} in \eqref{commuA1} and combining the result with \eqref{5;44dvvxgA1} and \eqref{hpcA1}, we conclude the proof.
\end{proof}
By \eqref{5.40RRA1} and Lemma \ref{lem4.1££A1}
\begin{align}\label{5.50sssA1}
i(h_k[P,\Op^h(c)]\psi u_k,u_k)_{L^2(\mathcal{L})}&=( \Op^h(\tilde{c}_0)\psi u_k,u_k)_{L^2(\mathcal{L})}+(\Op^h(\tilde{c}_1)h_kD_d\psi u_k,u_k)_{L^2(\mathcal{L})}\nonumber
\\&\quad+(\Op^h(\tilde{c}_2)h_k^2D_d^2 \psi u_k,u_k)_{L^2(\mathcal{L})}+ o(1)_{k \rightarrow +\infty}.
\end{align}
Applying Lemma \ref{massleakA1} together with Proposition 5.21 in \cite{1A1}, one has
\begin{align} \label{c_0 tildeA1}
( \Op^h(\tilde{c}_0)\psi u_k,u_k)_{L^2(z\geq 0)} \underset{k\rightarrow +\infty}{\longrightarrow}\langle\mu,\tilde{c}_0\rangle.
\end{align}
The limits of the two remaining terms on the right-hand side of \eqref{5.50sssA1} are given by the following lemma.
\begin{lemma}
\label{lem11A1}
One has
\begin{align}
(\Op^h(\tilde{c}_1) h_k D_d \psi u_k,u_k)_{L^2(z\geq 0)} \underset{k\rightarrow +\infty}{\longrightarrow} \langle\mu,\tilde{c}_1\zeta \rangle,
\end{align}
and
\begin{align}
(\Op^h(\tilde{c}_2) h_k^2 D_d^2\psi u_k,u_k)_{L^2(z\geq 0)} \underset{k\rightarrow +\infty}{\longrightarrow} \langle \mu, \tilde{c}_2\zeta^2\rangle.
\end{align}

\end{lemma}
\begin{proof}[Proof]
Consider $\phi \in C^\infty_c(\mathbb{R})$ such that $\phi(\sigma)=1$ for $|\sigma| \leq 1$ and $R>1$. One writes
\begin{align*}
( \Op^h(\tilde{c}_1) h_kD_d \psi u_k,u_k)_{L^2(z\geq 0)}&=( \Op^h(\tilde{c}_1) \mathds{1}_{z\geq 0} h_kD_d \psi u_k,\mathds{1}_{z\geq 0} u_k) _{L^2(\mathbb{R}^{d+1})}
\\&=I_1+I_2,
\end{align*}
with 
\begin{align*}
    &I_1=( \Op^h(\tilde{c}_1)\Op^h\big(\phi(\frac{\zeta}{R})\big)\mathds{1}_{z\geq 0}h_k D_d\psi u_k, \mathds{1}_{z\geq 0} u_k)_{L^2(\mathbb{R}^{d+1})},
\end{align*}
and
\begin{align*}
 \quad\quad I_2=( \Op^h(\tilde{c}_1)\Op^h\big(1-\phi(\frac{\zeta}{R})\big)\mathds{1}_{z\geq 0} h_kD_d\psi u_k,\mathds{1}_{z\geq 0}u_k)_{L^2(\mathbb{R}^{d+1})}.
\end{align*}
Since $\Op^h(\tilde{c}_1)$ is bounded on $L^2(\mathbb{R}^{d+1})$ (see Corollary \ref{lem4.7A1}) and $(u_k)_k$ is bounded in $L^2_{\mathrm{loc}}(\mathcal{L})$, uniformly with respect to $k$, we get
\begin{align*}
| I_2 | \lesssim \|\Op^h\big(1-\phi(\frac{\zeta}{R})\big)  \mathds{1}_{z\geq 0} h_k D_d \psi u_k\|_{L^2(\mathbb{R}^{d+1})}.
\end{align*}
Furthermore, by Lemma \ref{lem0A1}, one has $\mathds{1}_{z\geq 0}h_k D_d\psi  u_k \in H^1_{sc}(\mathcal{L})$, then
$\mathds{1}_{z\geq 0} h_k D_d\psi u_k \in H^s_{sc}(\mathcal{L})$,
for any $s\in [0,\frac{1}{2}[$, where the semi-classical Sobolev space $H^s_{sc}(\mathcal{L})$ is defined by the set of $u\in \mathcal{S}'(\mathcal{L})$ such that $\Op^h(\langle \xi \rangle^s)u\in L^2(\mathcal{L})$ equipped with the norm $\| u\|_{H^s_{sc}(\mathcal{L})}=\|\Op^h(\langle \xi \rangle^s)u\|_{L^2(\mathcal{L})}$. Then, applying Lemma 4.3 from \cite{2A1}, for such $s$, we have
\begin{align*}
| I_2|\lesssim  R^{-s},
\end{align*}
uniformly in $k$ for all $R>1$. Thus, we conclude that
\begin{align}\label{lim22A1}
\underset{k,R\rightarrow +\infty }{\lim} I_2= 0.
\end{align}
Now, let us consider the first term $I_1$.
Observing that \begin{align*}
\mathds{1}_{z\geq 0} h_kD_d\psi u_k=h_kD_d(\psi\mathds{1}_{z\geq 0}u_k)+ih_k(\psi u_k)_{\mid z=0}\otimes\delta_{z=0},
\end{align*}
 we write
\begin{align*}
I_1=I_1^{'}+I_2^{'},
\end{align*}
where 
\begin{align*}
I_1^{'}=( \Op^h(\tilde{c}_1)\Op^h(\phi(\frac{\zeta}{R}))h_kD_d(\psi \mathds{1}_{z\geq 0}u_k),\mathds{1}_{z\geq 0} u_k )_{L^2(\mathbb{R}^{d+1})},
\end{align*}
and
\begin{align*}
I_2^{'}=ih_k( \Op^h(\tilde{c}_1)\Op^h(\phi(\frac{\zeta}{R}))((\psi u_k)_{\mid z=0}\otimes\delta_{z=0}),\mathds{1}_{z\geq 0} u_k)_{L^2(\mathbb{R}^{d+1})}.
\end{align*}
Note that \begin{align*}
\Op^h\big(\phi(\frac{\zeta}{R})\big)h_kD_d=\Op^h\big(\zeta\phi(\frac{\zeta}{R})\big),
\end{align*} and
\begin{align*}
\Op^h(\tilde{c}_1)\Op^h\big(\zeta\phi(\frac{\zeta}{R})\big)=\Op^h(\zeta\phi(\frac{\zeta}{R})\tilde{c}_1).
\end{align*}
Then, using the fact that $(u_k)_k$ is bounded in $L^2_{\mathrm{loc}}(\mathcal{L})$, uniformly with respect to $k$, we get 
\begin{align} \label{limgA1}
I_1^{'}&=( \Op^h\big(\zeta\phi(\frac{\zeta}{R})\tilde{c}_1\big)\psi\mathds{1}_{z\geq 0}u_k,\mathds{1}_{z\geq 0}u_k)_{L^2(\mathbb{R}^{d+1})}\nonumber
\\&\underset{k\rightarrow +\infty}{\longrightarrow}\langle \mu, \zeta\phi(\frac{\zeta}{R})\tilde{c}_1\rangle\underset{R\rightarrow +\infty}{\longrightarrow}\langle \mu,\zeta \tilde{c}_1\rangle.
\end{align}
For $ I_2^{'}$ one writes
\begin{align*}
\Op^h\big(\phi(\frac{\zeta}{R})\big)\delta_{z=0}=h_k^{-1}R \widehat{\phi}(\frac{Rz}{h_k}),
\end{align*}
 where  $\widehat{\phi}$ is the Fourier transform of $\phi$. Thus, since $\Op^h(\tilde{c}_1)$ is bounded on $L^2(\mathbb{R}^{d+1})$ we get
\begin{align}
| I_2^{'}| &\leq | ( \Op^h(\tilde{c}_1)\big((\psi u_k)_{\mid z=0} R \widehat{\phi}(\frac{Rz}{h})\big),\mathds{1}_{z\geq 0} u_k)_{L^2(\mathbb{R}^{d+1})} |\nonumber
\\&\leq R \|(\psi u_k)_{\mid z=0}\|_{L^2(\mathbb{R}^d)}\|\widehat{\phi}(\frac{Rz}{h_k})\|_{L^2(\mathbb{R})}\|u_k\|_{L^2(z\geq 0)}
\nonumber\\&  \label{5?54A1}\lesssim R^{\frac{1}{2}} h_k^{\frac{1}{2}} \|(\psi u_k)_{\mid z=0}\|_{L^2(\mathbb{R}^d)}\\&\label{5?55A1}\lesssim  R^{\frac{1}{2}} h_k^{\frac{1}{2}},
\end{align}
where in \eqref{5?54A1} we used that $(u_k)_k$ is bounded in $L^2_{\mathrm{loc}}(\mathcal{L})$, uniformly with respect to $k$ and the fact that $\phi \in C^\infty_c(\mathbb{R})$ and in \eqref{5?55A1} we used that $({u_k}_{\mid z=0})_k$ is bounded in $L^2_{\mathrm{loc}}(\partial \mathcal{L})$.\\
This yields
\begin{align}\label{LIMé"A1}
 I_2^{'} \underset{k\rightarrow +\infty}{\longrightarrow} 0.
 \end{align}
With \eqref{lim22A1}, \eqref{limgA1} and \eqref{LIMé"A1} we conclude the proof of the first part of Lemma \ref{lem11A1}.
For the second part, we first  write
\begin{align*}
( \Op^h(\tilde{c}_2) h_k^2 D_d^2 \psi u_k,u_k )_{L^2(z\geq 0)}&=  ( \Op^h(\tilde{c}_2)\psi h_k^2 D_d^2 u_k,u_k )_{L^2(z\geq 0)}\\&\quad+( \Op^h(\tilde{c}_2) h_k^2 [D_d^2,\psi] u_k,u_k )_{L^2(z\geq 0)}
 \\&=( \Op^h(\tilde{c}_2)\psi h_k^2 D_d^2 u_k,u_k )_{L^2(z\geq 0)}+ O(h_k),
\end{align*}
where in the last line we use $[D_d^2,\psi]=D_d^2\psi+2(D_d\psi)D_d$ and the sequences $(u_k)_k$ and $(h_kD_du_k)_k$ are bounded in $L^2_{\mathrm{loc}}(\mathcal{L})$ by Lemma \ref{lem0A1}.\\
Recall that using the equation $h_k^2Pu_k=0$,
one has
\begin{align*}
h_k^2 &g^{d,d}(x) D_d^2 u_k= h_k^2 D_t^2u_k-\sum\limits_{1\leq j\leq d-1}h_k^2 g^{d,j}(x)D_dD_ju_k- \sum\limits_{1\leq i,j\leq d-1}h_k^2 g^{i,j}(x)D_iD_ju_k \\& -\sum\limits_{1\leq i\leq d-1}h_k^2 g^{i,d}(x)D_iD_du_k
-h_k^2(\text{det}\; g)^{-\frac{1}{2}}\sum \limits_{1\leq i,j\leq d}D_i\Big((\text{det}\; g)^{\frac{1}{2}}g^{i,j}(x)\Big)D_ju_k
+h_k^2u_k.
\end{align*}
Then, given that $g^{d,d}(x)\neq 0$ and $(u_k)_k$ is bounded in $L^2_{\mathrm{loc}}(\mathcal{L})$, one writes
\begin{align}\label{A1A1}
( &\Op^h(\tilde{c}_2)\psi h_k^2 D_d^2 u_k,u_k )_{L^2(z\geq 0)}= J_1+J_2+J_3+J_4+J_5 +O(h_k^2),
\end{align}
where
\begin{align*}
&J_1= (\Op^h(\tilde{c}_2)\psi\frac{1}{g^{d,d}(x)} h_k^2 D_t^2 u_k,u_k)_{L^2(z\geq 0)}  
\\&J_2=- \sum_{1\leq j\leq d-1} ( \Op^h(\tilde{c}_2)\psi \frac{1}{g^{d,d}(x)} h_k^2 g^{d,j}(x)D_dD_ju_k,u_k)_{L^2(z\geq 0)}
\\&J_3=- \sum_{1\leq i\leq d-1} ( \Op^h(\tilde{c}_2)\psi \frac{1}{g^{d,d}(x)} h_k^2 g^{i,d}(x)D_iD_du_k,u_k)_{L^2(z\geq 0)}
\\&J_4=- \sum_{1\leq i,j\leq d-1} (\Op^h(\tilde{c}_2)\psi \frac{1}{g^{d,d}(x)} h_k^2 g^{i,j}(x)D_iD_ju_k,u_k)_{L^2(z\geq 0)}
\\&J_5=- \sum_{1\leq i,j\leq d} (\Op^h(\tilde{c}_2)\psi \frac{(\text{det}\; g)^{-\frac{1}{2}}}{g^{d,d}(x)} D_i\big((\text{det}\; g)^{\frac{1}{2}}g^{i,j}(x)\big)h_k^2 D_j u_k, u_k)_{L^2(z\geq 0)}.
\end{align*}
By Lemma \ref{lem0A1} more precisely the fact that $(h_kD_ju_k)_k$ is bounded in $L^2_{\mathrm{loc}}(\mathcal{L})$, we get
\begin{align}
J_5=O(h_k).\label{J5A1}
\end{align}
For $J_1$, consider $\tilde{\psi}\in C^\infty_c(\phi_\mathcal{L}(O))$ equal
to $1$ in a neighborhood of  $\mathrm{supp}(\psi)$. Due to Lemma \ref{lem0A1}, one has 
\begin{align}
J_1&
=([\Op^h(\tilde{c}_2),\frac{\psi}{g^{d,d}(x)}]\tilde{\psi}h_k^2 D_t^2u_k,u_k)_{L^2(z\geq 0)}
+( \frac{\psi}{g^{d,d}(x)}\Op^h(\tilde{c}_2)\tilde{\psi}h_k^2D_t^2 u_k,u_k)_{L^2(z\geq 0)}\nonumber
\\&
=(\frac{\psi}{g^{d,d}(x)}\Op^h(\tilde{c}_2\tau^2)\tilde{\psi} u_k,u_k)_{L^2(z\geq 0)}
+( \frac{\psi}{g^{d,d}(x)}\Op^h(\tilde{c}_2)[\tilde{\psi},D_t^2 ]h_k^2u_k,u_k)_{L^2(z\geq 0)}+O(h_k)\label{2.163333TH}
\\&
=(\frac{\psi}{g^{d,d}(x)}\Op^h(\tilde{c}_2\tau^2)\tilde{\psi} u_k,u_k)_{L^2(z\geq 0)}+O(h_k),\label{A2A1}
\end{align}
where in \eqref{2.163333TH} we use Proposition 5.10 of \cite{1A1} with $\tilde{c}_2 \in \Sigma^{0,n}_{T,0}(\langle\eta'\rangle^{-N};\mathbb{R}^{d+1}\times \mathbb{R}^{d})$, for any $n \in \mathbb{N} \cup \{+\infty\}$, where $n \geq d$ and $N \in \mathbb{N}$ and in the last line we use Lemma \ref{lem0A1} along with $[\tilde{\psi},D_t^2 ]=-D_t^2\tilde{\psi}-2(D_t\tilde{\psi})D_t$.
To treat the second term, we write
\begin{align}
J_2&
=-\sum_{1\leq j\leq d-1} ( [\Op^h(\tilde{c}_2), \psi \frac{g^{d,j}(x)}{g^{d,d}(x)} ] \tilde{\psi} h_k^2D_dD_ju_k,u_k)_{L^2(z\geq 0)}\nonumber
\\&\quad
-\sum_{1\leq j\leq d-1} (\psi\frac{g^{d,j}(x)}{g^{d,d}(x)} \Op^h(\tilde{c}_2)h_k^2 \tilde{\psi}D_dD_ju_k,u_k)_{L^2(z\geq 0)}\nonumber
\\&
=-\sum_{1\leq j\leq d-1}  (\psi \frac{g^{d,j}(x)}{g^{d,d}(x)} \Op^h(\tilde{c}_2\zeta\xi_j) \tilde{\psi}u_k,u_k)_{L^2(z\geq 0)}\nonumber\\&\quad-\sum_{1\leq j\leq d-1}  (\psi \frac{g^{d,j}(x)}{g^{d,d}(x)} \Op^h(\tilde{c}_2) [\tilde{\psi},D_dD_j]h_k^2u_k,u_k)_{L^2(z\geq 0)}+O(h_k)\label{J_21A1}
\\&
= -\sum_{1\leq j\leq d-1}  (\psi \frac{g^{d,j}(x)}{g^{d,d}(x)} \Op^h(\tilde{c}_2\zeta\xi_j) \tilde{\psi}u_k,u_k)_{L^2(z\geq 0)}+O(h_k)\label{J_22A1},
\end{align}
where the term $O(h_k)$ in \eqref{J_21A1} arises from Proposition 5.10 of \cite{1A1} and in \eqref{J_22A1} from Lemma \ref{lem0A1} along with $[\tilde{\psi},D_dD_j]=-(D_dD_j \tilde{\psi})-(D_j\tilde{\psi})D_d- (D_d\tilde{\psi})D_j$.
Proceeding similarly as for $J_2$, we find that
\begin{align}\label{A3A1}
J_3=-\sum_{1\leq i\leq d-1}  (\psi \frac{g^{i,d}(x)}{g^{d,d}(x)} \Op^h(\tilde{c}_2\zeta\xi_i) \tilde{\psi}u_k,u_k)_{L^2(z\geq 0)}+O(h_k),
\end{align}
and
\begin{align}\label{A4A1}
J_4=-\sum_{1\leq i,j\leq d-1}  (\psi \frac{g^{i,j}(x)}{g^{d,d}(x)} \Op^h(\tilde{c}_2\xi_j\xi_i) \tilde{\psi}u_k,u_k)_{L^2(z\geq 0)}+O(h_k).
\end{align}
Inserting  \eqref{J5A1}, \eqref{A2A1}, \eqref{J_22A1}, \eqref{A3A1} and \eqref{A4A1} in \eqref{A1A1} and using respectively Lemma 5.24, formula (5.28) in \cite{1A1}, $\psi=1$ in a
neighborhood of the $y$-projection of $\mathrm{supp}(\tilde{c}_2)$ and the fact that $\mathrm{supp}(\mu)\subset \mathrm{Char}(p)$ by Proposition\ref{hhA1}, we get
\begin{align*}
 \underset{k\rightarrow +\infty}{\mathrm{lim}}( \Op^h(\tilde{c}_2)\psi h_k^2 D_d^2 u_k,u_k )_{L^2(z\geq 0)}&=\langle\mu,\tilde{c}_2(\frac{\tau^2}{g^{d,d}(x)}- \underset{(i,j)\neq (d,d)}{\underset{1\leq i,j\leq d}{\sum}}\frac{g^{i,j}(x)}{g^{d,d}(x)} \xi_i\xi_j)\rangle\\&=\langle\mu,\tilde{c}_2(\zeta^2-\frac{p}{g^{d,d}(x)})\rangle\\&=\langle\mu,\tilde{c}_2\zeta^2\rangle.
 \end{align*}
 This completes the proof of the second item in Lemma  \ref{lem11A1}.
\end{proof}
From \eqref{operlimiteTh}, \eqref{5.50sssA1}, \eqref{c_0 tildeA1}, Lemma \ref{lem11A1},  along with \eqref{5.40RRA1}, we conclude that
\begin{align}\label{5.66finA1}
i([ h_kP, \Op^h(c)\psi]u_k,u_k)_{L^2(z\geq 0)}  \underset{k \rightarrow +\infty}{\longrightarrow}  \langle \mu,H_pc\rangle= \langle {}^t H_p\mu,c\rangle.
\end{align}
An integration by parts leads to
\begin{align}
\label{3**A1}
i([h_k P,\Op^h(c)\psi]u_k,u_k)_{L^2(z\geq 0)} & =-i h_k({\Op^h(c)\psi u_k}_{\mid z=0},
\partial_zu_{k_{\mid z=0}})_{L^2(\partial \mathcal{L})}
\\& \quad
+i h_k ({\partial_z \Op^h(c)\psi u_k}_{\mid z=0},{u_k}_{\mid z=0})_{L^2(\partial \mathcal{L})}.\nonumber
 \end{align}
Writing further
\begin{align*}
 \partial_z \Op^h(c)\psi &= \Op^h(c)\psi \partial_z +\Op^h(\partial_z c)\psi +\Op^h(c)(\partial_z\psi)
 \\&=\Op^h(c)\psi \partial_z+Op^h\big(\partial_zc_0+(\partial_z c_1) \zeta\big)\psi +\Op^h(c)(\partial_z\psi),
 \end{align*}
 implies
 \begin{align} \label{5.68sixbA1}
 i h_k ({\partial_z \Op^h(c)\psi u_k}_{\mid z=0},{u_k}_{\mid z=0})_{L^2(\partial \mathcal{L})}&=i  ({\Op^h(c)\psi h_k \partial_z u_k }_{\mid z=0}, {u_k}_{\mid z=0})_{L^2(\partial \mathcal{L})}\nonumber
 \\&
 \quad+i h_k({\Op^h(\partial_zc_0)\psi u_k}_{\mid z=0},{u_k}_{\mid z=0})_{L^2(\partial \mathcal{L})}\nonumber
 \\&\quad+ h_k({\Op^h(\partial_zc_1)h_k\partial_z\psi u_k}_{\mid z=0},{u_k}_{\mid z=0})_{L^2(\partial \mathcal{L})}\nonumber
 \\&
 \quad +i h_k ({\Op^h(c)(\partial_z \psi) u_k}_{\mid z=0},{u_k}_{\mid z=0})_{L^2(\partial \mathcal{L})}\nonumber
\\&=i h_k ({\Op^h(c)\psi \partial_z u_k}_{\mid z=0},{u_k}_{\mid z=0})_{L^2(\partial \mathcal{L})}+O(h_k),
 \end{align}
where the term $O(h_k)$ arises from the fact that the sequences $({u_k}_{\mid z=0})_k$ and $(h_k{\partial_zu_k}_{\mid z=0})_k$ are bounded in $L^2_{\mathrm{loc}}(\partial \mathcal{L})$.\\
Hence, \eqref{3**A1} and \eqref{5.68sixbA1} give
\begin{align}\label{comfi1A1}
i([h_k P,\Op^h(c)\psi]u_k,u_k)_{L^2(z\geq 0)} & =-i h_k({\Op^h(c)\psi u_k}_{\mid z=0},
\partial_zu_{k_{\mid z=0}})_{L^2(\partial \mathcal{L})}\nonumber
\\& \quad
+i  ({\Op^h(c)\psi h_k\partial_z u_k}_{\mid z=0},{u_k}_{\mid z=0})_{L^2(\partial \mathcal{L})}+O(h_k).
\end{align}
Writing
\begin{align*}
\Op^h(c)\psi &=\Op^h(c_0)\psi +\Op^h(c_1)h_kD_z\psi
\\&=\Op^h(c_0)\psi+\Op^h(c_1)\psi h_k D_z-i h_k \Op^h(c_1)(\partial_z\psi),
\end{align*}
implies 
 \begin{align} 
 ({\Op^h(c)\psi u_k} _{\mid z=0}, h_k \partial_z&u_{k_{\mid z=0}})_{L^2(\partial \mathcal{L})}=({\Op^h(c_0)\psi u_k}_{\mid z=0}, h_k \partial_z u_{k_{\mid z=0}})_{L^2(\partial \mathcal{L})}\nonumber\\&-i( {\Op^h(c_1)\psi h_k \partial_z u_k}_{\mid z=0}, h_k \partial_z u_{k_{\mid z=0}})_{L^2(\partial \mathcal{L})}+O(h_k),\label{confi20A1}
 \end{align}
and
\begin{align}
(\Op^h(c)\psi h_k \partial_z u_{k_{\mid z=0}},{u_k}_{\mid z=0})_{L^2(\partial \mathcal{L})}
&=({\Op^h(c_0)\psi h_k \partial_z u_k}_{\mid z=0},{u_k}_{\mid z=0})_{L^2(\partial \mathcal{L})}
\nonumber\\&\quad-i({\Op^h(c_1)\psi h_k^2 \partial_z^2 u_k}_{\mid z=0}, {u_k}_{\mid z=0})_{L^2(\partial \mathcal{L})} 
+O(h_k)
\nonumber\\& =L_1+L_2+L_3 +O(h_k),\label{confi21A1}
\end{align}
where
\begin{align*}
&L_1=({\Op^h(c_0)\psi h_k \partial_z u_k}_{\mid z=0},{u_k}_{\mid z=0})_{L^2(\partial \mathcal{L})}\\&L_2= 
-i ({\Op^h(c_1)\psi h_k^2 \partial_t^2 u_k}_{\mid z=0}, {u_k}_{\mid z=0})_{L^2(\partial \mathcal{L})}-i({\Op^h(c_1)\psi h_k^2  u_k}_{\mid z=0}, {u_k}_{\mid z=0})_{L^2(\partial \mathcal{L})}
\\&\quad\quad+i \underset{1\leq i,j \leq d-1}{\sum} ({\Op^h(c_1)\psi h_k^2(\text{det}\; g)^{-\frac{1}{2}}\partial_{x_i}\big((\text{det}\; g)^{\frac{1}{2}} g^{i,j}(x)  \partial_{x_j} u_k}\big)_{\mid z=0}, {u_k}_{\mid z=0})_{L^2(\partial \mathcal{L})}
\\&\quad \quad+i\underset{1\leq j \leq d}{\sum}\big({\Op^h(c_1)\psi h_k^2 (\text{det}\; g)^{-\frac{1}{2}}\partial_{z}\big((\text{det}\; g)^{\frac{1}{2}}g^{d,j}(x) \big) {\partial_{x_j} u_k}}_{\mid z=0}, {u_k}_{\mid z=0}\big)_{L^2(\partial \mathcal{L})}
\\&\quad:=i({\Op^h(c_1)\psi h_k^2R(t,x',D_t,D_x)u_k}_{\mid z=0},{u_k}_{\mid z=0})_{L^2(\mathcal{L})}
\\&L_3=i\underset{1\leq j \leq d-1}{\sum}\big({\Op^h(c_1)\psi h_k^2 g^{d,j}(x)  {\partial_{z}\partial_{x_j} u_k}}_{\mid z=0}, {u_k}_{\mid z=0}\big)_{L^2(\partial \mathcal{L})}
\\&\quad\quad
+i\underset{1\leq i \leq d-1}{\sum}\big({\Op^h(c_1)\psi h_k^2 (\text{det}\; g)^{-\frac{1}{2}}\partial_{x_i}\big((\text{det}\; g)^{\frac{1}{2}}g^{i,d}(x)  \partial_{z} u_k}\big)_{\mid z=0}, {u_k}_{\mid z=0}\big)_{L^2(\partial \mathcal{L})}.
\end{align*}
Note that the term $O(h_k)$ in \eqref{confi20A1} arises from the fact that the sequences  $({u_k}_{\mid z=0})_k$ and $(h_k{\partial_zu_k}_{\mid z=0})_k$  are bounded in $L^2_{\mathrm{loc}}(\partial \mathcal{L})$ and \eqref{confi21A1} follows from the same argument together with the equation $h_k^2Pu_k=0$.\\
From proposition \ref{pro9A1}, one has $g^{i,d}(x)=g^{d,j}(x)=0$ at $z=0$ for $i,j\neq d$, then 
\begin{align}\label{L_3A1}
L_3=0.
\end{align}
Arguing as in the proof of Lemma \ref{lem4.1££A1}, we obtain
\begin{align}\label{L_2A1}
L_2=i({\Op^h\big(c_1R(\tau,x',\xi')\big)\psi u_k}_{\mid z=0},{u_k}_{\mid z=0})_{L^2(\partial \mathcal{L})}+O(h_k),
\end{align}
where \begin{align*}
R(\tau,x',\xi')=\tau^2-\underset{1\leq i,j\leq d-1}{\sum}g^{i,j}(x',0)\xi_i\xi_j.
\end{align*}
Collecting \eqref{confi20A1}, \eqref{confi21A1}, \eqref{L_3A1} and \eqref{L_2A1} together with \eqref{comfi1A1}, yields
\begin{align}\label{5.74finnnA1}
i([h_k P,\Op^h(c)\psi]u_k,u_k)_{L^2(z\geq 0)} &=
-i({\Op^h(c_0)\psi u_k}_{\mid z=0}, h_k {\partial_z u_k}_{\mid z=0})_{L^2(\partial \mathcal{L})}\nonumber
\\&\quad
-( {\Op^h(c_1)\psi h_k \partial_z u_k}_{\mid z=0}, h_k {\partial_z u_k}_{\mid z=0})_{L^2(\partial \mathcal{L})}\nonumber
\\&\quad
+i({\Op^h(c_0)\psi h_k \partial_z u_k}_{\mid z=0},{u_k}_{\mid z=0})_{L^2(\partial \mathcal{L})}\nonumber
\\&\quad
-({\Op^h\big(c_1R(\tau,x',\xi'))\psi u_k}_{\mid z=0},{u_k}_{\mid z=0})_{L^2(\partial \mathcal{L})}
+O(h_k).
\end{align}
Given that the sequences $({u_k}_{\mid z=0})_k$ and $({h_k \partial_z u_k}_{\mid z=0})_k$ are bounded in $L^2_{\mathrm{loc}}(\partial \mathcal{L})$, Proposition \ref{Pro2µµA1},  ensures the existence of a Hermitian semi-classical measure on $T^*\partial \mathcal{L}$ associated to the vector $\tilde{U}_k= \,^t({u_k}_{\mid z=0},{h_k \partial_z u_k}_{\mid z=0})$ given by
\begin{align*}
\nu=
\begin{pmatrix}
    \nu^{0,0} & \nu^{0,1} \\
    \nu^{1,0} & \nu^{1,1} \\
\end{pmatrix}.
\end{align*}
Hence, using the fact that $\nu^{0,1}=\overline{ \nu^{1,0}}$, it follows from \eqref{5.66finA1} and \eqref{5.74finnnA1} that
\begin{align}
\label{3A1}
\langle ^t H_p\mu,c\rangle &=-i \langle \nu^{0,1},c_0 \rangle - \langle \nu^{1,1}, c_1\rangle+i\langle \nu^{1,0} , c_0 \rangle -\langle \nu^{0,0}, c_1 R(\tau,x',\xi')\rangle  \nonumber
\\&=-2\langle \mathrm{Im} (\nu^{1,0}), c_0\rangle -\langle R(\tau,x',\xi')\nu^{0,0}+\nu^{1,1} ,c_1\rangle. 
\end{align}
\begin{lemma}\label{vsuppA1}
 One has 
 \begin{align}
 &\mathrm{supp}\big(R(\tau,x',\xi')\nu^{0,0}+\nu^{1,1}\big)\subset \big(\;^\parallel \mathcal{H_\partial} \cup \,^\parallel \mathcal{G}_\partial\big) \cap \{\alpha \leq \tau\leq \alpha^{-1}\}, \label{5.12999A1}
\\&\mathrm{supp}\big( \mathrm{Im} (\nu^{1,0})\big)\subset \big(\;^\parallel \mathcal{H_\partial} \cup \,^\parallel \mathcal{G}_\partial\big) \cap \{\alpha \leq \tau\leq \alpha^{-1}\}. \label{5.1300000A1}
  \end{align}
\end{lemma}
\begin{proof}[Proof] 
The inclusion in $ \{\alpha \leq \tau\leq \alpha^{-1}\}$ follows from \eqref{25*A1}. \\
Let $d(t,x,\tau,\xi')\in C^\infty_c(\mathbb{R}^{d+1}\times \mathbb{R}^d)$ be a smooth function compactly supported near $\{z=0\}$, and assume that $d_{|z=0}$ is supported in the elliptic region, that is, 
\begin{align} \label{suppgA1}
\mathrm{supp}(d_{|z=0})\subset(^\parallel \mathcal{G}_\partial\cup ^\parallel \mathcal{H}_\partial)^c.
\end{align}
 We choose the support of $d$ sufficiently small such that  
\begin{align*}
d(t,x,\tau,\xi')\neq 0 \Rightarrow p(x,\tau,\xi',\zeta)\neq 0, \quad \text{for all }  \zeta \in \mathbb{R}.
\end{align*}
Consider \begin{align} \label{eucdivisioA1}
    q(t,x,\tau,\xi)=-\frac{d(t,x,\tau,\xi')\zeta}{p(x,\tau,\xi',\zeta)}.
    \end{align}
Note that $q$ is $C^1(\mathbb{R}^d)$ in $x$-variable, smooth and compactly supported in $(t,\tau,\xi')$-variables and admits a polyhomogeneous development in the $\zeta$ variable (see formulas (8.5) and (8.6) of Proposition 8.2 in \cite{1A1}). The expression \eqref{eucdivisioA1} takes the form given by \eqref{decomA1} with 
$c=0$, $c_0 = 0$ and $c_1 = d$. Then,
using formula \eqref{3A1}, we obtain
\begin{align*}
0= \langle R(\tau,x',\xi')\nu^{0,0}+\nu^{1,1} ,d_{|z=0}\rangle.
\end{align*}
Hence, with \eqref{suppgA1} we conclude \eqref{5.12999A1}. For \eqref{5.1300000A1}, using the same function $d$, we define
\begin{align}
q(t,x,\tau,\xi)=-\frac{d(t,x,\tau,\xi')}{p(x,\tau,\xi',\zeta)}.
\end{align}
Applying formula \eqref{3A1} with $c=0$, $c_0 = d$ and $c_1 = 0$, we get
\begin{align*}
0= \langle \mathrm{Im} (\nu^{1,0}) ,d_{|z=0}\rangle.
\end{align*}
Thus, by \eqref{suppgA1} we deduce \eqref{5.1300000A1}, which completes the proof of Lemma \ref{vsuppA1}.
\end{proof}
Let $\varrho \in \big(\,^\parallel \mathcal{H_\partial} \cup \,^\parallel \mathcal{G}_\partial\big)  \cap \{\alpha \leq \tau\leq \alpha^{-1}\}$.
One set $\varrho^\pm=(\varrho,\zeta^\pm)=(t,x',z=0,\tau,\xi',\zeta^\pm)$
where $\zeta^\pm$ is defined as in \eqref{6£££A1}.
Recall that for $c(t,x,\tau,\xi) \in \Sigma_0^\mathcal{H} (\mathbb{R}^{2d+2}) $, we write
\begin{align*}
\chi(\tau)c(t,x,\tau,\xi)=c_0(t,x,\tau,\xi')+c_1(t,x,\tau,\xi')\zeta+q(t,x,\tau,\xi)p(t,x,\tau,\xi),
\end{align*}
where $c_0$ and $c_1$ are tangential symbols and  $\chi$ is equal to $1$ in a neighborhood of $[\alpha,\alpha^{-1}]$.\\
Since $\varrho^\pm\in \mathrm{Char}(p)$ and $\zeta^+=-\zeta^-$, we have
\begin{align*}
c_0(t,x',z=0,\tau,\xi')&=\frac{1}{2}\big( c(\varrho^{+})+c(\varrho^{-})\big)
\\&=\frac{1}{2}\langle\delta_{\varrho^{+}}+\delta_{\varrho^{-}},c\rangle,
\end{align*}
and
\begin{align*}
c_1(t,x',z=0,\tau,\xi')&=\frac{c(\varrho^{+})-c(\varrho^{-})}{\zeta^+-\zeta^-}
\\&=\frac{\langle\delta_{\varrho^{+}}-\delta_{\varrho^{-}},c\rangle}{\zeta^+-\zeta^-}.
\end{align*}
Plugged in (\ref{3A1}) and using Lemma \ref{vsuppA1}, we obtain
\begin{align*}
\langle& H_p\mu,c\rangle=-\langle ^t H_p\mu,c\rangle \nonumber
\\&=
\int_{^\parallel \mathcal{H}_\partial \cup ^\parallel \mathcal{G}_\partial} 	\langle \delta_{\varrho^{+}}+\delta_{\varrho^{-}},c\rangle d\big(\mathrm{Im}(\nu^{1,0}) \big)+\int_{\,^\parallel \mathcal{H}_\partial \cup \,^\parallel \mathcal{G}_\partial} \frac{\langle\delta_{\varrho^{+}}-\delta_{\varrho^{-}},c\rangle}{\zeta^+-\zeta^-} d\big(\nu^{1,1}+R(\tau,x',\xi')\nu^{0,0}\big),
\end{align*}
which concludes the proof of Theorem \ref{th1A1}.
\section{Proof of the main result} \label{conclusionA1}
After verifying above the assumptions of Theorem \ref{th1A1} for the sequence $(u_k)_k$, we derive the equation of the measures given in \eqref{13£££A1}.
To conclude the main result of this paper, we employ this equation along with the boundary condition given in \eqref{eqimA1}. 
To proceed, let $C =(O, \phi_\mathcal{L})$ be a local chart, where $O$ is a neighborhood of a point $\varrho^0 \in \partial \mathcal{L}$ and let $a(t,x,\tau,\xi)$ be a symbol in $ \Sigma_0(\phi_\mathcal{L}(O)\times \mathbb{R}^{2d})$. Consider $\psi\in C^\infty_c(\phi_\mathcal{L}(O))$ such that $\psi=1$ on the $(t,x)$-projection of $\mathrm{supp}(a)$. We keep the notations $a$, $\nu$ instead of their local representative $a^c$ and $\nu^c$ in the local chart $C$ as given in Remark \ref{remarK 4.11A1}. Using the boundary condition $h_k\partial_zu_{k_{\mid z=0}}=\Op^h(\tau)u_{k_{\mid z=0}}$, Lemma 5.24 and formula (5.28) in \cite{1A1}, we obtain
\begin{align*}
(\Op^h(a) \psi h_k \partial_z u_{k_{\mid z=0}},u_{k_{\mid z=0}})_{L^2(\partial \mathcal{L})}&=(\Op^h(a) \psi \Op^h(\tau)u_{k_{\mid z=0}}, u_{k_{\mid z=0}})_{L^2(\partial \mathcal{L})}
\\&
\underset{k\rightarrow +\infty}{\longrightarrow} \langle \nu^{0,0},\tau a \rangle.
\end{align*}
Applying the same arguments as above, we find
\begin{align*}
(\Op^h(a) \psi h_k \partial_z u_{k_{\mid z=0}},h_k \partial_z u_{k_{\mid z=0}})_{L^2(\partial \mathcal{L})}&=(\Op^h(\tau)^*\Op^h(a) \psi \Op^h(\tau)u_{k_{\mid z=0}}, u_{k_{\mid z=0}})_{L^2(\partial \mathcal{L})}
\\&
\underset{k\rightarrow +\infty}{\longrightarrow} \langle \nu^{0,0},\tau^2 a \rangle.
\end{align*}
Thus, we conclude the following relations
\begin{align*}
\nu^{1,1}=\tau^2\nu^{0,0}  \quad \quad \text{and} \quad \quad \nu^{1,0}=\tau\nu^{0,0}. 
\end{align*}
Therefore, since $\nu^{0,0}$ is real then $\mathrm{Im} (\nu^{1,0})=0$ and we rewrite \eqref{13£££A1} as follows
\begin{align}\label{tranA1}
H_p\mu=\int_{\,^\parallel \mathcal{H}_\partial \cup \,^\parallel \mathcal{G}_\partial} \frac{\delta_{\varrho^{+}}-\delta_{\varrho^{-}}}{\zeta^+-\zeta^-} d\Big(\big(\tau^2+ R(\tau,x',\xi')\big)\nu^{0,0}\Big).
\end{align} 
Using \eqref{tranA1} and noting that the measure $\big(\tau^2+ R(\tau,x',\xi')\big)\nu^{0,0}$ is positive on $\,^\parallel \mathcal{H}_\partial \cup \,^\parallel \mathcal{G}_\partial$, one obtains the following theorem proven in \cite{11A1}.
\begin{theorem}\label{generabichaA1} If a nonnegative measure $\mu$ satisfies a transport equation of the form given in \eqref{tranA1}, then $\mathrm{supp}(\mu)$ is a union of maximal generalized bicharacteristics.
\end{theorem}
Let $\varrho^0 \in \mathrm{supp}(\mu)$ (such a point exists by the the first item of Proposition \ref{hhA1}). Since $\mathrm{supp}(\mu) \subset \mathrm{Char}(p)\cap T^*\mathcal{L}$ (again by Proposition \ref{hhA1}), Theorem \ref{theximaxiA1}  ensures the existence of a maximal generalized bicharacteristic $^G\gamma$ passing through $\varrho^0$. With the interior geometric control condition fulfilled by $(\omega, T-2\delta)$, the bicharacteristic $^G\gamma$ reaches a point above $\omega \times ]\delta,T-\delta[$, where the measure $\mu$ vanishes by the last item of Proposition \ref{hhA1}. Thus, in view of Theorem \ref{generabichaA1}, this leads to a contradiction and concludes the proof of our main result.
\section*{\large \textbf{Acknowledgements.}}  
The auther thanks her PhD advisor Jérôme Le Rousseau for his constant support and guidance.

\begin{appendix}
\section{Some technical results}\label{appenAA1}
\renewcommand{\theequation}{\thesection.\arabic{equation}}
In this appendix, we prove some technical results used in sections \ref{intro2.1tH} and \ref{uuuA1}.

\begin{proof}[Proof of Lemma \ref{selfA1}]
Let $U=(u^0,u^1)\in D(A)$ and $V=(v^0,v^1)\in D(A)$. An integration by parts yields
\begin{align*}
(AU,V)_{\mathcal{H}}&=(-iu^1,v^0)_{H^1(\mathcal{M})}+(i(1-\Delta_g)u^0,v^1)_{L^2(\mathcal{M})}
\\&=-i(u^1,v^0)_{L^2(\mathcal{M})}-i(\nabla_gu^1,\nabla_gv^0)_{L^2V(\mathcal{M})}+i(u^0,v^1)_{L^2(\mathcal{M})}-i(\Delta_gu^0,v^1)_{L^2(\mathcal{M})}
\\&=-i(u^1,v^0)_{L^2(\mathcal{M})}+i(u^0,v^1)_{L^2(\mathcal{M})}-i(u^1_{|\partial\mathcal{M}},\partial_nv^0_{|\partial\mathcal{M}})_{L^2(\partial\mathcal{M})}+i(u^1,\Delta_gv^0)_{L^2(\mathcal{M})}
\\&\quad
-i(\partial_nu^0_{|\partial\mathcal{M}},v^1_{|\partial\mathcal{M}})_{L^2(\partial\mathcal{M})}+i(\nabla_gu^0,\nabla_gv^1)_{L^2V(\mathcal{M})}.
\end{align*}
Using the boundary conditions $\partial_nv^0_{|\partial\mathcal{M}}=-iv^1_{|\partial\mathcal{M}}$ and $\partial_nu^0_{|\partial\mathcal{M}}=-iu^1_{|\partial\mathcal{M}}$, we get
\begin{align*}
(AU,V)_{\mathcal{H}}&=-i(u^1,v^0)_{L^2(\mathcal{M})}+i(u^0,v^1)_{L^2(\mathcal{M})}+(u^1_{|\partial\mathcal{M}},v^1_{|\partial\mathcal{M}})_{L^2(\partial\mathcal{M})}+i(u^1,\Delta_gv^0)_{L^2(\mathcal{M})}
\\&\quad
-(u^1_{|\partial\mathcal{M}},v^1_{|\partial\mathcal{M}})_{L^2(\partial\mathcal{M})}+i(\nabla_gu^0,\nabla_gv^1)_{L^2V(\mathcal{M})}
\\&=(u^0,-iv^1)_{L^2(\mathcal{M})}+i(\nabla_gu^0,\nabla_gv^1)_{L^2V(\mathcal{M})}+(u^1,iv^0)_{L^2(\mathcal{M})}+(u^1,-i\Delta_gv^0)_{L^2(\mathcal{M})}
\\&=(U,AV)_{\mathcal{H}},
\end{align*}
 implying that $(A,D(A))$ is symmetric. \\
Next, we show that the operator $\mathrm{Id} +iA:D(A)\rightarrow \mathcal{H}$ is surjective. Consider $(f^0,f^1)\in \mathcal{H}$. We seek $(u^0,u^1)\in D(A)$ such that
 \begin{align*} (\mathrm{Id} +iA)
\begin{pmatrix}
u^0 \\
 u^1
\end{pmatrix}=\begin{pmatrix}
f^0 \\
 f^1
\end{pmatrix},
\end{align*}
which reads
\begin{equation} \label{laxA1}
    \begin{cases}
    \quad \quad  u^1&=f^0-u^0 \\
     -\Delta_gu^0+2u^0&=f^0-f^1 \\
     (\partial_nu^0+iu^1)_{|\partial\mathcal{M}}&= 0.
    \end{cases}
\end{equation}
We introduce the following sesquilinear form on $H^1(\mathcal{M})$ 
\begin{align*}
a(u^0,v^0)=( \nabla_gu^0,\nabla_gv^0)_{L^2V(\mathcal{M})}+2(u^0,v^0)_{L^2(\mathcal{M})}-i(u^0_{|\partial\mathcal{M}},v^0_{|\partial\mathcal{M}})_{L^2(\partial\mathcal{M})}.
\end{align*}
By Cauchy Schwarz and trace formula, this sesquilinear form is continuous. On the other hand, $a(.,.)$ is coercive since 
\begin{align*}
|a(u^0,u^0)|\geq\|\nabla_gu^0\|^2_{L^2V(\mathcal{M})}+2\|u^0\|^2_{L^2(\mathcal{M})}\gtrsim \|u^0\|^2_{H^1(\mathcal{M})}.
\end{align*} 
Consider the continuous antilinear form on $H^1(\mathcal{M})$ $$L(v^0)=(f^0-f^1,v^0)_{L^2(\mathcal{M})}-i(f^0_{|\partial\mathcal{M}},v^0_{|\partial\mathcal{M}})_{L^2(\partial\mathcal{M})}.$$
Therefore, by Lax-Milgram’s theorem (see Theorem 1.1 in \cite{2333A1}) there exists a unique $u^0\in H^1(\mathcal{M})$ such that 
\begin{align*}
a(u^0,v^0)=L(v^0), \quad\quad \quad \quad\forall\, v^0\in H^1(\mathcal{M}).
\end{align*}
By Proposition \ref{pro*A1}, one has $u^0\in H^2(\mathcal{M})$ satisfying
\begin{equation*}
    \begin{cases}
     -\Delta_gu^0+2u^0&=f^0-f^1 \\
     \quad\quad\partial_nu^0_{|\partial\mathcal{M}}&=i(u^0- f^0)_{|\partial\mathcal{M}}\in H^{\frac{1}{2}}(\partial\mathcal{M}).
    \end{cases}
\end{equation*}
Hence, $u^1=f^0-u^0 \in H^1(\mathcal{M})$ and \eqref{laxA1} is verified, which implies that $\mathrm{Id} +iA:D(A)\rightarrow \mathcal{H}$ is surjective. In the same way, we show that the operator $A+i \;\mathrm{Id}:D(A)\rightarrow \mathcal{H}$ is surjective by considering the coercive sesquilinear form on $H^1(\mathcal{M})$ defined by
\begin{align*}
a(u^0,v^0)=( \nabla_gu^0,\nabla_gv^0)_{L^2V(\mathcal{M})}+2(u^0,v^0)_{L^2(\mathcal{M})}+i(u^0_{|\partial\mathcal{M}},v^0_{|\partial\mathcal{M}})_{L^2(\partial\mathcal{M})},
\end{align*}
and the continuous antilinear form on $H^1(\mathcal{M})$
 \begin{align*}
 L(v^0)=-i(f^1+f^0,v^0)_{L^2(\mathcal{M})}+(f^0_{|\partial\mathcal{M}},v^0_{|\partial\mathcal{M}})_{L^2(\partial\mathcal{M})}.
 \end{align*}
By Theorem 8.3 in \cite{66A1}, it follows that $(A,D(A))$ is selfadjoint.
Consequently, the operator $\mathrm{Id} +iA:D(A)\rightarrow \mathcal{H}$ is bijective and so the operator $(A-i \;\mathrm{Id})^{-1}:\mathcal{H}\rightarrow D(A)$ is well defined. Furthermore, for $U=(u^0,u^1)\in D(A)$ and $\lambda>0$, using the fact that $A$ is symmetric, one writes
\begin{align}\label{symmmA1}
\|(\lambda \mathrm{Id}+iA)U\|_{\mathcal{H}}^2&=\lambda^2\|U\|_{\mathcal{H}}^2+\|AU\|_{\mathcal{H}}^2-2\lambda\; \mathrm{Im}(AU,U)_{\mathcal{H}}\nonumber
\\&=\lambda^2\|U\|_{\mathcal{H}}^2+\|AU\|_{\mathcal{H}}^2\geq \lambda^2\|U\|_{\mathcal{H}}^2.
\end{align}
This yields,
\begin{align}\label{sy122A1}
\|( A-i \;\mathrm{Id})^{-1}U\|_{\mathcal{H}}^2=\|( \mathrm{Id}+iA)^{-1}U\|_{\mathcal{H}}^2\leq \|U\|_{\mathcal{H}}^2, \quad \text{for all} \quad U\in \mathcal{H}.
\end{align}
As $i$ is in the resolvent
set $\rho(A)$ of $A$, $(A-i \;\mathrm{Id})^{-1}:\mathcal{H}\rightarrow D(A)$ is bounded by \eqref{sy122A1}
and the injection $\iota:D(A)\rightarrow \mathcal{H}$ is compact by the
Rellich–Kondrachov theorem [\cite{24A1}, Theorem 9.16], then  $\iota\circ (A-i \;\mathrm{Id})^{-1}$ is also compact. Therefore, $A$ has a compact resolvent on $\mathcal{H}$, which concludes the proof.
\end{proof}
\begin{proof}[Proof of Lemma \ref{semigroupeA1}]By Lemma  \ref{selfA1}, the operators $(\pm iA,D(A))$, with $D(A)$ dense in $\mathcal{H}$, are maximal monotone. Then, $(\pm iA, D(A))$ satisfy the assumptions of the Lumer–Phillips Theorem (see, for example, Theorem 1.4.3 in \cite{1888A1} and Theorem 12.23 in \cite{20A1}), which ensures that $(A, D(A))$ is the infinitesimal generator of a strongly continuous unitary group $S(t)$ on $\mathcal{H}$.
\end{proof}
The proofs of Lemmas \ref{lem1Th} and \ref{lem2Th} rely on an application of Schur’s Lemma, which is recalled below.
\begin{lemma}[Schur’s Lemma] \label{shurlem} Let the operator $\mathcal{K}: \mathcal{S}(\mathbb{R}^n)\rightarrow \mathcal{S}(\mathbb{R}^m)$ with Schwartz kernel $K(x,y)$ on $\mathbb{R}^m \times \mathbb{R}^n$ such that $K(.,y)$ and $K(x,.)$ are $L^1$-functions for almost all $x\in \mathbb{R}^m$ and $y\in \mathbb{R}^n$ respectively, with moreover
\begin{align*}
    \underset{x\in\mathbb{R}^m}{\mathrm{ess\;sup}}\int_{\mathbb{R}^n} |K(x,y)| dy \leq A \quad \text{and}\quad  \underset{y\in\mathbb{R}^n}{\mathrm{ess \; sup}}\int_{\mathbb{R}^m} |K(x,y)| dx \leq B,
\end{align*}
for some $A\geq 0$ and $B\geq 0$. Then, the operator $\mathcal{K}$ extends as a continuous operator from $L^2(\mathbb{R}^n)$ to $ L^2(\mathbb{R}^m)$ with 
\begin{align*}
    \|\mathcal{K}\|_{L^2(\mathbb{R}^n) \rightarrow L^2(\mathbb{R}^m)} \leq (AB)^\frac{1}{2}
\end{align*}
\end{lemma}
Lemma \ref{shurlem} follows the same argument as in the case $m=n$. 
\begin{proof}[Proof of Lemma \ref{lem1Th}] 
Set $a'(x',\xi)=a(x',0,\xi)$. We write
\begin{align} \label{operatorker}
    (\Op^h(a)u)_{|z=0}&=\Big(\frac{1}{(2\pi)^d}\int_{\mathbb{R}^{2d}} e^{i(x-y)\xi} a(x,h\xi)u(y) \; dy d\xi \Big)_{|z=0} \nonumber
    \\&=\frac{1}{(2\pi)^d}\int_{\mathbb{R}^{2d}} e^{i(x'-y')\xi'-iy_d\zeta} a'(x',h\xi)u(y)\; dy d\xi,
\end{align}
with associated kernel given by
\begin{align} \label{noy1}
    K'(x',y)=\frac{1}{(2\pi)^d} \int_{\mathbb{R}^d}e^{i(x'-y')\xi'-iy_d\zeta} a'(x',h\xi) \;d\xi
    &= \frac{1}{(2\pi h)^d}\int_{\mathbb{R}^d} e^{i\frac{(x'-y')}{h}\xi'-i\frac{y_d}{h}\zeta}a'(x',\xi)\;d\xi \nonumber
    \\&= h^{-d} k_{a'}(x',\frac{x'-y'}{h},\frac{-y_d}{h}),
\end{align}
with $$k_{a'}(x',\nu)=\frac{1}{(2\pi)^d} \int_{\mathbb{R}^d} e^{i \nu.\xi} a'(x',\xi)\;d\xi.$$ 
Let $L=(1-i\nu.\nabla_\xi)/\langle \nu \rangle ^2$. Noting that $L\, \mathrm{exp}(i\nu.\xi)=\mathrm{exp}(i\nu.\xi)$, one has
\begin{align*}
    k_{a'}(x',\nu)=\frac{1}{(2\pi)^d} \int_{\mathbb{R}^d} e^{i \nu.\xi} (^tL)^{d+1} a'(x',\xi) \;d\xi.
\end{align*}
Since $a$ is compactly supported in $\xi'$, smooth in $\xi$ with decay $\langle \zeta\rangle ^{-2}$, we deduce
\begin{align} \label{kprime1}
    |k_{a'}(x',\nu)|\lesssim \int_{\mathbb{R}^d} \frac{1}{\langle \nu \rangle^{d+1}} \frac{1}{\langle \zeta \rangle^2} \frac{1}{\langle \xi'\rangle^N} \;d\xi\lesssim \frac{1}{\langle \nu \rangle^{d+1}}, 
\end{align}
where $N> d-1$.
Using \eqref{noy1}, \eqref{kprime1} and the change of variables $\frac{x'-y'}{h}= \nu'$ and $\frac{-y_d}{h}=\nu_d$, we get
\begin{align} \label{2a167lem1}
    \underset{x'\in \mathbb{R}^{d-1}}{\mathrm{ess\; sup}} \int_{\mathbb{R}^d}|K'(x',y)|\;dy= \underset{x'\in \mathbb{R}^{d-1}}{\mathrm{ess\; sup}} \; h^{-d}\int_{\mathbb{R}^d} | k_{a'}(x',\frac{x'-y'}{h},\frac{-y_d}{h})| \;dy &=\underset{x'\in \mathbb{R}^{d-1}}{\mathrm{ess\; sup}} \int_{\mathbb{R}^d} |k_{a'}(x',\nu)|\;d\nu \nonumber
    \\&
    \lesssim \int_{\mathbb{R}^d} \frac{1}{\langle \nu \rangle^{d+1}} \;d\nu \lesssim 1.
\end{align}
Similarly,
\begin{align}\label{2a168lem1}
    \underset{y\in \mathbb{R}^{d}}{\mathrm{ess\; sup}} \int_{\mathbb{R}^{d-1}} |K'(x',y)|\;dx'& = \underset{y\in \mathbb{R}^{d}}{\mathrm{ess\; sup}} \; h^{-d}
    \int_{\mathbb{R}^{d-1}}  | k_{a'}(x',\frac{x'-y'}{h},\frac{-y_d}{h})|\; dx'\nonumber
    \\&
    =  \underset{y\in \mathbb{R}^{d}}{\mathrm{ess\; sup}} \; h^{-1} \int_{\mathbb{R}^{d-1}} |k_{a'}(y'+h\nu',\nu',\frac{-y_d}{h}) \;d\nu'  \nonumber
    \\&\lesssim h^{-1} \int_{\mathbb{R}^{d-1}} \frac{d\nu'}{\langle (\nu',\frac{y_d}{h})\rangle^{d+1}} 
    \lesssim h^{-1}\int_{\mathbb{R}^{d-1}} \frac{d\nu'}{\langle \nu'\rangle^{d+1}}  \lesssim h^{-1}.
\end{align}
The result then follows immediately from Lemma \ref{shurlem}.
\end{proof}
\begin{proof}[Proof of Lemma \ref{lem2Th}]
According to \eqref{operatorker}, the kernel of the operator $\mathcal{K}$ can be written as
    \begin{align*}
       L(x',y)&= K'(x',y)\Big(\theta(y)-\theta(x',0)\Big)
       \\&= h^{-d} k_{a'}(x',\frac{x'-y'}{h},\frac{-y_d}{h}) \Big( \theta(y)- \theta(x',0)\Big)
       \\&=h^{-d}  k_{a'}(x',\nu)\Big(\theta(x'-h\nu',-h\nu_d)- \theta(x',0)\Big),
    \end{align*}
where $K'(x',y)$ is defined in \eqref{noy1}, $a'(x',\xi)=a(x',0,\xi)$, $\frac{x'-y'}{h}= \nu'$, $\frac{-y_d}{h}=\nu_d$ and $$k_{a'}(x',\nu)=\frac{1}{(2\pi)^d} \int_{\mathbb{R}^d} e^{i \nu.\xi} a'(x',\xi)\;d\xi.$$
Since $\theta \in C^1(\mathbb{R}^d)$, a first-order Taylor expansion yields
\begin{align*}
\theta(x'-h\nu',-h\nu_d)- \theta(x',0)=-h \underset{j}{\sum} \nu_j \Theta_j(x',h\nu),
\end{align*}
with $$\Theta_j(x',h\nu)= \int_0^1\partial_{x_j}\theta\big((x',0)-s h\nu\big)\; ds.$$
Using the smoothness of $a$ in $\xi$ and the identity $-i\nu_j k_{a'}(x',\nu)=k_{(\partial_{\xi_j}a)'}(x',\nu)$, we obtain
\begin{align*}
L(x',y)=-h^{1-d}\;\underset{j}{\sum} \nu_j k_{a'}(x',\nu) \Theta_j(x',h\nu)=-i h^{1-d}\;\underset{j}{\sum} \nu_j k_{(\partial_{\xi_j}a)'}(x',\nu)\Theta_j(x',h\nu).
\end{align*}
Finally, using that $a$ is compactly supported in $\xi'$, smooth in $\xi$ with decay $\langle \zeta\rangle ^{-2}$ and arguing as in \eqref{2a167lem1} and \eqref{2a168lem1}, we deduce
\begin{align*}
     &\underset{x'\in \mathbb{R}^{d-1}}{\mathrm{ess\; sup}} \int_{\mathbb{R}^d} |L(x',y)| \; dy \lesssim h,
     \\& \underset{y \in \mathbb{R}^{d}}{\mathrm{ess\; sup}} \int_{\mathbb{R}^{d-1}} |L(x',y)| \; dx' \lesssim 1.
\end{align*}
The conclusion then follows from Lemma \ref{shurlem}.
\end{proof}
\begin{proof}[Proof of Lemma \ref{lem3Th}]
With $a'(x',\xi)=a(x',0,\xi)$, one writes
\begin{align*}
\mathcal{K}u =(\Op^h(a) (u \; \delta))_{|z=0}&
\frac{1}{(2\pi)^d} \int_{\mathbb{R}^{2d}} e^{i(x'-y')\xi'-iy_d\zeta} a'(x',h\xi',h\zeta)u(y')\;\delta(y_d) \;dy' dy_d d\xi' d\zeta
\\& =\frac{1}{(2\pi h)^{d}} \int_{\mathbb{R}^{2d-1}} e^{i\frac{(x'-y')}{h}\xi'} a'(x',\xi',\zeta)u(y') \;dy' d\xi' d\zeta
\\&=\frac{h^{-1}}{2\pi} \Op^h(t)u(x'), 
\end{align*}
where 
\begin{align*}
t(x',\xi')= \int_{\mathbb{R}} a'(x',\xi',\zeta) \; d\zeta,
\end{align*}
is a tangential symbol compactly supported in the $(x',\xi')$-variables, $C^0(\mathbb{R}^{d-1})$ in the $x'$-variable and $C^\infty(\mathbb{R}^{d-1})$ in $\xi'$. Consequently, 
\begin{align*}\|\mathcal{K}u\|_{L^2(\mathbb{R}^{d-1})} \lesssim h^{-1}\|u_{|z=0}\|_{L^2(\mathbb{R}^{d-1})}.
\end{align*}
This concludes the proof of Lemma \ref{lem3Th}.
\end{proof}
\end{appendix}

\end{document}